\documentclass[12pt]{amsart}
\usepackage{amsmath}
\usepackage{amsfonts}
\usepackage{amssymb}
\usepackage{mathtools}
\usepackage{enumitem}
\usepackage[alphabetic]{amsrefs}

\usepackage[usenames,dvipsnames]{xcolor}

\usepackage[english]{babel}
\usepackage[latin1]{inputenc}
\usepackage{fancyhdr}
\usepackage{indentfirst}
\usepackage{graphicx}
\usepackage{float}
\usepackage{caption}
\usepackage{newlfont} 
\usepackage{amssymb}
\usepackage{amsmath}
\usepackage{mathtools}
\usepackage{latexsym}
\usepackage{amsthm}
\usepackage{enumitem}
\usepackage{esint}
\theoremstyle{plain}
\newtheorem{thm}{Theorem}[section]
\newtheorem{exp-thm}[thm]{Expected Theorem}
\newtheorem{exp-lem}[thm]{Expected Lemma}
\newtheorem{exp-cor}[thm]{Expected Corollary}
\newtheorem{lem}[thm]{Lemma}
\newtheorem{prop}[thm]{Proposition}    
\newtheorem{cor}[thm]{Corollary}
\newtheorem{defin}[thm]{Definition}
\theoremstyle{remark}                 
\newtheorem{remark}[thm]{Remark}

\newcommand{\R}{\mathbb{R}}

\newcommand{\N}{\mathbb{N}}

\newcommand{\G}{\mathbb{G}}

\newcommand{\scal}[2]{\langle {#1} , {#2}\rangle}
\newcommand{\Hnabla}{\nabla_{{\G}}}
\newcommand{\divg}{\mathrm{div}_{{\G}}}
\newcommand{\bq}{\mathbf{q}}
\newcommand{\bxi}{\boldsymbol{\xi}}
\newcommand{\bzeta}{\boldsymbol{\zeta}}
\newcommand{\e}{\varepsilon}
\newcommand{\osc}{\mathrm{osc}}
\renewcommand{\leq}{\leqslant}
\renewcommand{\le}{\leqslant}
\renewcommand{\geq}{\geqslant}
\renewcommand{\ge}{\geqslant}
\allowdisplaybreaks

\makeatletter
\DeclareMathOperator*{\dist}{dist}

\makeatletter
\def\cleardoublepage{\clearpage\if@twoside \ifodd\c@page\else
	\hbox{}
	\thispagestyle{empty}
	\newpage
	\if@twocolumn\hbox{}\newpage\fi\fi\fi}
\makeatother

\numberwithin{equation}{section}

\begin{document}
	
	\title{Lipschitz regularity for almost minimizers of \\ a one-phase $p$-Bernoulli-type functional \\ in Carnot Groups of step two}
	
	\author{Fausto Ferrari}
	\address{Fausto Ferrari: Dipartimento di Matematica Universit\`a di Bologna Piazza di Porta S.Donato 5  40126, Bologna-Italy}
	\email{fausto.ferrari@unibo.it }

	\author{Enzo Maria Merlino}
	\address{Enzo Maria Merlino: Dipartimento di Matematica Universit\`a di Bologna Piazza di Porta S.Donato 5 40126, Bologna-Italy}
	\email{enzomaria.merlino2@unibo.it }

	\date{\today}
        \thanks{The authors are members of Gruppo Nazionale per l'Analisi Matematica, la Probabilit\'a e le loro Applicazioni (GNAMPA) of the Istituto Nazionale di Alta Matematica (INdAM): F.F. has been partially supported by the INDAM-GNAMPA research project 2023 "Equazioni completamente non lineari locali e non local" and E.M.M. is partially supported by the GNAMPA research project 2023 "Equazioni nonlocali di tipo misto e geometrico". The authors are also supported by INDAM-GNAMPA project 2024 "Free boundary problems in noncommutative structures and degenerate operators" and PRIN 2022 7HX33Z - CUP J53D23003610006, "Pattern formation in nonlinear phenomena".}
	
	\begin{abstract} 
            In this paper, in a Carnot group $\G$ of step $2$ and homogeneous dimension $Q$, we prove that almost minimizers of the (horizontal) one-phase $p$-Bernoulli-type functional 
            $$ J_p(u,\Omega):=\int_{\Omega}\Big( |\nabla_{\mathbb{G}} u(x)|^p+\chi_{\{u>0\}}(x)\Big)\,dx$$
            whenever $p>p^\#:=\frac{2Q}{Q+2}$, are locally Lipschitz continuous with respect Carnot-Carath\'eodory distance on $\G$. This implies an H\"older continuous regularity from an  Euclidean point of view.        
        \end{abstract}
\maketitle	\vspace{-2.4em}
\tableofcontents
\section{Introduction}
In this paper, we study the regularity of the almost minimizers of a one-phase (horizontal) $p$-Bernoulli-type functional in Carnot groups of step two. We are interested in the regularity properties of almost minimizers associated with degenerate functionals in noncommutative structures, endowed by a rich geometry, as in the case of the Carnot groups of step $2$, started in \cite{FFM} when $p=2$.

The classical motivation to study this type of free boundary problem arises from flame propagation and jet flows models \cites{Ack77,Beu57,BCN}, and it is also related to shape optimization problems, see \cite{BuVel}. 

In the Euclidean setting, regularity issues concerning minimizers of such functionals were introduced in the pioneering work of Alt and Caffarelli, see \cite{AC}. Further contributions to this topic appeared in \cites{CJK04,DeJe09,JeSa15}  and we refer to \cites{CSbook,Vel} and the references within for a comprehensive overview of the subject. More recently, regularity of almost minimizes and their free boundaries were also studied in \cites{DT,DET}. A different approach is proposed in \cite{DS}, based on nonvariational techniques.

The research about the nonlinear Euclidean setting of $p$-Bernoulli-type functional, ruled by $p$-Laplace operator, appeared in \cite{DFFV}, where the authors generalize the contribution introduced in \cite{DS} without applying mean value properties to $p$-harmonic functions. However, the tools introduced in \cite{DFFV} have been useful in \cite{FFM}, facing the linear framework given by (horizontal) Bernoulli-type functional in Carnot groups of step 2, where smooth vector fields, generating the first algebra stratum, are involved. In fact, in this case, the norm of the intrinsic gradient of an (intrinsic) harmonic function is not an (intrinsic) subharmonic function, contrary to what happens in the Euclidean case. 

Hence, in this paper, combing the approach of \cite{FFM} and the strategy of \cite{DFFV}, we manage to prove the (intrinsic) Lipschitz continuous behaviour of almost minimizers of  $p$-Bernoulli-type functionals in Carnot groups of step 2. Dealing with the nonlinear framework, we pay particular attention to the noncommutative setting we are considering, since due to the linearization argument employed, we need more general regularity estimates with respect to the ones applied in \cite{FFM}.

The mathematical setting we need will be presented later on. We refer to Section \ref{sec:preliminary} for the main notation and definitions. Nevertheless, in order to introduce the main result of the paper, we recall here the definition of {\it almost minimizer}, we deal with.

Let $\G$ be a Carnot Group of step 2. Assume that $\Omega\subset\G$ is a measurable domain and $p>1$ is fixed.  We define the following energy functional
	\begin{equation}\label{def-J-p}
		J_p(u,\Omega):=\int_{\Omega}\Big( |\Hnabla u(x)|^p+\chi_{\{u>0\}}(x)\Big)\,dx
	\end{equation}
for all $u\in HW^{1,p}(\Omega)$, $u\ge0$ a.e. in $\Omega$, where $\Hnabla u$ is the so-called \emph{horizontal gradient} of $u$. 

The condition $u\geq 0$ a.e. corresponds, in the context of free boundary problems, to consider \emph{one-phase} problems (solutions which may change sign are related to \emph{two-phase} problems).

In this paper, we focus on regularity properties for almost minimizers of $J_p$ defined in \eqref{def-J-p}. The notion of almost minimizers is well-known in the Calculus of Variations. In particular, the notion of almost minimizers associated with energy functionals was introduced by Anzellotti in \cite{Anz83}. 
Concerning Bernoulli-type energy functionals, almost minimizers have been studied in \cites{DT,DEST,DEST23,DS} in the Euclidean setting and in \cite{FFM} within Carnot groups. In the setting of this work, the precise notion of almost minimizers for $J_p$ we are dealing with is the following one.
\begin{defin}\label{def-AM} Let $\kappa\ge0$, $\beta>0$ be constants and let  $\Omega\subseteq\G$ be an open set. We say that $u\in HW^{1,p}(\Omega)$ is an almost minimizer for $J_p$ in $\Omega$, with constant $\kappa$ and exponent $\beta$, if $u\geq 0$ $\mathcal{L}^n$-almost everywhere in $\Omega$ and
	\begin{equation}\label{inequality-of-J-p-u-alm-min}
			J_p(u,B_\varrho(x))\leq(1+\kappa \varrho^\beta)J_p(v,B_\varrho(x)),
	\end{equation}
		for every (intrinsic) ball $B_\varrho(x)$ such that $\overline{B_\varrho(x)}\subset \Omega$ and for every $v\in HW^{1,p}(B_\varrho(x))$ such that $u-v\in HW^{1,p}_0(B_\varrho(x))$.
\end{defin}
As well as in \cite{FFM}, the rough idea of the definition of almost minimizes is that, locally, the energy is not necessarily minimal among all competitors but \emph{almost minimal}, in the sense that it cannot decrease by a factor more than $1+\kappa \varrho^\beta$, where $\varrho$ is the scale of the localization. Thus, almost minimizers can be considered as perturbations of minimizers and such perturbations have a smaller contribution at small scales. In general, almost minimizers only satisfy a variational inequality but they are not solutions to some partial differential equations. Thus, the main problem in facing their regularity properties is the lack of monotonicity formulas, contrary to what usually happens to minimizers.

Now, we are in position to state our main regularity result. Here the integer $Q$ is the \emph{homogeneous dimension} of $\G$, see Section \ref{sec:preliminary}.
\begin{thm}\label{theor-Lipsch-contin-alm-minim}
    Suppose that $p>p^\#:=\frac{2Q}{Q+2}$. Let $u$ be an almost minimizer for  $J_p$ in  $B_1$ with constant  $\kappa$ and exponent  $\beta$. Then,
		\[\left\|\Hnabla u\right\|_{L^{\infty}(B_{1/2})}\le C\Big(\left\|\Hnabla u\right\|_{L^p(B_1)}+1\Big),\]
    where  $C>0$ is a constant depending on $\kappa$, $\beta$, $Q$ and $p$.
		
    In addition, $u$ is uniformly Lipschitz continuous in a neighborhood of  $\left\{u=0\right\}$, namely if  $u(0)=0$ then
		\[\left|\Hnabla u\right|\le C\quad \mbox{in }B_{r_0},\]
    for some  $C>0$, depending only on  $Q$, $p$, $\kappa$ and  $\beta$, and  $r_0\in(0,1)$, depending on  $Q$, $p$, $\kappa$, $\beta$ and  $\|\Hnabla u\|_{L^p(B_1)}$.
\end{thm}
Hence, the main contribution of our paper concerns the intrinsic Lipschitz regularity of almost minimizers to \eqref{def-J-p}. Such a result extends the regularity result proved in \cite{FFM}, only obtained in the case of $p=2$.

In particular, to prove the intrinsic Lipschitz regularity of the almost minimizer stated in Theorem \ref{theor-Lipsch-contin-alm-minim}, we have to face a double difficulty: the noncommutative structure of the geometry and the nonlinear form of the studied functional. More precisely, the vector fields generating the first stratum of the Lie algebra $X_1,\ldots, X_{m_1}$, in general, have no trivial commutator, i.e.
$$
\left[X_i,X_j\right]\not=0 \quad \Leftrightarrow \quad X_iX_j\not=X_jX_i, \quad\text{for every } i,j\in \{1,\ldots,m_1\}, \text{ if } i\not=j.
$$
As a consequence, denoting by $\mathcal{L}$ the positive sub-Laplacian on $\G$, it results
\begin{equation}\label{noncommuta}
X_i\mathcal{L}\not=\mathcal{L}X_i,
\end{equation} for all $i=1,\ldots,m_1$. The lack of commutativity, encoded in \eqref{noncommuta}, already faced in the linear case discussed in \cite{FFM}, translates into a lack of $\G$-harmonicity of functions $X_iu$, for $i\in \{1,\ldots,m_1\}$. Consequently, the approach adopted in \cite{DS} cannot be straightforwardly generalized. To overcome this problem, we apply some regularity estimates proved in \cite{CM22}. In addition, in this paper, we have to deal also with the nonlinear behaviour of the $p$-sub-Laplacian associated with $p$-Bernoulli functional \eqref{def-J-p}. In particular, regularity estimates of \cite{CM22}, are no longer sufficient. To apply the linearization argument performed in \cite{DFFV}, regularity estimates for subelliptic equations with variable coefficients, as ones studied in \cite{SS} in the case of the Heisenberg group,  are needed. Therefore, to generalize the results of \cite{DFFV} to the framework of the Carnot groups of step two, such regularity estimates represent a key point and their treatment appears in Section \ref{sec:reg}. We apply techniques, introduced in \cites{min,DM} in the Euclidean setting, and then employed in \cite{SS} in the case of the Heisenberg group, based on comparison estimates of corresponding equations with frozen coefficient. However, as well as in \cite{SS}, in our subelliptic setting, extra terms coming from nonvanishing commutators of the vector fields generating the first algebra stratum appear. Therefore, estimates are not always as strong as those in the Euclidean setting but they are still enough for our purposes.

As we already point out in Theorem \ref{theor-Lipsch-contin-alm-minim}, we deduce that almost minimizes $u$ for $J_p$, in an open domain $\Omega\subset\G$, are  locally intrinsic Lipschitz functions, i.e. there exists a constant $C>0$ such that for every $x,y\in B \subset\subset \Omega$
$$
|u(x)-u(y)|\leq Cd_{c}(x,y),
$$
where $d_{c}$ is the Carnot-Charath\'eodory distance, see Definition \ref{def_dcc} in the Section 2. However, this Lipschitz intrinsic regularity only implies H\"older continuous regularity, from the Euclidean point of view, since the Carnot-Charath\'eodory distance is not equivalent to the Euclidean one, see for instance \cite{BLU}*{Proposition 5.15.1.}. More precisely, denoting by $d_E$ the Euclidean distance on $\G\equiv\R^n$, for every compact set $K\subset\G$, there exists a constant $c_K>0$ such that
$$
(c_K)^{-1}\,d_E(x,y)\leq d_{c}(x,y)\leq c_K\,d_E^{1/k}(x,y),
$$
for all $x,y \in K$, where $k$ is the step of the Carnot group. In this paper, we always consider the case of $k=2$.

We also remark that the approach to the regularity theory of free boundary problems, based on monotonicity formulas, is well understood in the Euclidean setting. For example, it is well known that the so-called Alt-Caffarelli-Fredmann monotonicity formula is one of the key tools used to prove Lipschitz continuity of minimizer to two-phase problems, see \cite{ACF} and also \cite{CSbook}. Nevertheless, these techniques appear difficult to extend  to the noncommutative setting of Carnot groups. Indeed, as recently proved in \cites{FF,FG}, an intrinsic Alt-Caffarelli-Friedman monotonicity formula, written as the natural counterpart to the classical Euclidean one, seems to fail. Furthermore, the viscosity approach developed in the nonlinear framework, see for instance \cites{DP05,GleRic,FL,FL2}, appears more challenging due to the theoretical existence of characteristic points on the free boundary. Nevertheless, with our approach, we overcome these difficulties without relying on monotonicity formulas by dealing with almost minimizers inspired by the \cites{DS,DFFV,FFM}, where Lipschitz regularity has been showed. 

The paper is organized as follows. In Section \ref{sec:preliminary} we introduce the main notation and tools we rely on all the paper. Then in Section \ref{sec:har_rep} we deal with the $(\G,p)$-harmonic proving some estimates you use in our argument. In Section \ref{sec:dichotomy} we prove some of the key tools of the proof associated with a dichotomy procedure of \cite{DS}. Next, in Section \ref{sec:reg} we study some regularity results concerning subellipic equations with variable coefficient, this is a key step to perform a linearization argument as in \cite{DFFV}. Section \ref{sec:impr} concludes the dichotomy argument, here we present an \emph{improvement} of the result in Section \ref{sec:dichotomy} using the regularity estimated obtained in Section \ref{sec:reg}. Finally, in Section \ref{sec:lip1} we conclude with the proof of Theorem \ref{theor-Lipsch-contin-alm-minim}.

\section{Notation and preliminary results}\label{sec:preliminary}

\subsection{Carnot Groups}

A connected and simply connected Lie group $(\G,\cdot)$  (in general non-abelian) is said to be a  {\em Carnot group  of step $k$} if its Lie algebra ${\mathfrak{g}}$ admits a {\it $k$-step stratification}, i.e. there exist linear subspaces $V_1,...,V_k$ such that
\begin{equation}\label{stratificazione}
{\mathfrak{g}}=V_1\oplus\ldots\oplus V_k,\quad [V_1,V_i]=V_{i+1},\quad
V_k\neq\{0\},\quad V_i=\{0\}{\textrm{ if }} i>k,
\end{equation}
where $[V_1,V_i]$ is the subspace of ${\mathfrak{g}}$ generated by the commutators $[X,Y]$ with $X\in V_1$ and $Y\in V_i$. The first layer $V_1$, the so-called {\sl horizontal layer},  plays a key role in the theory since it generates all the algebra $\mathfrak g$ by commutation. We point out that a stratified Lie algebra can admit more than one stratification, however, the stratification turns out to be unique up to isomorphisms, therefore, the related Carnot group structure is essentially unique (see \cite{LD16}*{Proposition 1.17}). Note that when $k=1$ the group $\G$ is Abelian, we return to the Euclidean situation.

Setting $m_i=\dim(V_i)$ and $h_i=m_1+\dots +m_i$, with $h_0=0$, for $i=1,\dots,k$ (so that $h_k=n=\dim \mathfrak g=\dim \G$), we choose a basis $e_1,\dots,e_n$ of $\mathfrak g$ adapted to the stratification , that is such that
$$e_{h_{j-1}+1},\ldots,e_{h_j}\;\text {is a basis of}\; V_j\;\text{for each}\; j=1,\ldots, k.$$ 
Let $X=\{X_1,\dots,X_{n}\}$ be the family of left-invariant vector fields such that $X_i(e)=e_i$, $i=1,\dots,n$, where $e$ is the identity of $(\G,\cdot)$. 
Thanks to \eqref{stratificazione}, the subfamily $\{X_1,\dots,X_{m_1}\}$ generates by commutations all the other vector fields, we will refer to $X_1,\ldots,X_{m_1}$ as {\it generating vector fields} of $(\G,\cdot)$.  

The map $X\mapsto X(e)$, that associated with a left-invariant vector field $X$ its value in $e$, defines an isomorphism from $\mathfrak g$ to $T\G_e$ (in turn identified with $\R^n$). We systematically use these identifications. Furthermore, by the assumption that $\G$ be simply connected, the exponential map is a global real-analytic diffeomorphism from $\mathfrak g$ onto $\G$ (see e.g. \cites{V,CGr}), so any $x\in\G$ can be written in a unique way as $x=\exp(x_1X_1+\dots+x_nX_n)$. Using these {\it exponential coordinates}, we identify $x$ with the $n$-tuple $(x_1,\ldots,x_n)\in\R^n$ and we identify $\G$ with $(\R^n,\cdot)$ where the explicit expression of the group operation $\cdot$ is determined by the Campbell-Hausdorff formula (see e.g. \cites{BLU,FS}). In this coordinates $e=(0,\dots,0)$ and $(x_1,\dots,x_n)^{-1}=(-x_1,\dots,-x_n)$, and the adjoint operator in $L^2(\G)$ of $X_j$, $X_j^{*}$, $j=1,\dots,n$, turns out to be $-X_j$ (see, for instance,\cite{FSSC03}*{Proposition 2.2}). Moreover, if $x\in \G$ and $i=1,\ldots,k$, we set $x^{(i)}:= (x_{h_{i-1}+1},\ldots,x_{h_{i}})\in \R^{m_i}$, so that we can also identify $x$ with $[x^{(1)},\ldots,x^{(k)}]\in \R^{m_1}\times \ldots \times \R^{m_k}=\R^n$.

Two important families of automorphism of $\G$ are the so-called \emph{intrinsic translations} and the \emph{intrinsic dilations} of $\G$. For any $x\in\G$, the {\it (left) translation} $\tau_x:\G\to\G$ is defined as 
$$ z\mapsto\tau_x z:=x\cdot z. $$
For any $\lambda >0$, the {\it dilation} $\delta_\lambda:\G\to\G$, is defined as 
\begin{equation}\label{dilatazioni}
\delta_\lambda(x_1,...,x_n)=(\lambda^{\alpha_1}x_1,...,\lambda^{\alpha_n}x_n),
\end{equation} 
where $\alpha_i\in\N$ is called {\it homogeneity of the variable} $x_i$ in $\G$ (see \cite{FS} Chapter 1) and is defined as
\begin{equation}\label{omogeneita2}
\alpha_j=i \quad\text {whenever}\; h_{i-1}+1\leq j\leq h_{i},
\end{equation}
hence $1=\alpha_1=...=\alpha_m<\alpha_{m+1}=2\leq...\leq\alpha_n=k.$

From the definition \eqref{dilatazioni}, one easily verifies the following properties of intrinsic dilations.
\begin{lem}\label{DilationsProperties}
For all  $\lambda, \mu >0$ one has:
\begin{enumerate}
\item[(1)]  $\delta_{1} = \mathrm{id}_\G$;
\item[(2)] $\delta^{-1}_{\lambda} = \delta_{\lambda^{-1}}$;
  \item[(3)]
  $\delta_{\lambda}\circ \delta_{\mu} =\delta_{\lambda \mu}$;
\item[(4)] for every $p, p'\in \G$ one has $\delta_\lambda(p) \circ \delta_\lambda(p') = \delta_\lambda(p\cdot p')$.
\end{enumerate}
\end{lem}

By left translation, the horizontal layer defines a subbundle of the tangent bundle $T\G$ over $\G$: the subbundle of the tangent bundle spanned by the family of vector fields $X=(X_1,\dots,X_{m_1})$ plays a crucial role in the theory, it is called the {\it horizontal bundle} $H\G$; the fibres of $H\G$ are 
$$ H\G_x=\mbox{span }\{X_1(x),\dots,X_{m_1}(x)\},\qquad x\in\G.$$
A subriemannian structure is defined on $\G$, endowing each fibre of $H\G$ with a scalar product $\scal{\cdot}{\cdot}_{x}$ and with a norm $\vert\cdot\vert_x$ that make the basis $X_1(x),\dots,X_{m_1}(x)$ an orthonormal basis. That is if $v=\sum_{i=1}^{m_1} v_iX_i(x)=(v_1,\dots,v_{m_1})$ and $w=\sum_{i=1}^{m_1} w_iX_i(x)=(w_1,\dots,w_{m_1})$ are in $H\G_x$, then $\scal{v}{w}_{x}:=\sum_{j=1}^{m_1} v_jw_j$ and $|v|_x^2:=\scal{v}{v}_{x}$.

The sections of $H\G$ are called {\it horizontal sections}, and for any $x\in\G$, a vector of $H\G_x$ is an {\it horizontal vector} while any vector in $T\G_x$ that is not horizontal is a vertical vector. Each horizontal section is identified by its canonical coordinates with respect to this moving frame $X_1(x),\dots,X_{m_1}(x)$. This way, an horizontal section $\varphi$ is identified with a function $\varphi =(\varphi_1,\dots,\varphi_{m_1}) :\R^n \rightarrow\R^{m_1}$. When dealing with two horizontal sections $\phi$ and $\psi$, we drop the index $x$ in the scalar product and in the norm. 

We collect in the following results some properties of the group operation and of the canonical vector fields, see \cite {BLU}.

\begin {prop} The group operation has the following form
   $$ x \cdot y = x + y + \mathcal {Q} (x, y), \quad \text{for all } x, y \in \mathbb {G}$$
where $ \mathcal {Q} = (\mathcal {Q} _1, \ldots, \mathcal {Q} _n): \mathbb {G}\times \mathbb {G} \rightarrow \mathbb {G}$
and every $ \mathcal {Q} _i $, for $ i = 1, \ldots, n $, is a homogeneous polynomial of degree $ \alpha_i $ respecting the intrinsic dilations of $ \mathbb {G} $, that is
   $$ \mathcal {Q} _i (\delta_ \lambda x, \delta_ \lambda y) = \lambda ^ {\alpha_i} \mathcal {Q} _i (x, y) \quad \text{for all } x, y \in \mathbb {G}\quad \text{and} \quad \text{for } \lambda> 0. $$ Moreover, for all $x, y \in \mathbb {G} $ we have
\begin {itemize}
\item [(i)] $ \mathcal {Q} $ is antisymmetric, that is
    $$ \mathcal {Q} _i (x, y) = - \mathcal {Q} _i (-y, -x), \quad \text{for } i = 1, \ldots, n. $$
\item [(ii)] $$ \mathcal {Q} _1 (x, y) = \ldots = \mathcal {Q} _ {m_1} (y, x) = 0 $$
$$ \mathcal {Q} _i (x, 0) = \mathcal {Q} _i (0, y) = 0 \text{ and } \mathcal {Q} _i (x, x) = \mathcal {Q} _i (x, -x) = 0, \quad \text { for } m_1 <i \leq n $$
$$ \mathcal {Q} _i (x, y) = \mathcal {Q} _i (x_1,, x _ {{h_j} -1}, y_1, \ldots, y _ {{h_j} -1}), \quad \text { if } 1 <j \leq k \text { and } i \leq h_j $$
\item [(iii)] $$ \mathcal {Q} _i (x, y) = \sum_ {k, h} {\mathcal {R} _ {h, k} ^ i (x, y) (x_ky_h- x_hy_k)} $$ where the functions $ \mathcal {R} _ {h, k} ^ i $ are homogeneous polynomials of degree $ \alpha_i- \alpha_k- \alpha_h $ which respect the intrinsic dilations and the sum is extended to all $ h $ and $ k $ such that $ \alpha_h + \alpha_k \leq \alpha_i $.
\end {itemize} \label {group_operation}
\end {prop}

The following result is contained in \cite{FSSC03}*{Proposition 2.2}.

\begin {prop} \label {campi-omogenei0} The vector fields $ X_j $ have polynomial coefficients and are of the form
\begin{equation}\label{campi-omogenei}
 X_j (x) = \partial_j + \sum_ {i> h_l} ^ n {q_ {i, j} (x) \partial_i}, \quad \text{ for } \: j = 1, \ldots, n \: \: \text{ and } \: \: j \leq h_l
\end{equation}
where $ q_ {i, j} (x) = \frac {\partial \mathcal {Q} _i} {\partial y_j} (x, y) \rvert_ {y = 0} $ (the $ \mathcal {Q} _i $ are those defined in the proposition \ref {group_operation}) for $ h_ {l-1} <j \leq h_l $ and $ 1 \leq l \leq k $; so if $ h_ {l-1} <j \leq h_l $ then $ q_ {i, j} (x) = q_ {i, j} (x_1, \ldots, x_ {h_ {l-1}}) $ and $ q_ {i, j} (0) = 0 $.
\end {prop}

\subsection{Intrinsic distance and gauge pseudo-distance} 
An absolutely continuous curve $ \gamma: [0, T] \rightarrow \mathbb {G} $ is called \emph {sub-unitary} with respect to $ X_1, \ldots, X_ {m_1} $ if it is a \emph {horizontal curve}, that is, if there exist real measurable functions $ c_1, \ldots, c_ {m_1}: [0, T] \rightarrow \mathbb {R} $ such that
$$ \dot {\gamma} (s) = \sum_ {j = 1} ^ {m_1} {c_j (s) X_j (\gamma (s))}, \quad \text {for $ \mathcal {L} ^ 1 $ -a.e.} s \in [0, T] $$ with $ \sum_ {j = 1} ^ {m_1} {c_j(s) ^ 2} \leq 1 $, for $ \mathcal {L} ^ 1 $ -a.e.  $s \in [0, T] $.

\begin {defin}[Carnot-Carath\'eodory distance] Let $ \mathbb {G} $ be a Carnot group. For $ p, q \in \mathbb {G} $ we define their \emph {Carnot-Carath\'eodory distance} $ d_c (p, q) $ as
$$ d_c (p, q): = \mathrm {inf} \{T> 0: \text{it exists an sub-unitary curve } \gamma \text { with },\: \gamma (0) = p, \: \gamma (T) = q \}. $$
\end{defin}\label{def_dcc}

The last definition is well-placed: the set of sub-unitary curves connecting $ p $ and $ q $ is non-empty, by Chow's theorem (\cite{BLU}*{Theorem 19.1.3}), since by \eqref{stratificazione}, the rank of the Lie algebra generated by $ X_1, \ldots, X_ {m_1} $ is $n$; hence, $d_c$ is a distance on $\G$ inducing the same topology as the standard Euclidean distance. We shall denote by $ B_r(p) $ the open balls, centred in $p$ of radius $r>0$, associated with the distance $ d_c $. For the sake of simplicity, if $p=e$ we will use the notation $B_r(e)=B_r$. 

The Carnot-Carath\'eodory metric $d_c$ is equivalent to a more explicitly defined pseudo-distance, called the  \emph{gauge pseudo-distance}, defined as follows.  Let $||\cdot ||$ denote the Euclidean distance to the origin in the Lie algebra $\mathfrak g$. For $u = u_1 + \cdots + u_k \in \mathfrak g$ with $u_i\in V_i$, one defines
\begin{equation}\label{gauge}
|u|_{\mathfrak g} := \left(\sum_{i=1}^k ||u_i||^{2r!/i}\right)^{\frac{1}{2r!}}.
\end{equation}
The \emph{non-isotropic gauge} in $\G$ is defined by letting 
\begin{equation}\label{nig}
|p|_{\G} = |\exp^{-1} p|_{\mathfrak g}, \quad\quad \ \ \ \ \ p\in \G,
\end{equation}
see \cite{Fo} and \cite{FS}. Since the exponential mapping $\exp:\mathfrak g \to \G$ is of class $C^\infty$ (actually analitic) diffeomorphism, the maps $p\to |p|_\G$ is $C^\infty(\G\setminus \{e\})$.
Notice that from \eqref{gauge} and \eqref{dilatazioni} we have for any $\lambda>0$
\begin{equation}\label{gau}
|\delta_\lambda(p)|_\G = \lambda |p|_\G.
\end{equation}
The \emph{gauge pseudo-distance} in $\G$ is defined by
\begin{equation}\label{rho0}
d(p,q):=|p^{-1} \cdot q|_\G.
\end{equation}
The function $d$ has all the properties of a distance, except symmetry and the triangle inequality, which is satisfied with a universal constant, usually different from one on the right-hand side, see \cites{FS,BLU}. Since the dilatations are group automorphisms, from \eqref{rho0} and \eqref{gau}, it follows that $d$ id homogeneous of degree one with respect to the group dilatation, that is for any $\lambda>0$ 
\begin{equation}\label{gauge2}
d(\delta_\lambda(p),\delta_\lambda(p')) = \lambda d(p,p').
\end{equation}
It is well known, see for instance Proposition 5.1.4 in \cite{BLU}, that there exist universal constants $\Lambda_\G$ such that for $p\in \G$ one has
\begin{equation}\label{rhod}
\Lambda_\G^{-1} |p|_\G \le d_c(e,p) \le \Lambda_\G |p|_\G.
\end{equation}


The integer
   $$ Q: = \sum_ {j = 1} ^ n {\alpha_j} = \sum_ {i = 1} ^ k {i\text{dim}(V_i)} $$
is the \emph {homogeneous dimension} of $ \mathbb {G} $. It is the Hausdorff dimension of $ \mathbb {G} \cong \mathbb {R} ^ n $ with respect to the distance $ d_c $, see \cite{Mit}.

The $ n $ -dimensional Lebesgue measure $\mathcal {L} ^ n $, is the Haar measure of the group $ \mathbb {G} $ that is, for every $ \mathcal {L} ^ n $ -measurable set $E \subset \mathbb {G}$ and for each $x\in\G$ it results $\mathcal {L} ^ n \left (x \cdot E \right) = \mathcal {L} ^ n \left (E \right)$. Moreover, if $ \lambda> 0 $ then $ \mathcal {L} ^ n \left (\delta_ \lambda \left (E \right) \right) = \lambda ^ Q \mathcal {L} ^ n \left (E \right)$. In particular, for any $ r> 0 $ and any $ p \in \mathbb {G} $, it holds $$ \mathcal {L} ^ n \left (B_r (p) \right) = r ^ Q \mathcal {L} ^ n \left (B_1 (p)\right) = r ^ Q \mathcal {L} ^ n \left (B_1 \right) $$ where $ Q $ is the homogeneous dimension.

All the spaces $ L ^ p (\mathbb {G}) $ that we will consider are defined with respect to the Lebesgue measure $ \mathcal {L} ^ n $. If $A\subset \G$ is $\mathcal{L}^n$-measurable, we write $|A|=\mathcal{L}^n(A)$.

A map $ L: \mathbb {G} \rightarrow \mathbb {R} $ is \emph{$\mathbb {G} $-linear} if it is a group homeomorphism from $ \mathbb {G} \equiv (\mathbb {R} ^ n, \cdot) $ to $ (\mathbb {R}, +) $ and if it is homogeneous of degree 1 with respect to the intrinsic dilations of $ \mathbb {G} $, that is $L (\delta_ \lambda x) = \lambda Lx$ for $\lambda> 0$ and $x \in \mathbb {G}$. Similarly, we say that a map $\ell:\G\to\R$ is $\G$-affine if there exists a linear map and $L$ and $c\in\R$ such that $\ell(x)=L(x)+c$, for every $x\in\G$.

Given a basis $X_1,\ldots, X_{n}$, all $\G$-linear maps are represented as follows.
\begin{prop}\label{rap-lin}
    A map $L:\G\to\R$ is $\G$-linear if and only if there is $a=(a_1,\ldots,a_{m_1})\in\R^{m_1}$ such that, if $x=(x_1,\ldots,x_n)\in\G$, then $$L(x)=\sum_{i=1}^{m_1}a_ix_i.$$
\end{prop}
Moreover, if $x=(x_1,\ldots,x_n)\in\G$  and $x_0\in\G$ are given, we set 
\begin{equation*}
\pi_{x_0}(x):=\sum_{j=1}^{m_1}x_jX_j(x_0).
\end{equation*}

\subsection{Folland-Stein and horizontal Sobolev classes}
Fixed $\Omega \subseteq \G$, the action of vector fields $X_j$, with $j=1,\ldots,m_1$ on a function $f:\Omega\to\R$ is specified by the Lie derivative: we say that a function $f$ is differentiable along direction $ X_j $, for $ j = 1, \ldots, m_1 $, in $ x_0 \in \G $ if the map $ \lambda \mapsto f (\tau_ {x_0} (\delta_ \lambda (e_j)) $ is differentiable in $ \lambda = 0 $, where $ e_j $ is the $ j $-th vector of the canonical basis of $ \mathbb {R} ^ {m_1} $. In this case, we will write $$ X_j f (x_0) = \left. \frac {d} {d \lambda} f (\tau_ {x_0} (\delta_ \lambda e_j) ) \right | _ {\lambda = 0}.$$ 
If instead $ f \in L ^ 1_ {loc} (\Omega) $, $ X_jf $ exists in a weak sense, if 
  $$ \int_ \Omega {fX_j \varphi} \: d\mathcal{L}^n = - \int_ \Omega {\varphi X_jf} \: d\mathcal{L}^n$$
for each $ \varphi \in C _0 ^ \infty (\Omega) $.

Once a basis $X_1,\ldots, X_{m_1}$ of the horizontal layer is fixed we define, for any function $f:\Omega\to \R$ for which the partial derivatives $X_jf$ exist, the horizontal gradient of $f$, denoted by $\Hnabla f$, is defined as the horizontal section
\begin{equation}\label{horiz-grad}
\Hnabla f:=\sum_{i=1}^{m}(X_if)X_i,
\end{equation}
whose coordinates are $(X_1f,\ldots,X_{m_1}f)$. Moreover, if $\phi=(\phi_1,\dots,\phi_{m_1})$ is an horizontal section such that $X_j\phi_j\in L^1_{loc}(\G)$ for $j=1,\dots,m_1$, we define $\divg\phi$ as the real-valued function
\begin{equation}\label{horiz-div}
\divg\phi:=-\sum_{j=1}^{m_1}X_j^*\phi_j=\sum_{j=1}^{m_1}X_j\phi_j.
\end{equation}

The positive \emph{sub-Laplacian} operator on $\G$ is the second-order differential operator, given by
$$
\mathcal{L}:= \sum_{j=1}^{m_1}X^*_jX_j=-\sum_{j=1}^{m_1} X_j^2.
$$
It is easy to see that
$$
\mathcal{L}u = - \divg (\Hnabla u).
$$
The operator $\mathcal L$ is left-invariant with respect to group translations and homogeneous of degree two with respect to group dilatations, i.e. for any $x\in\G$ and $\lambda>0$ we have
$$\mathcal L (u\circ \tau_x) = (\mathcal Lu)\circ \tau_x, \qquad \mathcal{L}(\delta_\lambda u) = \lambda^2 \delta_\lambda(\mathcal{L} u).$$

Furthermore, by the assumptions \ref{stratificazione}, the system $\{X_1,...,X_{m_1}\}$ satisfies the finite rank condition 
\[
\text {rank Lie} [X_1,\ldots,X_{m_1}] = n,
\] 
therefore by  H\"ormander's theorem \cite{H} the operator $\mathcal{L}$ is hypoelliptic. However, when the step $k$ of $\G$ is greater than one, the operator $\mathcal{L}$ fails to be elliptic at every point of $\G$.

More in general, given $1<p<\infty$ we consider the quasilinear operator, known as the \emph{$p$-sub-Laplacian}, defined by
\begin{equation}\label{plap}
\mathcal{L}_p u:=\sum_{j=1}^{m_1} X_j(|\Hnabla u|^{p-2} X_ju ).
\end{equation}

The H\"older classes $C^{k, \alpha}$ has been introduced by Folland and Stein, see \cites{Fo, FS}. The functions in these classes are H\"older continuous with respect to the metric $d_c$  

\begin{defin}[Folland-Stein classes]\label{eq:fsclass}
Let $\Omega \subseteq \G$ an open set and $0<\alpha\leq 1$ and $u:\Omega\to \R$ a function. We say that
$u\in C^{0, \alpha}(\Omega)$ if there exists a constant $M>0$, such that 
$$|u(x) - u(y)| \leq M\, d_c(x,y)^{\alpha}, \qquad\text{for every }\ x,y\in\Omega.$$Defining the H\"older seminorm of $u\in C^{0,\alpha}(\Omega)$ as 
\begin{equation}\label{eq:holdsemi}
[u]_{C^{0,\alpha}(\Omega)}:= \underset{\underset{x \neq y}{x,y\in \Omega}}{\sup} \frac{|u(x)-u(y)|}{d_c(x,y)^{\alpha}},
\end{equation}
the space $C^{0,\alpha}(\Omega)$ is a Banach space with respect to the norm $$\|u\|_{C^{0,\alpha}
(\Omega)}:= \|u\|_{L^\infty(\Omega)} + [u]_{C^{0,\alpha}(\Omega).}$$ 
For any $k\in \mathbb N$, the spaces $C^{k, \alpha}(\Omega)$ are defined inductively as follows: we say that $u \in C^{k, \alpha}(\Omega)$ if $X_i u \in C^{k-1,\alpha}(\Omega)$ for every $i=1,\ldots, m_1$.
\end{defin}

In the following, we will place with $C^{k,\lambda}(\Omega,H\G)$ the space of all the horizontal sections $\phi:\Omega\to H\G$ with $\phi:=(\phi_1,\ldots,\phi_{m_1})$, such that $\phi_j\in C^{k,\lambda}(\Omega)$, for $j=1,\ldots,m_1$. While if $u\in C^{k,\alpha}(\Omega')$, for any $\Omega'\subset\subset\Omega$ we will set $u\in C^{k,\alpha}_{loc}(\Omega)$.

Finally, note that the class $C^{0,1}(\Omega)$ coincides with the class of Lipschitz continuous function on $\Omega$ with respect to the metric $d_c$.
\begin{defin}[Horizontal Sobolev spaces]\label{def:hsob}
Given an open set $\Omega \subseteq \G$, a function $u:\Omega\to \R$ and $1\leq p<\infty$, the Horizontal Sobolev spaces are defined as 
$$ HW^{1,p}(\Omega):=\left\{u\in L^p(\Omega): \ X_ju\in L^p(\Omega),\ \text{for all }\, j=1,\ldots, m_1\right\},$$ 
which is a Banach space with the norm $$\|u\|_{HW^{1,p}(\Omega)} := \|u\|_{L^p(\Omega)} + \|\Hnabla u\|_{L^p(\Omega, H\G)}.$$
\end{defin}

The subspace $HW^{1,p}_0(\Omega)$ of $HW^{1,p}(\Omega)$ are defined as the closure with respect the norm $\|\cdot\|_{HW^{1,p}(\Omega)}$ of $C^\infty_0(\Omega)$.

\begin{defin}[Morrey-Campanato spaces]\label{defin-Campanato-spaces}
	Let $\Omega \subseteq \G$. For every $1\le p<+\infty$ and $\lambda \in (0,+\infty)$,  an horizontal section $f=(f_1,\ldots,f_{m_1})\in L^p_{loc}(\Omega,H\G)$, is said to be in the Morrey-Campanato space $\mathcal{E}^{\lambda,\,p}(\Omega,H\G)$ if
	$$[f]_{\mathcal{E}^{\lambda,\,p}(\Omega,H\G)}:=\sup_{B\subset\Omega}\left(\frac1{|B|^{1+p\lambda}}\int_B|f(y)-\mathbf{f_B}(y)|^p\,dy\right)^{1/p}<+\infty,$$
	where the supremum is taken over all balls $B\subset \Omega$ and $\mathbf{f_B}$ is the constant horizontal section given by
	$$\mathbf{f_B}(y):=\sum_{i=1}^{m_1}\left(\fint_Bf_i(z)\,dz\right)X_i(y),\quad \text{for } y\in\Omega.$$
\end{defin}

\begin{remark}
	For every $u=(u_1,\ldots,u_{m_1})\in \mathcal{E}^{\lambda,p}(\Omega,H\G)$, the quantity $[u]_{\mathcal{E}^{p,\lambda}(\Omega,H\G)}$ is a seminorm in $\mathcal{E}^{\lambda,p}(\Omega,H\G)$ and it is equivalent to the quantity \begin{equation}\label{equivalent-Campanato-seminorm} 
    \sup_{B\subset\Omega}\left(\frac1{|B|^{1+p\lambda}}\inf_{\boldsymbol{\xi}}\int_{B}\left|u(x)-\bxi(x)\right|^p\,dx\right)^{1/p},
	\end{equation}
	where the infimum is taken over all the constant horizontal sections such that 
	\begin{equation}\label{b-xi}
		\bxi(x):=\sum_{i=1}^{m_1}\xi_iX_i(x), \quad \text{for }x\in\G \text{ and some }\xi:=(\xi_1,\ldots,\xi_{m_1})\in\R^{m_1}.
	\end{equation} 
    For the proof we refer to Remark 3.6 in \cite{FFM}.
    \end{remark}


\begin{thm}\emph{(\cite{MS79}*{Theorem 5})}\label{theo-isomorph-Campanato-spaces-Holder-spaces}
	For every $\Omega \subseteq \G$, $\lambda\in (0,1)$ and $p\in[1,+\infty)$ one has $$ \mathcal{E}^{\lambda,p}(\Omega,H\G)=C^{0,\lambda}(\Omega, H\G).$$
	More precisely, it results that a function $\phi$ belongs to $\mathcal{E}^{\lambda,p}(\Omega,H\G)$ if and only if $\phi$ is equal $\mathcal{L}^n$-almost everywhere to a function $\psi$ with belongs to the H\"older class $C^{0,\lambda}(\Omega,H\G)$. Moreover the seminorms $[\phi]_{\mathcal{E}^{\lambda,p}(\Omega,H\G)}$ and $[\psi]_{C^{0,\lambda}(\Omega,H\G)}$ are equivalent.
\end{thm}

\section{The $(\G,p)$-harmonic replacement} \label{sec:har_rep}
In this section, we recall the notion of $(\G,p)$-harmonic replacement.

\begin{defin} Let $\Omega\subseteq \G$ an open set and $u\in HW^{1,p}( \Omega)$. We say that $v\in HW^{1,p}( \Omega)$ is the $(\G,p)$-harmonic replacement of $u$ in $ \Omega$ if
\begin{equation}\label{p-g-minim}
    \int_{ \Omega}\left|\Hnabla v\right|^p \,dx=\min_{u-w\in HW^{1,p}_0( \Omega)}\int_{ \Omega}\left|\Hnabla w\right|^p \,dx.
\end{equation}
\end{defin}
	
	In particular, if $v$ is the $(\G,p)$-harmonic replacement of $u$ in $ \Omega$, it satisfies the critical condition
	\begin{equation}\label{weak-definition-of-p-harmonic-function}
		\int_{ \Omega}\left|\Hnabla v(x)\right|^{p-2} \left \langle \Hnabla v(x), \Hnabla\varphi(x) \right \rangle\,dx=0
	\end{equation}
	for all $\varphi\in HW^{1,p}_{0}( \Omega)$, i.e.  $v$ is the unique weak solution for the Dirichlet problem 
 \begin{equation}\label{p-dir-prob}
     \left\{\begin{matrix}
        \mathcal{L}_pv=0\qquad \text{in } \Omega
        \vspace{0.1cm} \\
        u-v\in HW^{1,p}_0( \Omega)
        \end{matrix}\right. 
\end{equation}
	
    In the following Lemma, we provide some energy estimates for the $(\G,p)$-harmonic replacement, which will be useful in the sequel.
	
    \begin{lem}\label{lemma-p-harm-replac}
		Let $\Omega\subseteq \G$ open and $u\in HW^{1,p}( \Omega)$.
		Let $v$ be the $(\G,p)$-harmonic replacement of $u$ in $ \Omega$.
		Then,
		\begin{itemize}
			\item[\emph{(}i\emph{)}] if $1<p<2,$ then
			\begin{equation}\label{first-ineq-lemma-p-harm-repl}\begin{split} &
					\int_{ \Omega}\left|\Hnabla u(x)-\Hnabla v(x)\right|^p\,dx\\
					&\qquad\le C\left(\int_{ \Omega}\left(\left|\Hnabla u(x)\right|^p-\left|\Hnabla v(x)\right|^p\right)\,dx\right)^{\frac{p}{2}}
					\left(\int_{ \Omega}\big(\left|\Hnabla u(x)\right|+\left|\Hnabla v(x)\right|\big)^p\,dx\right)^{1-\frac{p}{2}},\end{split}
			\end{equation}
			for some positive constant $C$ depending only on $\G$ and $p$;
			\item[\emph{(}ii\emph{)}]if $p\ge 2,$ then
			\begin{equation}\label{second-ineq-lemma-p-harm-repl}
				\int_{ \Omega}\left|\Hnabla u(x)-\Hnabla v(x)\right|^p\,dx
				\le C\int_{ \Omega}\big(\left|\Hnabla u(x)\right|^p-\left|\Hnabla v(x)\right|^p\big)\,dx,
			\end{equation}
			for some positive constant $C$ depending only on $\G$ and $p$.		\end{itemize}		
	\end{lem}	
	
	\begin{proof}
		For all $s\in[0,1]$, we consider the family of functions $u_s(x):=su(x)+(1-s)v(x)$, so that  $u_0=v$ and $u_1=u$. As a consequence,
		\begin{align*}
			\int_{ \Omega}\Big(\left|\Hnabla u(x)\right|^p-\left|\Hnabla v(x)\right|^p\Big)\,dx
			&=\int_{ \Omega}\left(\int_0^1\frac{d}{ds}\left|\Hnabla u_s(x)\right|^p\,ds\right)\,dx\\
                &=\int_{ \Omega}\left(\int_0^1 p\left|\Hnabla u_s(x)\right|^{p-2}\left\langle\Hnabla u_s(x),\Hnabla(u-v)(x)\right\rangle\,ds\right)\,dx.
		\end{align*}
		Since $u-v\in HW^{1,p}_0(\Omega)$, using it as test function in  \eqref{weak-definition-of-p-harmonic-function}, by Fubini's Theorem we get
		\begin{align*}
			\int_{ \Omega}\Big(&\left|\Hnabla u(x)\right|^p-\left|\Hnabla v(x)\right|^p\Big)\,dx
			\\ &=\;
                p\,\left[\int_{ \Omega}\left(\int_0^1\left|\Hnabla u_s(x)\right|^{p-2}
			\left\langle\Hnabla u_s(x), \Hnabla(u-v)(x)\right\rangle\,ds\right)\,dx\right.
			\\&\qquad\qquad\qquad \left.
			-\int_0^1\left(\int_{ \Omega}\left|\Hnabla v(x)\right|^{p-2}\left\langle\Hnabla v(x), \Hnabla(u-v)(x)\right\rangle\,dx\right)\,ds\right]\\
			&=\;p\int_0^1\bigg(\int_{ \Omega\cap A}\left\langle\left|\Hnabla u_s(x)\right|^{p-2}\Hnabla u_s(x)
			-\left|\Hnabla v(x)\right|^{p-2}\Hnabla v(x),\Hnabla(u-v)(x)\right\rangle\,dx\bigg)\,ds,
		\end{align*}
        where $A=\left\{x\in\Omega:\; |\Hnabla u(x)-\Hnabla v(x)|>0\right\}.$
		Then, recalling that, 
		\begin{equation}\label{difference-u-^-s-v}
			u_s(x)-v(x)=s(u(x)-v(x)) \qquad \text{for all }x\in\Omega,
		\end{equation}
		it results
		\begin{equation}\begin{split}\label{lower-bound-integral-nabla-u-nabla-v-1}
				\int_{ \Omega}\Big(&\left|\Hnabla u(x)\right|^p-\left|\Hnabla v(x)\right|^p\Big)\,dx\\
				=\;&p\int_0^1\frac{1}{s}\bigg(\int_{ \Omega\cap A}\left\langle\left|\Hnabla u_s(x)\right|^{p-2}\Hnabla u_s(x)
			-\left|\Hnabla v(x)\right|^{p-2}\Hnabla v(x),\Hnabla(u_s-v)(x)\right\rangle\,dx\bigg)\,ds.
	\end{split}\end{equation}

		By apply some well-known inequalities (see e.g. Section 3 of \cite{DP05} or Lemma 2.1 in \cite{DAM98}), for any horizontal sections $\bxi,\bzeta:\G\to H\G$ such that $|\bxi|+|\bzeta|>0$ in $\G$ any for any $1<p<\infty$ we have 
	\begin{equation}\label{stime-p-1}
	    \left\langle\left|\bxi\right|^{p-2}\bxi-\left|\bzeta\right|^{p-2}\bzeta,\bxi-\bzeta\right\rangle \ge \gamma \left|\bxi-\bzeta\right|^2\left(\left|\bxi\right|+\left|\bzeta\right|\right)^{p-2}.
        \end{equation}
		for some positive constant $\gamma$ depending only on $\G$ and $p$. Moreover, for $p\ge 2$, for any horizontal sections $\bxi,\bzeta:\G\to H\G$
        \begin{equation}\label{stime-p-2}
	    \left\langle\left|\bxi\right|^{p-2}\bxi-\left|\bzeta\right|^{p-2}\bzeta,\bxi-\bzeta\right\rangle \ge \gamma \left|\bxi-\bzeta\right|^p        
        \end{equation}
        for some positive constant $\gamma$ of \eqref{stime-p-1}.

    \noindent{\em Step 1: the case $1<p<2$.} At first, we notice that we can apply inequalities \eqref{stime-p-1} with the choices $\bxi:=\Hnabla u$ and $\bzeta:=\Hnabla v$ on $A$. Indeed, by \eqref{difference-u-^-s-v}, for any $x\in A$, it results
    $$|\Hnabla u_s(x)|+|\Hnabla v(x)|\geq |\Hnabla u_s(x)-\Hnabla v(x)|=s|\Hnabla u(x)-\Hnabla v(x)|>0.$$   
    Thus, by inequality \eqref{stime-p-1} and \eqref{lower-bound-integral-nabla-u-nabla-v-1}, we obtain that
		\begin{equation}\label{lower-bound-integral-nabla-u-nabla-v-bis-1}\begin{split}
				&\int_{ \Omega}\Big(\left|\Hnabla u(x)\right|^p-\left|\Hnabla v(x)\right|^p\Big)\,dx\\
				&\qquad \qquad \quad \ge\;p\,\gamma\int_0^1\frac{1}{s}\left(\int_{ \Omega\cap A}\left|\Hnabla u_s(x)
					-\Hnabla v(x)\right|^2(\left|\Hnabla u_s(x)\right|
					+\left|\Hnabla v(x)\right|)^{p-2}\,dx\right)\,ds \\
                    &\qquad \qquad \quad =\;p\,\gamma\int_0^1 {s}\left(\int_{ \Omega\cap A}\left|\Hnabla u(x)
					-\Hnabla v(x)\right|^2(\left|\Hnabla u_s(x)\right|
					+\left|\Hnabla v(x)\right|)^{p-2}\,dx\right)\,ds\\
                    &\qquad \qquad \quad =\;p\,\gamma\int_0^1 {s}\left(\int_{ \Omega}\left|\Hnabla u(x)
					-\Hnabla v(x)\right|^2(\left|\Hnabla u_s(x)\right|
					+\left|\Hnabla v(x)\right|)^{p-2}\,dx\right)\,ds,            \end{split}\end{equation}
            where in the first equality we used the identity in \eqref{difference-u-^-s-v}. 

		
		Furthermore, recalling that $s\in[0,1]$, for any $x\in \Omega$ we have that \vspace{0.2cm}
		\begin{equation*}\begin{split}
		    \left|\Hnabla u_s(x)\right|+\left|\Hnabla v(x)\right|&\le s\left|\Hnabla u(x)\right|+(1-s)\left|\Hnabla v(x)\right|+\left|\Hnabla v(x)\right| \\
            &=s\left|\Hnabla u(x)\right|+(2-s)\left|\Hnabla v(x)\right|\\
            &\leq 2\,\big(\left|\Hnabla u(x)\right|+\left|\Hnabla v(x)\right|\big),
	     \end{split}\end{equation*}  
	   which together with \eqref{lower-bound-integral-nabla-u-nabla-v-bis-1}, leads to 		
		\begin{equation}\label{lower-bound-integral-nabla-u-nabla-v-1-<-p-leq-2-1}\begin{split}
				\int_{ \Omega}\Big(&\left|\Hnabla u(x)\right|^p-\left|\Hnabla v(x)\right|^p\Big)\,dx\\
				&\ge\; p\,\gamma\, 2^{p-2}\int_0^1 {s}\left(\int_{ \Omega} \left|\Hnabla u(x)-\Hnabla v(x)\right|^2
				\big(\left|\Hnabla u(x)\right|+\left|\Hnabla v(x)\right|\big)^{p-2}\,dx\right)\,ds\\
				&=\;p\,\gamma\, 2^{p-3} \int_{ \Omega}\left|\Hnabla u(x)-\Hnabla v(x)\right|^2\big(
				\left|\Hnabla u(x)\right|+\left|\Hnabla v(x)\right|\big)^{p-2}\,dx.
		\end{split}\end{equation}
		
		Moreover, one has
		\begin{align}\label{ineq-with-holder}
			\nonumber\int_{ \Omega}&\left|\Hnabla u(x)-\Hnabla v(x)\right|^p\,dx\\ \nonumber
			&=\;\int_{ \Omega}\left|\Hnabla u(x)-\Hnabla v(x)\right|^p\big(\left|\Hnabla u(x)\right|+\left|\Hnabla v(x)\right|\big)^{\frac{p(p-2)}{2}}\big(
			\left|\Hnabla u(x)\right|+\left|\Hnabla v(x)\right|\big)^{-\frac{p(p-2)}{2}}\,dx\\ 
			&\le\;\left(\int_{ \Omega}\left|\Hnabla u(x)-\Hnabla v(x)\right|^{2}\big(
			\left|\Hnabla u(x)\right|+\left|\Hnabla v(x)\right|\big)^{p-2}\,dx\right)^{\frac{p}{2}}\\ \nonumber
                &\qquad\qquad\qquad \qquad \qquad \qquad\qquad\qquad \qquad \qquad\quad\cdot\left(\int_{ \Omega}\big(\left|\Hnabla u(x)\right|+\left|\Hnabla v(x)\right|\big)^p\,dx\right)^{1-\frac{p}{2}}. 
		\end{align}
            where in last inequality we apply the H\"older's inequality with H\"older exponent $2/p$ and conjugate exponent
		\[\left(\frac{2}{p}\right)'=\frac{2/p}{2/p-1}=\frac{2}{2-p}.\]
	Then using estimates \eqref{lower-bound-integral-nabla-u-nabla-v-1-<-p-leq-2-1} for the first term of the right hand side of \eqref{ineq-with-holder}, we obtain
		\begin{align*}
			\int_{ \Omega}&\left|\Hnabla u(x)-\Hnabla v(x)\right|^p\,dx\\
			&\qquad \le \; C\left(\int_{ \Omega}\big(\left|\Hnabla u(x)\right|^p-\left|\Hnabla v(x)\right|^p\big)\, 
			dx\right)^{\frac{p}{2}}\left(\int_{ \Omega}\big(\left|\Hnabla u(x)\right|+\left|\Hnabla v(x)\right|\big)^p\,dx\right)^{1-\frac{p}{2}},
		\end{align*}
		for some suitable constant $C>0$, depending only on $\G$ and $p$. This proves concludes the proof of \eqref{first-ineq-lemma-p-harm-repl} when $1<p<2$.

        \noindent{\em Step 2: the case $p\geq 2$.} Applying inequalities \eqref{stime-p-2} with the choices $\bxi:=\Hnabla u$ and $\bzeta:=\Hnabla v$ and using estimate \eqref{lower-bound-integral-nabla-u-nabla-v-1}, we obtain
      \begin{equation*}
        \begin{split}
				\int_{ \Omega}\Big(\left|\Hnabla u(x)\right|^p-\left|\Hnabla v(x)\right|^p\Big)\,dx
                    &\ge\;p\,\gamma\int_0^1\frac{1}{s}\left(\int_{ \Omega\cap A}\left|\Hnabla u_s(x)-\Hnabla v(x)\right|^p \,dx\right)\,ds \\
                    &=p\,\gamma\int_0^1 {s}^{p-1}\left(\int_{ \Omega\cap A}\left|\Hnabla u(x)-\Hnabla v(x)\right|^p \,dx\right)\,ds,\\
                    &=p\,\gamma\int_0^1 {s}^{p-1}\left(\int_{ \Omega}\left|\Hnabla u(x)-\Hnabla v(x)\right|^p \,dx\right)\,ds,\\
                    &=\gamma \int_{ \Omega}\left|\Hnabla u(x)-\Hnabla v(x)\right|^p\,dx,
            \end{split}\end{equation*}        
	where in the first equality we used the identity in \eqref{difference-u-^-s-v}. This establishes \eqref{second-ineq-lemma-p-harm-repl} when $p\ge2$ and completes the proof of Lemma \ref{lemma-p-harm-replac}.
	\end{proof}

\section{A Dichotomy result}\label{sec:dichotomy}

As a first step, we a prove a \emph{dichotomy} result. Roughly speaking, two situations can occur: either the average of the energy of an almost minimizer decreases in a smaller ball, or the distance of its horizontal gradient and a suitable constant horizontal section becomes as small as we wish. This implies that $\G$-affine functions are the only ones for which the average does not improve in small balls.

Our first result is the following:
\begin{prop}\label{dic} Let $u \in HW^{1,p}(B_1)$ be such that
	\begin{equation}\label{first}J_p(u,B_1) \leq (1+\sigma) J_p(v,B_1)\end{equation}
	for all $v \in HW^{1,p}(B_1)$ such that $u-v\in HW^{1,p}_0(B_1)$.
	Denoting by
	\begin{equation}\label{def_a}a : = \left(\fint_{B_1} |\Hnabla u(x)|^p dx\right)^{1/p},\end{equation}
	there exists $\e_0\in(0,1)$ such that for every $\varepsilon\in(0,\e_0)$ there exist $\eta\in(0,1)$,  $M\ge1$ and $\sigma_0\in(0,1)$, depending on $\e$, $Q$ and $p$, such that, if $\sigma\in[0, \sigma_0]$ and $a\ge M$, then the following dichotomy holds. Either
	\begin{equation}\label{alt1}\left(\fint_{B_{\eta}} |\Hnabla u(x)|^2 dx\right)^{1/p} \leq \frac a 2, \end{equation}
	or
	\begin{equation}\label{alt2}\left( \fint_{B_\eta} |\Hnabla u(x)- \mathbf{q}(x)|^p dx\right)^{1/p} \leq \varepsilon a,\end{equation}
	where $\mathbf{q}:\G \to H\G$ is a constant horizontal section, that is 
	\begin{equation}\label{def-q}
		\mathbf{q}(x):= \sum_{i=1}^{m_1}q_jX_j(x), \quad x\in \G
	\end{equation}
	for some suitable $q=(q_1,\ldots,q_{m_1})\in \R^{m_1}$, with 
	\begin{equation}\label{bound-q}
		\frac a 4 < |\mathbf{q}| \leq C_0 a,
	\end{equation}
	for some universal constant $C_0>0$.
\end{prop}

\begin{proof} We split the proof into several steps.\medskip
	
	\noindent{\em Step 1: pointwise estimates. }
	Let $v\in HW^{1,p}(B_1)$ be the $(\G,p)$-harmonic replacement of $u$ in $B_1$. By Theorem 1.1 in \cite{CM22}, it exists a constant $C_0>0$, depending only on $p$ and $Q$, such that
	\begin{equation}\label{CM1}
		\sup_{B_{1/2}}\left|\Hnabla v\right|\leq C_0 \left( \fint_{B_1} |\Hnabla v|^p dx\right)^{1/p}.
	\end{equation}  
	Then, recalling \eqref{def_a}, by minimality condition of $v$ in Dirichlet energy \eqref{p-g-minim} together with \eqref{CM1}, we conclude that for every $x\in B_{1/2}$
	\begin{equation}\label{stima-grad-v}
		\left|\Hnabla v(x)\right|\leq C_0 a.
	\end{equation}  
	\medskip
	
	\noindent{\em Step 2: oscillation estimates.}
	By Theorem 2.1 in \cite{CM22}, we have the following oscillation estimates: for all $\eta\in(0,1/2)$ 
	\begin{equation}\label{CM2}
		\max_{1\leq i\leq m_1} \mathrm{osc}_{B_\eta}X_i v \leq c \bigg(\frac{\eta}{1/2}\bigg)^{\alpha} \left( \fint_{B_1} |\Hnabla v|^p dy\right)^{1/p},
	\end{equation} 
	for some $\alpha=\alpha(Q,p)\in (0,1]$ and some constant $c>0$ depending only on $p$ and $\G$.
	Therefore, for every $x\in B_\eta$ we have 
	\begin{equation}
		\max_{1\leq i\leq m_1} |X_i v(x)-X_iv(0)|\leq \max_{1\leq i\leq m_1} \mathrm{osc}_{B_\eta}X_i v \leq c \bigg(\frac{\eta}{1/2}\bigg)^{\alpha}\left( \fint_{B_1} |\Hnabla v|^p dy\right)^{1/p},
	\end{equation}
	thus, for any $i=1,\dots, m_1$ denoting by $q_i=X_iv(0)$ and by $\mathbf{q}:\G\to H\G$ the constant horizontal section defined as in \eqref{def-q}, it turns out
	\begin{equation}
		|\Hnabla v(x)-\mathbf{q}(x)|\leq c_1\bigg(\frac{\eta}{1/2}\bigg)^{\alpha} \left( \fint_{B_1} |\Hnabla v|^p dy\right)^{1/p},
	\end{equation}
	for any $x\in B_\eta$ and for some constant $c_1>0$ depending only on $p$ and $\G$. This gives that, for all $\eta\in (0,1/2]$,
	\begin{equation}\label{average-integral-nabla-v-q}\begin{split}
			\fint_{B_\eta}\left|\Hnabla v(x)-\bq(x)\right|^p\,dx\le\;& \fint_{B_\eta}\bigg(c^2_1\bigg(\frac{\eta}{1/2}\bigg)^{p\alpha} \fint_{B_1} |\Hnabla v (y)|^p \,dy\,\bigg)\,dx\\
			=\;& C_1\,\eta^{p\alpha} \fint_{B_1} |\Hnabla v |^p dx\\
			\le\;& C_1\,\eta^{p\alpha}\,a^p,
	\end{split}\end{equation}
	for some $\alpha\in(0,1]$ and $C_1>0$ depending only on $p$ and $\G$.	\medskip
	
	\noindent{\em Step 3: proximity to the $(\G,p)$-harmonic replacement.}
		By hypothesis \eqref{first} and the minimality of $v$ in \eqref{p-g-minim}, for all $\sigma <|B_1|$, we obtain
	\begin{equation}\label{diff-norm-grad}\begin{split}
				\int_{B_1}\big(\left|\Hnabla u(x)\right|^p-\left|\Hnabla v(x)\right|^p\big)\,dx		
				&\le J_p(u,B_1)-\int_{B_1}\left|\Hnabla v(x)\right|^p\,dx \\
				&\le  (1+\sigma)J_p(v,B_1)-\int_{B_1}\left|\Hnabla v(x)\right|^p\,dx \\
				&\le C\left(\sigma\int_{B_1}\left|\Hnabla v(x)\right|^p\,dx+1\right)\\
				&\le C\left(\sigma\int_{B_1}\left|\Hnabla u(x)\right|^p\,dx+1\right)
		\end{split}\end{equation}
    for some suitable positive constant $C$ depending only on $Q$.
    
    Now we distinguish two cases, $p\ge2$ and $1<p<2$. 

 \medskip
	
	\noindent{\em Step 3.1: the case $p\ge2$.}	
		If $p\ge 2$, then from \eqref{second-ineq-lemma-p-harm-repl} and \eqref{diff-norm-grad}
		we deduce that
		\begin{align*}
			\int_{B_1}\left|\Hnabla u(x)-\Hnabla v(x)\right|^p\,dx			
			&\le C\left(\sigma\int_{B_1}\left|\Hnabla u(x)\right|^p\,dx+1\right)
		\end{align*}
		for some suitable constant $C>0$ depending only on $p$ and $\G$. Consequently taking the average over $B_1$, we thereby obtain that
		\begin{equation*}
			\fint_{B_1}\left|\Hnabla u(x)-\Hnabla v(x)\right|^p\,dx\le C(\sigma a^p+1).
		\end{equation*}
		Using this and \eqref{average-integral-nabla-v-q}, we get that
		\begin{equation}\label{average-integral-nabla-u-q-p->-2}\begin{split}
				\fint_{B_\eta}\left|\Hnabla u(x)-\bq(x)\right|^p\,dx &\le2^{p-1}
				\fint_{B_\eta}\big(\left|\Hnabla u(x)-\Hnabla v(x)\right|^p+\left|\Hnabla v(x)-\bq(x)\right|^p\big)\,dx\\
				&\le 2^{p-1}\left(\frac{\left|B_1\right|}{\left|B_\eta\right|}\fint_{B_1}\left|\Hnabla u(x)-\Hnabla v(x)\right|^p\,dx+C_1\,\eta^{p\alpha }\,a^p\right)
				\\&\le 2^{p-1} \Big(C\eta^{-Q}(\sigma a^p+1)+C_1a^p\eta^{p\alpha}\Big),\\
				&= 2^{p-1}C \eta^{-Q}\sigma a^p+2^{p-1}C\eta^{-Q}+2^{p-1}C_1a^p\eta^{p\alpha}.
			\end{split}
		\end{equation}
		This yields that
		\begin{equation}\label{average-integral-nabla-u-to-p-B-eta-p->-2}
			\fint_{B_\eta}\left|\Hnabla u(x)\right|^p\,dx\le 2^{2(p-1)}C\eta^{-Q}\sigma a^p+2^{2(p-1)}C\eta^{-Q}+2^{2(p-1)}C_1a^p\eta^{\alpha p}+2^{p-1}\left|\bq\right|^p.
		\end{equation}
		 \medskip
	
	\noindent{\em Step 3.2: the case $1<p<2$.}		
		If $1<p< 2$, by virtue of \eqref{first-ineq-lemma-p-harm-repl}
		and \eqref{diff-norm-grad}, we obtain that
		\begin{align*}
			&\int_{B_1}\left|\Hnabla u(x)-\Hnabla v(x)\right|^p\,dx\\
			&\qquad\le\; C\left(\sigma
			\int_{B_1}\left|\Hnabla u(x)\right|^p\,dx+1\right)^{\frac{p}{2}}
			\left(\int_{B_1}\big(\left|\Hnabla u(x)\right|+\left|\Hnabla v(x)\right|\big)^p\,dx\right)^{1-\frac{p}{2}}\\
			&\qquad\le\; C \left(\sigma^{\frac{p}{2}}\left(\int_{B_1}\left|\Hnabla u(x)\right|^p\,dx\right)^{\frac{p}{2}}+1\right)\left(\int_{B_1}\big(\left|\Hnabla u(x)\right|^p+\left|\Hnabla v(x)\right|^p\big)\,dx\right)^{1-\frac{p}{2}}.
		\end{align*}
		Thus, since $v$ is the $(\G,p)$-harmonic replacement of $u$ in $B_1$,
		\begin{align*}
			&\int_{B_1}\left|\Hnabla u(x)-\Hnabla v(x)\right|^p\,dx\\
			&\qquad\le\; C\left(\sigma^{\frac{p}{2}}\left(\int_{B_1}\left|\Hnabla u(x)\right|^p\,dx\right)^{\frac{p}{2}}+1\right)
			\left(\int_{B_1}\left|\Hnabla u(x)\right|^p\,dx\right)^{1-\frac{p}{2}}\\
			&\qquad =\; C \left(\sigma^{\frac{p}{2}}\int_{B_1}\left|\Hnabla u(x)\right|^p\,dx+\left(\int_{B_1}\left|\Hnabla u(x)\right|^p\,dx\right)^{1-\frac{p}{2}}\right).
		\end{align*}
		Consequently, taking the average integral, we have that
		\begin{equation*}
			\fint_{B_1}\left|\Hnabla u(x)-\Hnabla v(x)\right|^p\,dx\le C\left(
			\sigma^{\frac{p}{2}}a^p+|B_1|^{-\frac{p}{2}}a^{p\left(1-\frac{p}{2}\right)}\right)
			\le C\left(\sigma^{\frac{p}{2}}a^p+a^{p\left(1-\frac{p}{2}\right)}\right).
		\end{equation*}
		{F}rom this and \eqref{average-integral-nabla-v-q}, we obtain that
		\begin{equation}\label{average-integral-nabla-u-q-1-<p-leq-2}\begin{split}
				\fint_{B_\eta}\left|\Hnabla u(x)-\bq\right|^p\,dx\le\;& 2^{p-1}
				\fint_{B_\eta}\big(\left|\Hnabla u(x)-\Hnabla v(x)\right|^p+\left|\Hnabla v(x)-\bq\right|^p\big)\,dx\\
				\le\;& 2^{p-1}\left(\frac{\left|B_1\right|}{\left|B_\eta\right|}\fint_{B_1}\left|\Hnabla u(x)-\Hnabla v(x)\right|^p\,dx+C_1\,\eta^{\alpha p}\,a^p\right)\\
				\le \;&
				2^{p-1}C\eta^{-Q}\sigma^{\frac{p}{2}}a^p+2^{p-1}C\eta^{-Q}a^{p\left(1-\frac{p}{2}\right)}
				+2^{p-1}C_1a^p\eta^{\alpha p}.
		\end{split}\end{equation}
		This gives that
		\begin{equation}\label{average-integral-nabla-u-to-p-B-eta-1<-p-leq-2}
			\fint_{B_\eta}\left|\Hnabla u(x)\right|^p\,dx\le 2^{2(p-1)}C\eta^{-Q}\sigma^{\frac{p}{2}} a^p+2^{2(p-1)}C\eta^{-Q}
			a^{p\left(1-\frac{p}{2}\right)}+2^{2(p-1)}C_1a^p\eta^{\alpha p}+2^{p-1}\left|\bq\right|^p.
		\end{equation}
		
		  \medskip
	
	\noindent{\em Step 4: perturbative estimates.}		
		Now, 
		given $\varepsilon_0\in(0, 1/4]$, we claim that 
		for every $\varepsilon\in (0,\epsilon_0)$ there exists $\eta$ small enough (depending on $\varepsilon$)
		such that if $\sigma$ is chosen sufficiently small and $a$ sufficiently large (depending on $\eta$,
		and thus on $\e$) then
		\begin{equation}\label{conditions-on-eta-sigma-a-p>2}
				2^{2(p-1)}C\eta^{-Q}\sigma a^p+2^{2(p-1)}C\eta^{-Q}+2^{2(p-1)}C_1a^p\eta^{\alpha p}\le 2^{p-1}\varepsilon^pa^p\le \frac{a^p}{2^{p+1}} \quad
				\mbox{if }p\ge 2,\\
		\end{equation}	
        while
            \begin{equation}\label{conditions-on-eta-sigma-a-1<p<2}				2^{2(p-1)}C\eta^{-Q}\sigma^{\frac{p}{2}} a^p+2^{2(p-1)}C\eta^{-Q}
				a^{p\left(1-\frac{p}{2}\right)}+2^{2(p-1)}C_1a^p\eta^{\alpha p}\le 2^{p-1}\varepsilon^pa^p\le \frac{a^p}{2^{p+1}} \quad 
				\mbox{if }1<p< 2.
		\end{equation}
To prove this we distinguish two cases.  \medskip
	
	\noindent{\em Step 4.1: the case $p\ge2$.}	
If $p\ge 2,$ we pick $\eta>0$ sufficiently small
such that $$\varepsilon^p-2^{p-1}C\eta> 2^{p-1}C_1\eta^{\alpha p}.$$ This allows us to define
		$$
		M:=\left(\frac{2^{p-1}C\eta^{-Q}}{\varepsilon^p-2^{p-1}C\eta-2^{p-1}C_1\eta^{\alpha p}}\right)^{1/p}
		.$$
		Note also that we can suppose $M\ge1$ by taking $\eta$ small enough.
		Let also 
		$$ \sigma_0:=\min\{|B_1|,\eta^{Q+1}\}.$$ With this setting, we obtain that, for every $a\geq M$ and for every $0<\sigma\leq \sigma_0,$
		\begin{eqnarray*}&&
2^{2(p-1)}C\eta^{-Q}\sigma a^p+ 2^{2(p-1)}C\eta^{-Q}+ 2^{2(p-1)}C_1a^p\eta^{\alpha p}\\&\leq
			& a^p\Big(  2^{2(p-1)}C\eta+ 2^{2(p-1)}C_1\eta^{\alpha p}\Big)
			+ 2^{2(p-1)}C\eta^{-Q}\\&=&
	a^p\Big(  2^{2(p-1)}C\eta+ 2^{2(p-1)}C_1\eta^{\alpha p}\Big)
			+2^{p-1}\,M^p\Big(\varepsilon^p-2^{p-1}C\eta-2^{p-1}C_1\eta^{\alpha p}\Big)
			\\
			&\leq& a^p\Big(  2^{2(p-1)}C\eta+ 2^{2(p-1)}C_1\eta^{\alpha p}\Big)
			+2^{p-1}\,a^p\Big(\varepsilon^p-2^{p-1}C\eta-2^{p-1}C_1\eta^{\alpha p}\Big)\\&=&
			2^{p-1}\varepsilon^p a^p+ a^p
			\Big(  2^{2(p-1)}C\eta+ 2^{2(p-1)}C_1\eta^{\alpha p}
			- 2^{2(p-1)} C\eta- 2^{2(p-1)}C_1\eta^{\alpha p}\Big)\\&=&2^{p-1}\varepsilon^p a^p,
		\end{eqnarray*}
		which proves \eqref{conditions-on-eta-sigma-a-p>2}.\medskip
	
	\noindent{\em Step 4.2: the case $1<p<2$.}		
		If instead $1<p<2,$ we pick $\eta>0$ small enough such that $$\varepsilon^p>2^{p-1}C\eta+2^{p-1}C_1\eta^{\alpha p}.$$ In this way, we can define
		$$M:=
		\left(\frac{C2^{p-1}\eta^{-Q}}{\varepsilon^p-2^{p-1}C\eta-2^{p-1}C_1\eta^{\alpha p}}\right)^{2/p^2}.$$
		Let also $\sigma_0:=\eta^{(Q+1)\frac{2}{p}}.$
Then, whenever $a\geq M$ and $0<\sigma\leq \sigma_0$ it follows that
		\begin{eqnarray*}
			&&\hspace{-5em} 2^{2(p-1)}C\eta^{-Q}\sigma^{\frac{p}{2}}a^p+ 2^{2(p-1)}C\eta^{-Q} a^{p\left(1-\frac{p}{2}\right)}+ 2^{2(p-1)}C_1a^p\eta^{\alpha p}\\&\leq& 2^{2(p-1)}C\eta a^p+ 2^{2(p-1)}C\eta^{-Q} a^{p\left(1-\frac{p}{2}\right)
			}+ 2^{2(p-1)}C_1a^p\eta^{\alpha p}\\&=&
			2^{p-1} a^p\Big( 2^{p-1}C\eta + 2^{p-1}C\eta^{-Q} a^{-\frac{p^2}{2}	}+ 2^{p-1}C_1\eta^{\alpha p}\Big)\\&\leq&
2^{p-1} a^p\Big( 2^{p-1}C\eta + 2^{p-1}C\eta^{-Q} M^{-\frac{p^2}{2}	}+ 2^{p-1}C_1\eta^{\alpha p}\Big)\\&=&
			2^{p-1} a^p\Big( 2^{p-1}C\eta + \varepsilon^p- 2^{p-1}C\eta- 2^{p-1}C_1\eta^{\alpha p}
			+ 2^{p-1}C_1\eta^{\alpha p}\Big)\\&=&  2^{p-1}\e^pa^p,
		\end{eqnarray*}
		which establishes \eqref{conditions-on-eta-sigma-a-1<p<2}.
		  \medskip
	
	\noindent{\em Step 5: conclusion of the proof.}		
		In order to complete the proof of Proposition \ref{dic},
		we now distinguish two cases according to the size of $\left|\bq\right|.$ More precisely, we first suppose that
		\[\left|\bq\right|\le\frac{a}{4}.\]
		Then, using \eqref{average-integral-nabla-u-to-p-B-eta-p->-2} and \eqref{conditions-on-eta-sigma-a-p>2} if $p\ge2$ or \eqref{average-integral-nabla-u-to-p-B-eta-1<-p-leq-2} and \eqref{conditions-on-eta-sigma-a-p>2} if $1<p<2$, we conclude that
		\begin{align*}
			\fint_{B_\eta}\left|\Hnabla u(x)\right|^p\,dx\le \frac{a^p}{2^{p+1}}+ 2^{p-1}\frac{a^p}{2^{2p}}=\frac{a^p}{2^{p+1}}+\frac{a^p}{2^{p+1}}=\frac{a^p}{ 2^{p}},
		\end{align*}
		and thus
		\[\left(\fint_{B_\eta}\left|\Hnabla u(x)\right|^p\,dx\right)^{1/p}\le\frac{a}{2},\]
		which is the first alternative in \eqref{alt1}.
		
		Otherwise, it holds that
		\[ \frac{a}{4}<\left|\bq\right|\le C_0\,a,\]
		and therefore, by either \eqref{average-integral-nabla-u-q-p->-2} and \eqref{conditions-on-eta-sigma-a-p>2} if $p\ge2$ or \eqref{average-integral-nabla-u-q-1-<p-leq-2} and \eqref{conditions-on-eta-sigma-a-1<p<2} if $1<p<2$, we conclude that
		\[\left(\fint_{B_\eta}\left|\Hnabla u(x)-\bq(x)\right|^p\,dx\right)^{1/p}\le \varepsilon a,\]
		which is the second alternative in \eqref{alt2}.
		The proof of Proposition \ref{dic} is thereby complete.
	\end{proof}	

 \section{Regularity estimates for a subelliptic equation with variable coefficients}\label{sec:reg} 
 Now, our aim is to prove  that the alternative in \eqref{alt2}
 can be \emph{improved} when $\varepsilon$ and $\sigma$ are sufficiently small.
 On the other hand, differently from what happens when $p=2$ (see Lemma 3.2 in \cite{FFM}), here the main difficulty relies on the fact that the problem is not linear when $p\neq2$. Therefore, for example, if $v_1$ and $v_2$ are the $(\G,p)$-harmonic replacements of $u_1$
	and $u_2$ in $\Omega\subset\G$, then it is not true that $v_1+v_2$
	is the $(\G,p)$-harmonic replacement of $u_1+u_2$, unless $p=2$. 
	
	To overcome this difficulty, our goal is to show that affine perturbation of $(\G,p)$-harmonic replacements satisfied a suitable subelliptic equation. With this purpose, in this section we prove some $C^{1,\alpha}$ regularity estimate that will we useful to obtain the counterpart of Lemma 3.2 in \cite{FFM} in our nonlinear setting.
 
	More precisely, we will show the following result.
	
	\begin{thm}\label{thm:c1au}
		Let $\G$ be a Carnot group of step 2 and let $\Omega\subset\G$ be an open set. Let $u\in HW^{1,2}(\Omega)$, be a weak solution of 
        \begin{equation}\label{eq-div-var-coeff}
                \mathrm{div}_\G \big(A(x)\Hnabla u(x))=0, \quad \text{for } x\in\Omega,
            \end{equation}
        where $A:\Omega \to \R^{m_1\times m_1}$ satisfies the following structure condition, for any $x,y\in\Omega$ 
        \begin{eqnarray}
        \label{st.cond1}
            &\nu |\xi|^2\leq \left\langle A(x)\xi,\xi\right\rangle \leq L|\xi|^2, \quad \text{for all }\xi\in H\G_x \\ \vspace{5em}
        \label{st.cond2}
            &\hspace{-4em}|A(x)-A(y)|\le L'd(x,y)^\alpha
        \end{eqnarray}
        for some $\nu>0$ and $L,L'\geq 1$ and $\alpha\in (0,1]$. Then $\Hnabla u\in C^{0,\gamma}_{loc}(\Omega)$, for some $\gamma=\gamma(\G,\nu,L,\alpha)\in(0,1)$. Moreover, there exists a constant $c=c(\G,\nu,L)>0$ such that for any $x_0\in\Omega$ it exists $\bar R=\bar R(\G,\nu,L,L',\alpha, \mathrm{dist}(x_0,\partial\Omega))>0$ such that for any $x,y\in B_{R}(x_0)\subset\Omega$, with $0<R\leq \bar R$ it follows
\begin{equation}\label{eq:teo-nuovo}
\max_{1\leq i\leq m_1}|X_i(x)-X_i(y)|\leq c\,d_c(x,y)^\gamma \fint_{B_R(x_0)}|\Hnabla u(x)|\, dx.
\end{equation}
\end{thm}

\begin{remark}
The regularity and apriori estimates of the homogeneous equation corresponding to \eqref{eq-div-var-coeff} with 
freezing of the coefficients is a key tool. Fixed $x_0\in\Omega$, we consider the equation 
\begin{equation}\label{eq:frzeq}
\divg (A(x_0)\Hnabla u) = 0 \quad\text{in}\ \Omega. 
\end{equation}
The $C^{1,\alpha}$ regularity of sub-elliptic equation of the form \eqref{eq:frzeq} within Carnot Groups of step two, under assumption \eqref{st.cond1} and \eqref{st.cond2}, has been dealt with in \cite{CM22} (see also \cite{MZ} in the case of Heisenberg group). In particular, in Theorem 1.1 in \cite{CM22}, similar as the Euclidean case (see \cites{DiB,Man86}), the following local estimate can be shown by using Sobolev's inequality and
Moser's iteration on the Caccioppoli type inequalities: for any $0<\sigma<1$ and metric ball $B_R=B_R(x_0)\subset \Omega$, it results
\begin{equation}\label{eq:CM22p-2}
\sup_{B_{\sigma R}}\ |\Hnabla u| \leq c (1-\sigma)^{-\frac{Q}{2}} \bigg(\fint_{B_R}|\Hnabla u(x)|^2\,dx\bigg)^\frac{1}{2}
\end{equation}
for some $c=c(\G,\nu,L)>0$. Furthermore, by another standard iteration argument proved in Lemma 3.38. in \cite{M-book}, for any $0<q<2$, $0<\sigma<1$ and metric ball $B_R=B_R(x_0)\subset \Omega$, we obtain the following estimate
\begin{equation}\label{cm-q=1}
\sup_{B_{\sigma R}}\ |\Hnabla u| \leq c (1-\sigma)^{-\frac{Q}{q}} \bigg(\fint_{B_R}|\Hnabla u(x)|^q\,dx\bigg)^\frac{1}{q}
\end{equation}
for some $c=c(\G,\nu,L,q)>0$.
\end{remark} 

We recall the notion of De Giorgi's class of functions in this setting, which would be required for Proposition \ref{prop:intosc0}.

\begin{defin}[De Giorgi's class]
For any $x\in\G$, and any metric ball $B_{\rho_0}(x)\subset \G$, the De Giorgi's class $DG^+(B_{\rho_0})$ consists of functions $w\in HW^{1,2}(B_{\rho_0}(x))\cap L^\infty(B_{\rho_0}(x))$, which satisfy the inequality 
\begin{equation}\label{eq:DG}
\int_{B_{\rho'}(x)}|\Hnabla (w(x)-k)^+|^2\, dx\le
\frac{\gamma}{(\rho-\rho')^2}
\int_{B_\rho(x)}|(w(x)-k)^+|^2\, dx+\chi^2|
A_{k,\rho}^+(w)|^{1-\frac{2}{Q}+\epsilon}
\end{equation}
for some $\gamma, \chi,\epsilon>0$, 
where $ A_{k,\rho}^+(w)=\{ x\in B_\rho: (w(x)-k)^+=\max (w(x)-k,0)>0\}$ for any arbitrary $k\in\R$, and  $0<\rho'<\rho\le \rho_0$. The class $DG^-(B_{\rho_0}(x))$ are similarly defined replacing $(w-k)^+$ with $(w-k)^-$ in \eqref{eq:DG}. We set $DG(B_{\rho_0}(x)):=DG^+(B_{\rho_0}(x))\cap DG^-(B_{\rho_0}(x))$.   
\end{defin}

\subsection{Comparison estimates}
In this subsection, we prove some comparison estimates that will be useful in the proof of Theorem \ref{thm:c1au}. We follow closely the approach of \cite{DM} in the Euclidean case, and of \cite{SS} for the Heisenberg group. From here on out, $\G$ will a Carnot group of step two and we denote $u\in HW^{1,2}(\Omega)$ as a weak solution of \eqref{eq-div-var-coeff}.

As a first step, relying on the regularity results in \cite{CM22}, we prove an integral oscillation decay estimate of the solution of the equation corresponding to \eqref{eq-div-var-coeff} with freezing coefficient, see \eqref{eq-froz} below.

\begin{prop}\label{prop:intosc0}
        Let $x_0\in\Omega$, and $r_0>0$ such that $=B_{r_0}(x_0) \subset \Omega$. Let $u\in C^{1,\sigma}(\Omega)$  for some $\sigma\in (0,1]$ be a weak solution of the equation corresponding to \eqref{eq-div-var-coeff} with \emph{freezing} coefficients in $x_0$, i.e. 
        \begin{equation}\label{eq-froz}
            \divg(A(x_0)\Hnabla u(x))=0, \quad \text{for }x\in \Omega.
        \end{equation}
        Then there exist $c=c(\G,L)>0$, and $\beta=\beta(\G)>0$ and such that for all $0<\varrho<r<r_0$, we have 
        \begin{equation}\label{int-osc-decay-freez-eq}
            \fint_{B_\varrho(x_0)}|\Hnabla u(x) - (\Hnabla u)_{B_\varrho(x_0)}|\,dx \leq c\left(\frac{\varrho}{r}\right)^\beta \Big(\fint_{B_r(x_0)}|\Hnabla u(x) - (\Hnabla u)_{B_r(x_0)}|\,dx + \chi r^\beta\Big)
        \end{equation}
        with $\chi:= M(r_0)/r_0^\beta$, where $$M(r_0)=\max_{1\le i\le m_1}\sup_{B_{r_0}(x_0)} |X_i u|.$$
        \end{prop}
        \begin{proof}
        Since in the proof we will consider all concentric balls centred in $x_0\in\G$, for the seek of brevity we will denote $B\varrho=B_\varrho(x_0)$ for every $\varrho>0$. Moreover let us denote by 
        \begin{equation}\label{def-M-omega-I}
            M(\varrho):=\max_{1\le i\le m_1}\sup_{B_{\varrho}} |X_i u|, \quad \omega(\varrho):=\max_{1\le i\le m_1}\osc_{B_{\varrho}} X_i u \quad\text{and}\quad I(\varrho) = \fint_{B_\varrho}|\Hnabla u - (\Hnabla u)_{B_\varrho}|\,dx
        \end{equation}
        for every $0<\varrho<r_0$. By definition follows trivially that 
        \begin{equation}\label{prop-omega-m}
            \omega(\varrho)=\max_{1\le i\le m_1}\osc_{B_{\varrho}}X_iu=\max_{1\le i\le m_1}\sup_{B_{\varrho}}X_iu-\inf_{B_{\varrho}}X_iu\leq \max_{1\le i\le m_1}2 \sup_{B_{\varrho}}X_iu=2M(\varrho).
        \end{equation}
      Now, since the operator appearing in \eqref{int-osc-decay-freez-eq} has constant coefficients, the equation \eqref{int-osc-decay-freez-eq} has been dealt with in those in \cite{CM22}. Thus applying the decay oscillation Lemma proved in \cite{CM22}*{Lemma 5.9} we know there there exists a constant $s\geq 0$ depending only on $\G$, such that for every $0<r\leq r_0/16$, we have 
        \begin{equation}\label{lemm5-6-CM}
            \omega(r) \leq (1-2^{-s})\,\omega(8r) + 2^s M({r_0})\Big(\frac{r}{{r_0}}\Big)^{\frac 1 2},
        \end{equation}
        for all $\beta=\beta\in(0,1/2]$. Up to take $\beta\in(0,1/2]$ sufficiently small, since $\omega=\omega(\varrho)$ is a positive and increasing function for every $\varrho>0$, we are in position to apply the standard iteration scheme proved in \cite{G}*{Lemma 7.3} on \eqref{lemm5-6-CM}, which implies that for any $0<\varrho<r\leq r_0$, we get
        \begin{equation}\label{3-4-M-S}
        \begin{aligned}
            \omega(\varrho)\leq c\Big(\left(\frac{\varrho}{r}\right)^\beta\omega(r)+\chi\varrho^\beta\Big)=c\left(\frac{\varrho}{r}\right)^\beta \big(\omega(r)+\chi r^\beta\big)
        \end{aligned}
        \end{equation}
        whit $\chi= M(r_0)/r_0^{\beta}$ and some constant $c>0$ depending only on $\G$. We notice that if $\varrho \leq \delta r$ for some $\delta\in (0,1)$, since, up to some constant $c>0$ depending only on $\G$, we have
        \begin{equation*}
            I(\varrho)=\fint_{B_\varrho}|\Hnabla u(x) - (\Hnabla u)_{B_\varrho}|\,dx \leq\fint_{B_\varrho}\fint_{B_\varrho}|\Hnabla u(x) - \Hnabla u(y)| \,dx\,dy \leq c\, \omega(\varrho)
        \end{equation*}
        we can conclude from \eqref{3-4-M-S} that 
        \begin{equation}\label{3-5-M-S}
            I(\varrho) \leq c\,\omega(\varrho) \leq 
            c\,\delta^{-\beta}\left(\frac{\varrho}{r}\right)^\beta\big(\omega(\delta r)+\chi r^\beta\big).
        \end{equation}
        for some $c>0$ depending only on $\G$. Now we claim that, to conclude the proof it is enough to prove that there exists $\delta\in (0,1)$, depending only on $\G$ such that the inequality  
        \begin{equation}\label{claim-m-s}
            \omega(\delta r) \leq c (I(r) +\chi r^{\beta})
        \end{equation}
    holds for some universal $c>0$. Indeed combing \eqref{claim-m-s} with \eqref{3-5-M-S}, we get \eqref{int-osc-decay-freez-eq}. 

    To prove the claim in \eqref{claim-m-s}, let us fix $r':=\delta r$, where $\delta\in(0,1)$ is to be chosen precisely later on. At first, we notice that there is no loss of generality by assuming that 
    \begin{equation}\label{assMS}
      \omega(r)\geq M(r_0) \left(\frac r {r_0}\right)^{\beta}, 
    \end{equation}
    since, otherwise \eqref{claim-m-s} trivially holds. Similarly, as in \cites{CM22,SS,DM} we consider the following complementary cases. \medskip
	
\noindent{\em Case 1}: For at least one index $l\in\{1,\ldots,m_1\}$, we have 
\begin{equation*}
  \text{either} \quad \Big|B_{4r'}\cap\Big\{X_lu<\frac{M(4r')}{4}\Big\}\Big|
 \leq \theta |B_{4r'}| \ \ \text{or}\ \ \Big|B_{4r'}\cap\Big\{ X_lu>-\frac{M(4r')}{4}\Big\}\Big|
  \leq \theta |B_{4r'}|. 
 \end{equation*}
where in $\theta>0$ is the universal constant in \cite{CM22}*{Corollary 5.6}. Similarly to the proof of Case 1 in the proof of \cite{CM22}*{Lemma 5.9}, combing \cite{CM22}*{Corollary 5.6} and \cite{CM22}*{Lemma 5.7} we conclude that for any $j=1,\ldots,m_1$, $X_ju$ belong to the De Giorgi class $DG^+(B_{2r'})$. Now by replacing $(X_lu-k)^+$ with $(X_lu-k)^-$ and
 $A^+_{k,r}(X_lu)$ with $A^-_{k,r}(X_lu)$, we car apply again \cite{CM22}*{Lemma 5.7} to conclude that for any $j=1,\ldots,m_1$, so $X_ju$ belong to the De Giorgi class $DG^-(B_{2r'})$ as well, $X_ju\in DG(B_{2r'})$ for every $j=1,\ldots,m_1$. Thus, in particular, the following local boundedness estimates hold  
\begin{align}
\label{supdg}
&\sup_{B_{r'}}(X_ju-\vartheta) \leq c\Big(\fint_{B_{2r'}}(X_ju-\vartheta)^+\,dx +\chi r'^\beta\Big), \\
\label{infdg}
&\sup_{B_{r'}}(\vartheta-X_ju) \leq c\Big(\fint_{B_{2r'}}(\vartheta-X_ju(x))^+\,dx +\chi r'^\beta\Big),
\end{align}
for any $\vartheta < M(r')$ and $j=1,\ldots,m_1$. Now, adding \eqref{supdg} and \eqref{infdg} with $\vartheta = (X_iu)_{B_{r'}}$, we get 
$$ \osc_{B_{r'}}X_ju\leq c\Big(\fint_{B_{2r'}}|X_ju(x)-(X_ju)_{B_{r'}}|\,dx +\chi r'^\beta\Big)\leq c\big(I(r)+\chi r^\beta\big)$$
for some $c=c(\G)>0$ and $\delta<1/2$, which completes the proof fo claim \eqref{claim-m-s} in this case. \medskip
	
\noindent{\em Case 2}: If Case 1 does not occur, for every $j=1,\ldots,m_1$, we have 
\begin{equation}\label{case-2-m-s}
  \Big|B_{4r'}\cap\Big\{X_ju<\frac{M(4r')}{4}\Big\}\Big|
 > \theta |B_{4r'}| \ \ \text{and}\ \ \Big|B_{4r'}\cap\Big\{ X_ju>-\frac{M(4r')}{4}\Big\}\Big|
  > \theta |B_{4r'}|. 
 \end{equation}
with $\theta>0$ as in Case 1.  By \eqref{case-2-m-s} we have 
\begin{equation}\label{sup-inf-m-s}
\inf_{B_{4r'}}X_ju\leq M(4r')/4 \quad \text{and}\quad \sup_{B_{4r'}}X_iu\geq -M(4r')/4
\end{equation}
for every $j=1,\ldots,m_1$, thus, by definition, we get  
\begin{equation}\label{ocsup}
\omega(4r')\geq M(4r')-M(4r')/4=3M(4r')/4.
\end{equation}
Now, let us suppose that $$K:=\max_{1\leq j\leq m_1}|(X_ju)_{B_r}| = |(X_ku)_{B_r}| \quad \text{for some}\quad k\in\{1,\ldots,m_1\}.$$ Then we observe that, if $K>2\omega(4r')$, by \eqref{ocsup} and \eqref{sup-inf-m-s},  for $x\in B_{4r'}$ we have
\begin{equation}\label{3-9-bis-m-s}
    \begin{split}
    |(X_ku)_{B_r}|-|X_ku(x)|\geq K-\inf_{B_{4r'}}X_ju >2\omega(4r')-\frac{M(4r')}{4} \geq  \frac{3M(4r')}{2}-\frac{M(4r')}/4 \geq \frac{M(4r')}{2}
    \end{split}
\end{equation}
Finally, choosing of $\delta<1/4$, \eqref{prop-omega-m} and \eqref{3-9-bis-m-s} imply
\begin{equation}\label{c21}
I(r)+\chi r^\beta\geq I(r)\geq c\fint_{B_{4r'}}|X_k u(x)-(X_ku)_{B_r}|\,dx \geq 
\frac{c}{2}\,M(4r')\geq \frac{c}{4}\,\omega(4r')\geq \frac{c}{4}\,\omega(r')
\end{equation}
for some constant $c>0$ depending only on $\G$. This implies the claim \eqref{claim-m-s}, when $K>2\omega(4r')$.

Let's move on to dealing with the case when $K\leq 2\omega(4r')=2\omega(4\delta r)$. Choosing $\delta<1/8$ and $q=1$ in \eqref{cm-q=1}, we conclude
\begin{equation}\label{3-10-bis-m-s}
   \omega(r/2)\leq 2M(r/2)(r/2)\leq c\fint_{B_r}|\Hnabla u(x)|\,dx
\end{equation}   
for some universal constant $c>0$ depending only on $\G$. Now, using \eqref{3-4-M-S}, \eqref{3-10-bis-m-s} and the fact that $K\leq 2\omega(4r')=2\omega(4\delta r)$ we obtain
\begin{equation}\label{llest1}
\begin{aligned}
\omega(4\delta r)&\leq c(8\delta)^\beta (\omega(r/2)+\chi r^\beta)\leq 
c\delta^\beta\Big(\fint_{B_r}|\Hnabla u(x)|\,dx+\chi r^\beta\Big) \\
&\leq c_1\delta^\beta( I(r)+ L+\chi r^\beta)
\leq c_1\delta^\beta( I(r)+ 2\omega(4\delta r)+\chi r^\beta)
\end{aligned}
\end{equation}
for some $c_1)>0$, depending only on $\G$ 
Now up to choosing $\delta$ sufficiently small, such that 
$2c_1\delta^\beta<1$, by \eqref{llest1} we get
\begin{equation}\label{c22}
\omega(4\delta r)\leq \frac{c_1\delta^\beta}{1-2c_1\delta^\beta}
\big(I(r)+\chi r^\beta\big).
\end{equation}
that prove the claim \eqref{claim-m-s} in Case 2, as well. This completes the proof of Proposition \ref{prop:intosc0}.
\end{proof}

Now, we prove some comparison estimate concerning the solution of Dirichlet problem with freezing of the coefficients on metric balls.

Let us fix $x_0\in\Omega$ and $R>0$, such that $B_{2R}(x_0)\subset\Omega$, and $u\in HW^{1,2}(B_{2R}(x_0))$ a weak solution of \eqref{eq-div-var-coeff}. Let us consider the Dirichlet problem
\begin{equation}\label{eq:frzdir}
 \begin{cases}
  \divg (A(x_0)\Hnabla v)=  0\ \ \text{in}\ B_R(x_0),\vspace{0.3em}\\
 \ \hspace{2em} v-u\in HW^{1,2}_0(B_R(x_0)). 
 \end{cases}
\end{equation}
\begin{lem}\label{lem:frzcomp}
Let $x_0\in\Omega$, and $R>0$ such that $B_{2R}=B_{2R}(x_0) \subset \Omega$. Let $u\in HW^{1,2}(B_{2R})$ be weak solution of \eqref{eq-div-var-coeff} and $v\in HW^{1,2}(B_R)$ a solution of Dirichlet problem \eqref{eq:frzdir}. Then there exists a constant $c=c(\G,\nu,L)>0$ such that 
\begin{equation}\label{eq:frzcomp}
\fint_{B_R}|\Hnabla v(x)-\Hnabla u(x)|^2\, dx\leq cL'R^{\alpha}\fint_{B_R}|\Hnabla u(x)|^2\, dx.
\end{equation}
\end{lem}
\begin{proof}	
As in the previous proof, for all concentric balls centred in $x_0\in\G$, we will denote $B\varrho=B_\varrho(x_0)$ for every $\varrho>0$. First of all, we noticed that from \eqref{eq:frzdir}, $v$ satisfies the variational problem
\begin{equation}\label{v-g-minim}
    \int_{B_R}\left\langle A(x)\Hnabla v(x),\Hnabla v(x)\right\rangle\,dx=\min_{u-w\in HW^{1,p}_0(B_R)}\int_{B_R}\left\langle A(x)\Hnabla w(x),\Hnabla w(x)\right\rangle\,dx.
\end{equation}
So by minimality of $v$ in \eqref{v-g-minim} and using \eqref{st.cond1}, we obtain
\begin{equation*}
\begin{split}
    \nu \int_{B_R}|\Hnabla v(x)|^2\, dx &\leq \int_{B_R}\left\langle A(x)\Hnabla v(x),\Hnabla v(x)\right\rangle\,dx \\
    &\leq\int_{B_R}\left\langle A(x)\Hnabla u(x),\Hnabla u(x)\right\rangle\,dx \leq L \int_{B_R}|\Hnabla u(x)|^2\,dx
\end{split}    
\end{equation*}
that implies
\begin{equation}\label{eq:enbd}
    \int_{B_R}|\Hnabla v(x)|^2\, dx\leq c \int_{B_R}|\Hnabla u(x)|^2\, dx,
\end{equation}
for some $c>0$ depending only on $\nu$ e $L$.

Then we notice that using $u-v$ as test function in \eqref{eq-div-var-coeff} and \eqref{eq:frzdir}, we get
\begin{equation*}
    \int_{B_{R}} \left\langle A(x)\Hnabla u(x),\Hnabla u(x) -\Hnabla v(x)\right\rangle\,dx=0=\int_{B_R}\left\langle A(x_0)\Hnabla v(x),\Hnabla u(x) -\Hnabla v(x)\right\rangle\, dx 
\end{equation*}    
that leads to 
\begin{align}\label{last-lemma-ms}
\nu\int_{B_R}&|\Hnabla u(x)-\Hnabla v(x)|^2\, dx \\ \nonumber
&\qquad \leq \int_{B_R}\left\langle A(x_0)\Hnabla u(x)-A(x_0)\Hnabla v(x),\Hnabla u(x) -\Hnabla v(x)\right\rangle\, dx\\ \nonumber
&\qquad= \int_{B_R}\left\langle A(x_0)\Hnabla u(x)- A(x)\Hnabla u(x),\Hnabla u(x) -\Hnabla v(x)\right\rangle\, dx\\ \nonumber
&\qquad\leq L'R^\alpha \int_{B_R}|\Hnabla u(x)|
|\Hnabla u(x)-\Hnabla v(x)|\, dx\\ \nonumber
&\qquad\leq c\,L'R^\alpha \int_{B_R}|\Hnabla u(x)|^2+|\Hnabla v(x)|^2\, dx \nonumber
\end{align}
where the first and second inequalities are consequences of \eqref{st.cond1} and \eqref{st.cond2} respectively and in the last inequality we use Young's and triangular inequality. Combining \eqref{eq:enbd} and \eqref{last-lemma-ms} we get the proof of \eqref{eq:frzcomp}. 
\end{proof}

\subsection{Proof of the Theorem \ref{thm:c1au}}
In this section, we present the proof Theorem \ref{thm:c1au}. As above, throughout this subsection, we denote by $u\in HW^{1,p}(\Omega)$ a weak solution of the equation \eqref{eq-div-var-coeff}. Furthermore, fixed $x_0\in \Omega$, for all concentric metric balls centered in $x_0$ we set $B_\rho=B_\rho(x_0)$ for every $\rho>0$.

With respect to the given data, let us set 
\begin{equation}\label{eq:r0}
 \bar R=\bar{R}(Q,\lambda,\Lambda,L',\alpha,\dist(x_0,\partial\Omega))>0,
\end{equation}
which shall be chosen as small as required later.
Let $$\bar R\leq \min\{1, \frac{1}{2}\dist(x_0,\partial\Omega), L'^{-1/\alpha}\}$$ to begin with, so that for any $R<\bar R$, we can suppose that 
\begin{equation}\label{scelta-R}
R<1 \quad L'R^\alpha <1 \quad \text{and}\quad B_R\subset \subset\Omega.
\end{equation}

Before dealing with the proof of the Theorem \ref{thm:c1au}, we preface a series of lemmas that are a consequence of the comparison estimates just seen and of regularity results in \cite{CM22}.
\begin{lem}\label{lem:amlip}	
For any $0<\rho\leq R \leq \bar R/2$, we have the estimate
\begin{equation}\label{eq:amlip}
\int_{B_\rho}|\Hnabla u(x)|\, dx\,\leq\, c\left(\frac{\rho}{R}\right)^{Q}\int_{B_R} |\Hnabla u(x)|\, dx 
\,+ c(L'R^\alpha)^\frac{1}{p}\fint_{B_{2R}}|\Hnabla u(x)|\, dx.
\end{equation}
\end{lem}

\begin{proof}
Let $v$ be the solution of \eqref{eq:frzdir}. By choosing $q=1$ in \eqref{cm-q=1}, it follows 
\begin{equation}\label{o0}
   \int_{B_\rho}|\Hnabla v(x)|\, dx \leq c \left(\frac{\varrho}{R}\right)^Q\int_{B_\rho}|\Hnabla v(x)|\, dx.
\end{equation}
for some constant $c>0$ depending only on $Q$.
Then we write
\begin{equation}\label{o1}
\int_{B_\rho}|\Hnabla u(x)|\, dx \leq \int_{B_\rho}|\Hnabla v|\, dx + \int_{B_\rho}|\Hnabla u(x)-\Hnabla v(x)|\, dx. 
\end{equation}
Now using \eqref{o0} and triangular inequality, we estimate the first term \eqref{o1} as
\begin{equation}\label{o2}	
\begin{aligned}
\int_{B_\rho}|\Hnabla v(x)|\, dx &\leq  c\left(\frac{\rho}{R}\right)^Q\int_{B_R}|\Hnabla v(x)|\, dx \\
&\leq c\left(\frac{\rho}{R}\right)^Q\int_{B_R}|\Hnabla u(x)|\, dx +c\left(\frac{\rho}{R}\right)^Q\int_{B_R}|\Hnabla v(x)-\Hnabla u(x)|\, dx.
\end{aligned}
\end{equation}

Then, we notice that H\"older inequality and \eqref{eq:frzcomp} imply 
\begin{equation}\label{eq:frzcomp'}
\fint_{B_R}|\Hnabla v(x)-\Hnabla u(x)|\, dx\leq c(L'R^\alpha)^\frac{1}{2}\Big(\fint_{B_R}|\Hnabla u(x)|^2\, dx\Big)^\frac{1}{2}.
\end{equation}
Now by Caccioppoli-type estimates proved in \cite{BKL}*{Section2} (see also equation (3.26) in \cite{Lu96}) and \eqref{scelta-R}, we have
\begin{equation}\label{Cac-est}
    \int_{B_R}|\Hnabla u(x)|^2\,dx\leq c_0R^2 \int_{B_{2R}}\left | \,u(x)-(u)_{B_{2R}} \right |^2 \,dx \leq c_0 \int_{B_{2R}}\left | \,u(x)-(u)_{B_{2R}} \right |^2 \,dx    \end{equation}
for some constant $c_0>0$. Then by Poincar\'e-Sobolev inequality (see e.g. \cites{FLW95,Lu94}) we have 
\begin{equation}\label{PS-2ast}
    \left(\int_{B_{2R}}\left | \,u(x)-(u)_{B_{2R}} \right |^2\,dx\right)^{\frac 1 2} \leq C \left(\int_{B_{2R}}|\Hnabla u(x)|^q\,dx\right)^{\frac{1}{q}}
\end{equation}
for $q:=\frac{2Q}{Q+2}<2$ and some universal constant $C>0$. Thus combing \eqref{Cac-est} and \eqref{PS-2ast}, up to remaining constant, we obtain 
\begin{equation}
    \left(\int_{B_R}|\Hnabla u(x)|^2\,dx\right)^{\frac 1 2 }\leq C \left(\int_{B_{2R}}|\Hnabla u(x)|^q\,dx\right)^{\frac 1 q }\quad \text{with  }q=\frac{2Q}{Q+2}.
\end{equation}
Recalling that $Q\geq2$, we are in position to apply Lemma 1.4 in \cite{BKL} with $r=1$, $q=\frac{2Q}{Q+2}$ and $p=2$ (so that $0<r<q<p$) to conclude that 
\begin{equation*}
    \left(\int_{B_R}|\Hnabla u(x)|^2\,dx\right)^{\frac 1 2 }\leq C' \int_{B_{2R}}|\Hnabla u(x)|\,dx
\end{equation*}
for some constant $C'>0$. So by \eqref{eq:frzcomp'}, we conclude that 
\begin{equation}\label{reverseHOLD}
    \fint_{B_R}|\Hnabla v(x)-\Hnabla u(x)|\, dx\leq c(L'R^\alpha)^\frac{1}{2}\fint_{B_{2R}}|\Hnabla u(x)|\, dx,
\end{equation}
fro some constant $c>0$ depending only on $Q$.

Finally, putting together \eqref{o1} and \eqref{o2} and using \eqref{reverseHOLD}, we conclude
\begin{equation}
\begin{split}
    \int_{B_\rho}|\Hnabla u(x)|\, dx &\leq c\left(\frac{\rho}{R}\right)^Q\int_{B_R}|\Hnabla u(x)|\, dx +c\left(\frac{\rho}{R}\right)^Q\int_{B_R}|\Hnabla v(x)-\Hnabla u(x)|\, dx\\ & \qquad\qquad\qquad\qquad\qquad\qquad\qquad\qquad\qquad\qquad\quad+\int_{B_\rho}|\Hnabla u(x)-\Hnabla v(x)|\, dx \\
    &\leq c\left(\frac{\rho}{R}\right)^Q\int_{B_R}|\Hnabla u(x)|\, dx + c\int_{B_R}|\Hnabla u(x)-\Hnabla v(x)|\, dx \\
    &\leq c\left(\frac{\rho}{R}\right)^Q\int_{B_R}|\Hnabla u(x)|\, dx +c(L'R^\alpha)^\frac{1}{2}\fint_{B_2R}|\Hnabla u(x)|\, dx,
\end{split}    
\end{equation}
fro some renaming constant $c>0$, depending only on $Q$. This concludes the proof of \eqref{eq:amlip}.
\end{proof}

Now using Lemma \ref{lem:amlip} and a standard perturbation lemma (see Lemma 2.1. in \cite{giaB}*{Chapter III}) we obtain the following regularity estimate.
\begin{prop}\label{prop:Aml}
There exists $c=c(Q,\nu,L)>0$ such that,
\begin{equation}\label{eq:Aml}
\int_{B_r}|\Hnabla u(x)|\, dx\leq c \left(\frac{r}{R}\right)^{Q-\bar\varepsilon}\int_{B_R}|\Hnabla u(x)|\, dx 
\end{equation}
holds for any $0<\bar\varepsilon <Q$ and $0<r\leq R\leq \bar R$.
\end{prop}
\begin{proof}
At first, let us fix $0<r\leq \bar R$ and denote 
$$ \phi(r):= \int_{B_r}|\Hnabla u(x)|\, dx.$$ 
By \eqref{eq:amlip} with the appropriate rescaling, we have
\begin{equation}\label{scelta-esp0}
\phi(\rho)\leq c\left(\frac{\rho}{r}\right)^{Q}\phi(r)+ c(L'r^\alpha)^\frac{2}{p}\phi(r),
\end{equation}
for any $\rho\leq r$ and $c=c(Q,\nu,L)>0$. So we can apply Lemma 2.1. in \cite{giaB}*{Chapter III} (see also \cite{SS}*{Lemma 4.2} and \cite{camp}*{Lemma 6.I-II)} with $\alpha=Q$, $B=0$ and $\beta=Q-\bar\varepsilon$ for some $0<\bar\varepsilon<Q$ with the appropriate reduction 
$$ (L'{\bar R}^\alpha)^\frac{1}{2}\leq\epsilon_0(Q,\nu,L),$$ to conclude that  
$$ \phi(r)  \leq c\left(\frac{r}{R}\right)^{Q-\bar\varepsilon}\phi(R), $$
for every $0<r\leq R\leq \bar R$, that completes the proof of \eqref{eq:Aml}. 
\end{proof}

Now, using estimate \eqref{eq:Aml2} and integral oscillation decay estimate for $v$ as in \eqref{eq:frzdir}, we prove $C^{1,\gamma}$ regularity of $u$. First, we have the following lemma.
 \begin{lem}\label{lem:intosc}
 There exist $\beta=\beta(Q,\nu,L)\in(0,1)$ and $c=c(Q,\nu,L)>0$ such that, for every $0<\varrho < r/4 < \bar R/2$, the following estimate holds:
\begin{equation}\label{eqlem3.4}
\fint_{B_\varrho}|\Hnabla u(x) - (\Hnabla u)_{B_\varrho}|\, dx \,\leq\, c\Big(\frac{\varrho}{r}\Big)^\beta  \fint_{B_{4r}}|\Hnabla u(x)|\, dx+c\,\Big(\frac{r}{\varrho}\Big)^Q (L'r^\alpha)^\frac{1}{2}\fint_{B_{4r}}|\Hnabla u(x)|\, dx.
\end{equation}
\end{lem}
\begin{proof}
To seek simplicity, in the following, we will denote all constants as $c$ but the values of which may vary from line to line and, unless explicitly specified otherwise, they are positive and depending only on $Q$, $L$ and $\nu$. Let $v$ the $\G$-harmonic replacement of $u$ in $B_R$. Since it results 
\begin{equation*}
    |(\Hnabla u)_{B_\varrho}-(\Hnabla v)_{B_\varrho}|=\left|\fint_{B_\varrho}\Hnabla u(x) - (\Hnabla v)_{B_\varrho}\, dx\right|\leq \fint_{B_\varrho}|\Hnabla u(x) - (\Hnabla v)_{B_\varrho}|\, dx,
\end{equation*}
using triangular inequality, we obtain
\begin{equation}\label{eq:io1}
\begin{aligned}
\fint_{B_\varrho}|\Hnabla u(x) - (\Hnabla u)_{B_\varrho}|\, dx &\leq 
2\fint_{B_\varrho}|\Hnabla u(x) - (\Hnabla v)_{B_\varrho}|\, dx \\
&\leq 2\fint_{B_\varrho}|\Hnabla v(x) - (\Hnabla v)_{B_\varrho}|\, dx 
+2\fint_{B_\varrho}|\Hnabla u(x) - \Hnabla v(x)|\, dx . 
\end{aligned}
\end{equation}
Now, we shall estimate both terms of the right-hand side of \eqref{eq:io1} 
separately. 

At first, we notice that 
\begin{equation*}
    \fint_{B_\varrho}|\Hnabla v(x) - (\Hnabla v)_{B_\varrho}|\, dx \leq \fint_{B_\varrho}|\Hnabla v(x)|\,dx +|(\Hnabla v)_{B_\varrho}|\leq 2 \fint_{B_\varrho}|\Hnabla v(x)|\,dx
\end{equation*}
as a consequence, applying estimate \eqref{int-osc-decay-freez-eq} (note that $v$ solution of \eqref{eq:frzdir} by Theorem 1.3 of \cite{CM22} is of class $C^{1,\beta}$ for some $\beta\in (0,1]$) and choosing $q=1$ in \eqref{cm-q=1} we obtain 
\begin{equation}\label{media-v-pass}
    \begin{split}
        \fint_{B_\varrho}|\Hnabla v(x) - (\Hnabla v)_{B_\varrho}|\, dx &\leq c\left(\frac{\rho}{r}\right)^\beta\left(\int_{B_r}|\Hnabla v(x) - (\Hnabla v)_{B_r}|\,dx+c\sup_{B_{2r}}|\Hnabla v|\left(\frac{1}{2}\right)^\beta\,\right)\\
    &\leq c\left(\frac{\rho}{r}\right)^\beta\left(\int_{B_r}|\Hnabla v(x)|\, dx + c\sup_{B_{2r}}|\Hnabla v|\right) \\
    &\leq c\left(\frac{\rho}{r}\right)^\beta\int_{B_{4r}}|\Hnabla v(x)|\, dx.
    \end{split}
\end{equation}

Using \eqref{media-v-pass}, we estimate 
the first term of \eqref{eq:io1} as 
\begin{equation*}
\begin{aligned}
\fint_{B_\varrho}|\Hnabla v(x) - (\Hnabla v)_{B_\varrho}|\, dx &\leq c\Big(\frac{\varrho}{r}\Big)^\beta 
\fint_{B_{4r}}|\Hnabla v(x)|\, dx \\
&\leq  c\Big(\frac{\varrho}{r}\Big)^\beta \fint_{B_{4r}}|\Hnabla u(x)|\, dx  
+  c\Big(\frac{\varrho}{r}\Big)^\beta \fint_{B_{4r}}|\Hnabla v(x)-\Hnabla u(x)|\, dx 
\end{aligned}
\end{equation*}
 
The second term of \eqref{eq:io1} is estimated simply as 
\begin{equation*} 
\fint_{B_\varrho}|\Hnabla u(x) - \Hnabla v(x)|\, dx \leq 
c\Big(\frac{r}{\varrho}\Big)^Q\fint_{B_{2r}}|\Hnabla u(x) - \Hnabla v(x)|\, dx,
\end{equation*}
that together \eqref{eq:io1} and  \eqref{reverseHOLD} implies 
\end{proof}

Now we are ready to prove Theorem \ref{thm:c1au}.
\begin{proof}[Proof of Theorem \ref{thm:c1au}] Let us consider $0<\varrho<r<R/4<\bar R/4$. From Lemma \ref{lem:intosc}, we get
\begin{equation}\label{eq:hold1}
\begin{aligned}
\int_{B_\varrho}|\Hnabla u(x) - (\Hnabla u)_{B_\varrho}|\, dx \,\leq\, c\Big(\frac{\varrho}{r}\Big)^{Q+\beta} & \int_{B_{4r}}|\Hnabla u(x)|\, dx
 +c(L'r^\alpha)^\frac{1}{2}\int_{B_{4r}}|\Hnabla u(x)|\, dx
\end{aligned}
\end{equation}
and from \eqref{eq:Aml} of Proposition \ref{prop:Aml}, we have, 
\begin{equation}\label{eq:Aml2}
\int_{B_{4r}}|\Hnabla u(x)|\, dx\leq c\left(\frac{r}{R}\right)^{Q-\bar\varepsilon}\int_{B_R}|\Hnabla u(x)|\, dx 
\end{equation}
Now, combing \eqref{eq:Aml2} and \eqref{eq:hold1}, we obtain 
\begin{equation}\label{hl1}
\int_{B_\varrho}|\Hnabla u(x) - (\Hnabla u)_{B_\varrho}|\, dx \,\leq\, c 
\Big(\frac{\varrho^{Q+\beta}R^{\bar\varepsilon}}{r^{\beta+\bar\varepsilon} R^{Q}}\Big)
\int_{B_R}|\Hnabla u(x)|\, dx+ c\,  (L'r^\alpha)^\frac{1}{2}\left(\frac{r}{R}\right)^{Q-\bar\varepsilon}\int_{B_{R}}|\Hnabla u(x)|\, dx.
\end{equation}
 
Moreover, choosing
$$\bar\varepsilon<\alpha/2, \quad \delta<\alpha/2-\bar\varepsilon \quad \text{ and }\quad L' \bar R^{\alpha-(\delta +\bar\varepsilon)}< 1$$ we obtain the estimate
\begin{equation}\label{stima-delta}
\int_{B_\varrho}|\Hnabla u(x) - (\Hnabla u)_{B_\varrho}|\, dx \leq c 
\Big(\frac{\varrho^{Q+\beta}}{r^{\beta+\bar\varepsilon}} + r^{Q+\delta}\Big)\fint_{B_R}|\Hnabla u(x)|\, dx. 
\end{equation}
Now, since $0<\varrho<r<1$, for some $\kappa\in(0,1)$, we can choose $r= \varrho^\kappa$, in \eqref{stima-delta} to obtain
\begin{equation*}
\begin{aligned}
\int_{B_\varrho}|\Hnabla u(x) - (\Hnabla u)_{B_\varrho}|\, dx &\leq c \big(\varrho^{Q+(1-\kappa)\beta-\kappa\bar\varepsilon} + \varrho^{\kappa(Q+\delta)}\big)
\fint_{B_R}|\Hnabla u(x)|\, dx \\
&\leq c\varrho^{Q+\lambda} \fint_{B_R}|\Hnabla u(x)|\, dx,
\end{aligned}
\end{equation*}
where the latter inequality follows when $Q+\lambda \leq \min\{Q+(1-\kappa)\beta-\kappa\bar\varepsilon\,,\, \kappa(Q+\lambda)\}$; 
indeed we can make sure that this is true with the choice 
of $\kappa = \kappa(\lambda)$ such that 
\begin{equation}\label{intervallo-kappa}
    \frac{Q+\lambda}{Q+\delta} \,\leq\, \kappa \,\leq\, \frac{\beta-\lambda}{\beta+\bar\varepsilon},
\end{equation}   
for any $0<\lambda< \beta\delta/(Q+\beta+\delta +\bar\varepsilon)$. Moreover, as soon as $\lambda, \bar\varepsilon$ are small enough, $\kappa=\kappa(\lambda)$ can be chosen close enough to $1$ to make sure that $\varrho^\kappa\leq R$, whenever $0<\varrho<R$. 
Thus, we have obtained 
\begin{equation}\label{final-step-campanato}  
\fint_{B_\varrho}|\Hnabla u - (\Hnabla u)_{B_\varrho}|\, dx \leq c\varrho^\lambda
\fint_{B_R}|\Hnabla u(x)|\, dx,
\end{equation} 
for any $0<\varrho< R\leq \bar R$. By the arbitrariness of of $0<\varrho< R$, it follows that $\Hnabla u\in \mathcal{E}^{\gamma,1}(B_R,H\G)$, with $\gamma=\lambda/Q$. The proof follows by applying Theorem \ref{theo-isomorph-Campanato-spaces-Holder-spaces}.
\end{proof}

\begin{remark}\label{stima-l-inf-rmk}
In the assumption of Theorem \ref{thm:c1au}, fixed $x_0\in\Omega$ and $R>0$ such that $2R<\bar R$, by \eqref{eq:teo-nuovo} and triangular inequality, for any $x,y\in B_R(x_0)$, it exists a constant $C>0$, depending only on $\G$, such that 
\begin{equation}\label{stima-l-inf}
\begin{split}
    |\Hnabla u(x)|-|\Hnabla u(y)|&\leq C\max_{1\leq i\leq m_1}|X_i(x)-X_i(y)|\\ &\leq  c_0\,d_c(x,y)^\gamma \fint_{B_R(x_0)}|\Hnabla u(x)|\, dx
    \leq c_0\,R^\gamma\fint_{B_R(x_0)}|\Hnabla u(x)|\,dx
    \end{split}
\end{equation}
up to renaming constant $c_0=c_0(\G,\nu,L)>0$. This implies for any $x\in B_R(x_0)$,
\begin{equation}\label{stima-l-inf-bis}
    |\Hnabla u(x)|\leq c_0\,R^\gamma\fint_{B_R(x_0)}|\Hnabla u(x)|\,dx+\fint_{B_R(x_0)}|\Hnabla u(z)|\, dz\leq c_1\fint_{B_{2R}(x_0)}|\Hnabla u(x)|\,dx
\end{equation}
up to renaming constant $c_1=c_1(\G,Q,\nu,L,R,\gamma)>0$. Finally, by \eqref{stima-l-inf-bis} and Jensen's inequality, for every $p\in[1,\infty)$, we obtain
\begin{equation}\label{stima-l-inf-tris}
    \sup_{B_R(x_0)}|\Hnabla u(x)|^p\leq c_2\fint_{B_{2R}(x_0)}|\Hnabla u(x)|^p\,dx\end{equation}
for some constant $c_2=c_2(\G,Q,\nu,L,R,\gamma,p)>0$.
\end{remark}

\section{Improvement of dichotomy}\label{sec:impr}
In this section, based on the regularity results presented in Section \ref{sec:reg}, we want to show that the alternative in \eqref{alt2} can be {\em improved} when $\e$ and $\sigma$ are sufficiently
small. 

Fist, we have the following Lemma that we will use later.
\begin{lem}\label{lemma:unifell}
		Let $\bq:\G\to H\G$ a constant horizontal section and let $\eta:\G\to H\G$ an horizontal section be such that $|\eta(x)|<\frac{|\bq|}2$, for any $x\in\G$.
		Let $F:\R^{m_1}\to \R^{m_1}$ the map defined by $F(z):=|z|^{p-2}z$. Let us consider
            \[ A(x):=\int_0^1D_{z}F\big(\bq(x)+t\eta(x) \big)\,dt \]       
            where, for $z\in H\G_x$, the Euclidean vector fields on $H\G_x\cong \R^{m_1}$ are denoted as $D_{z_j}$ for $j=1,\ldots,m_1$ and $D_z=(D_{z_1}\ldots,D_{z_{m_1}})$ is the Euclidean gradient. 
        Then,
		$$ \lambda\,|q|^{p-2} |\xi|^2 \le \left\langle A(x)\xi, \xi \right\rangle \le\Lambda\, |q|^{p-2} |\xi|^2,$$
	for any $x\in\G$ and $\xi\in H\G_x$, and some $\Lambda\ge\lambda>0$, depending on $p$.
\end{lem}

	\begin{proof}
 First of all, we notice that, for all $t\in(0,1)$ and $x\in\G$, by triangular inequality, we have
	\begin{equation}\begin{split}\label{triang_q-teta}
				&|\bq(x)+t\eta(x)|\le |\bq|+|\eta(x)|< |\bq|+\frac{|\bq|}2=\frac{3|\bq|}2\\
				{\mbox{and }}\quad& |\bq(x)+t\eta(x)|\ge |\bq|-|\eta(x)|>|\bq|-\frac{|\bq|}2=\frac{|\bq|}2.
		\end{split}\end{equation}
		
		Furthermore, for $z\in H\G_x\cong \R^{m_1}$, it results
		$$
		D_zF(z)=(p-2)|z|^{p-4}z\otimes z +|z|^{p-2}{\text{Id}},
		$$
		where ${\text{Id}}$ denotes the identity map on $H\G_x\cong\R^{m_1}$. As a consequence, for all $\xi\in H\G_x$,
		\begin{eqnarray*}
			\left\langle D_zF(z)\xi,\xi\right\rangle&=&(p-2)|z|^{p-4}(\langle z,\xi\rangle)^2+|z|^{p-2}|\xi|^2\\&\ge&
			-(2-p)^+|z|^{p-2}|\xi|^2+|z|^{p-2}|\xi|^2\\&=&\big(1-(2-p)^+\big)\,|z|^{p-2}|\xi|^2,
            \end{eqnarray*}
		that by integration leads to
		\begin{eqnarray*}
			&& \langle A(x)\xi,\xi\rangle \ge \big(1-(2-p)^+\big) |\xi|^2
			\int_0^1|\bq(x)+t\eta(x)|^{p-2} \,dt.
		\end{eqnarray*}
		Then, using \eqref{triang_q-teta} we obtain
		$$ \left\langle A(x)\xi,\xi\right\rangle\ge \frac{1-(2-p)^+}{2^{p-2}}  \,|\bq|^{p-2}|\xi|^2
		$$
		if $p\ge2$, and
		$$ \left\langle A(x)\xi,\xi\right\rangle\ge \big(1-(2-p)^+\big)\left(\frac32\right)^{p-2}|\bq|^{p-2}|\xi|^2
		$$
		if $p\in(1,2)$, for any $\xi\in H\G_x$.
		
		Similarly, for any $z,\xi\in H\G_x$, it results
		\begin{eqnarray*}
			\left\langle D_zF(z)\xi,\xi\right\rangle &=& 
            (p-2)|z|^{p-4}(\langle z,\xi\rangle)^2+|z|^{p-2}|\xi|^2\\   &\le&
			(p-2)^+|z|^{p-2}|\xi|^2+|z|^{p-2}|\xi|^2\\&=&\big(1+(p-2)^+\big)\,|z|^{p-2}|\xi|^2,
		\end{eqnarray*}
		that by integration leads to
		\begin{eqnarray*}
			&& \left\langle A(x)\xi,\xi\right\rangle\le \big(1+(p-2)^+\big) |\xi|^2
			\int_0^1|\bq(x)+t\eta(x)|^{p-2} \,dt.
		\end{eqnarray*}
		Hence, making again use of \eqref{triang_q-teta},
		$$ \left\langle A(x)\xi,\xi\right\rangle\le \big(1+(p-2)^+\big) \left(\frac32\right)^{p-2}|\bq|^{p-2}|\xi|^2
		$$
		if $p\ge2$, and
		$$ \left\langle A(x)\xi,\xi\right\rangle\le\frac{1+(p-2)^+}{2^{p-2}}\,|\bq|^{p-2}|\xi|^2
		$$
		if $p\in(1,2)$, for any $\xi\in H\G_x$.
	 The proof of Lemma \ref{lemma:unifell} is completed by choosing 
  \begin{equation*}
      \lambda:=\left\{\begin{matrix}
\frac{1-(2-p)^+}{2^{p-2}}, \ &\text{ if }p\geq2\vspace{1em}\\ 
\big(1-(2-p)^+\big)\left(\frac32\right)^{p-2}, \ \ &\text{ if }p\in (1,2)
\end{matrix}\right. 
  \end{equation*}
  and
  \begin{equation*}
      \Lambda:=\left\{\begin{matrix}
\big(1+(p-2)^+\big) \left(\frac32\right)^{p-2},\ &\text{ if }p\geq2\vspace{1em}\\ 
\frac{1+(p-2)^+}{2^{p-2}},\ &\text{ if }p\in (1,2).
\end{matrix}\right.
  \end{equation*}
	\end{proof}

Herewith, we can now state the following result that represent the counterpart of Lemma 2.3 \cite{DS} in Euclidean setting and extends, in the more general nonlinear setting dealt with in this paper, Lemma 3.2 in \cite{FFM}.
 
 \begin{lem}\label{lemma-second-alternative-dichotomy-improved}
Suppose that $p>p^\#=p^\#(Q):=\frac{2Q}{Q+2}$. 
Let $u\in HW^{1,p}(B_1)$ with be such that $u\ge0$ a.e. in $B_1$ and
    \begin{equation}\label{hp-bis}
	J_p(u,B_1)\le (1+\sigma) J_p(v,B_1)
    \end{equation}
for all $v\in HW^{1,p}(B_1)$ such that $u-v\in HW^{1,p}_0(B_1)$. Let
\begin{equation}\label{3.15BIS} a:=\left(\fint_{B_1}\left|\Hnabla u(x)\right|^p \,dx\right)^{1/p}\end{equation}
and suppose that $a\in [a_0,a_1]$, for some $a_1>a_0>0$. Assume also that
		\begin{equation}\label{second-alternative-dichotomy}
			\left(\fint_{B_1}\left|\Hnabla u(x)-\bq(x)\right|^p \,dx\right)^{1/p}\le \varepsilon a,
		\end{equation}
		for some constant horizontal section $\mathbf{q}:\G\to H\G$ as in \eqref{def-q} such that
    \begin{equation}\label{estimate-norm-q}
	\frac{a}{8}<|\mathbf{q}|\le 2C_0a,
    \end{equation}
    where $C_0>0$ is the universal constant in Proposition \ref{dic}.		
    
    There exist $\alpha_0\in(0,1]$ depending on $Q$ and $p$, such that for every $\alpha\in(0,\alpha_0)$ there exist
		\begin{itemize}
            \item$\rho\in(0,1)$, depending on $Q$, $p$ and $\alpha$,
		\item $\varepsilon_0\in(0,1)$, depending on $Q$, $p$, $a_0$, $a_1$ and $\alpha$,
		\item $c_0>0$, depending on $Q$, $p$, $a_0$, $a_1$ and $\alpha$,
		\end{itemize}
		such that, if $\varepsilon\in(0, \varepsilon_0]$ and $\sigma\in(0, c_0\varepsilon^P]$, with $P:=\max\{p,2\}$,
    then  
	\begin{equation}\label{tesi-lemma-improv}
		\left(\fint_{B_\rho}\left|\Hnabla u(x)-\widetilde{\mathbf{q}}(x)\right|^p\,dx\right)^{1/p}\le\rho^{\alpha}\varepsilon a,
	\end{equation}
    where $\mathbf{\widetilde{q}}:\G \to H\G$ is a constant horizontal section, i.e.
        \begin{equation}\label{def-q-tilde}
        \mathbf{\widetilde{q}}(x):= \sum_{j=1}^{m_1}\widetilde{q}_jX_j(x), \quad x\in \G
        \end{equation}
    for some suitable $\widetilde{q}=(\widetilde{q}_1,\ldots,\widetilde{q}_{m_1})\in \R^{m_1}$, with
	\begin{equation}\label{stima-q-q-tilde}
		\left|\mathbf{q}-\mathbf{\widetilde{q}}\right|\le \widetilde{C}\varepsilon a.
	\end{equation}
 for some universal constant $\widetilde{C}>0$.	
 \end{lem}

	\begin{proof}[Proof of Lemma \ref{lemma-second-alternative-dichotomy-improved}]We divide the proof into several steps.\medskip
	
	\noindent{\em Step 1: energy estimates for the $(\G,p)$-harmonic replacement and comparison of energies.}		
	Let us setting $\tau:=\frac 1 2 \min \left\{ \frac{1}{10\Lambda_\G^2},R_0 \right\}$, where $\Lambda_\G$ in the structural constant given by \eqref{rhod} depending only on $\G$, and $R_0>0$ is the constant given by Theorem 2.2 in \cite{Lu95}. Since $\Lambda_\G\geq 1$, it results that $\tau<1$. Let $\bar{v}$ denote the $(\G,p)$-harmonic replacement of $u$ in $B_{\tau}$ and let $v$ be defined as
		\begin{equation}\label{defin-of-special-competitor}
			v:=\begin{cases}
				\bar{v}&\mbox{in }B_{\tau},\\
				u&\mbox{in }B_1\setminus B_{\tau}.
			\end{cases}
		\end{equation}
    By definition \eqref{defin-of-special-competitor}, $v\in HW^{1,p}(B_1)$ and $u-v\in HW^{1,p}_0(B_1)$ thus, by hypothesis \eqref{hp-bis}, we have that \[J_p(u,B_1)\le (1+\sigma)J_p(v,B_1),\] this together with the additivity property of the functional $J_p$ with respect to the reference domain leads to
    \begin{align}\label{J_p-B-tau_lemma}
	J_p(u,B_{\tau})=\;&J_p(u,B_{1})- J_p(u,B_1\setminus B_{\tau})\nonumber \\
        \le\;& (1+\sigma)\,J_p(v,B_1)-J_p(u, B_1\setminus B_{\tau})\nonumber  \\              
        =\;& J_p(v,B_{\tau})+J_p(v,B_1\setminus B_{\tau})+\sigma J_p(v,B_{\tau})+\sigma J_p(v,B_1\setminus B_{\tau})-J_p(u,B_1\setminus B_{\tau})\nonumber \\ 
        =\;& J_p(v,B_{\tau})+J_p(v,B_1\setminus B_{\tau})+\sigma J_p(v,B_1)-J_p(u,B_1\setminus B_{\tau})\\
	=\;&J_p(v,B_{\tau})+J_p(u,B_1\setminus B_{\tau})+\sigma  J_p(v,B_1)-J_p(u,B_1\setminus B_{\tau})\nonumber	\\			=\;&J_p(v,B_{\tau})+\sigma J_p(v,B_1).\nonumber
    \end{align}
    By the definition of $J_p$ in \eqref{def-J-p}, \eqref{J_p-B-tau_lemma} reads as
		\begin{align*}
		  \int_{B_{\tau}}\left|\Hnabla u(x)\right|^p\,dx+\left|\left\{u>0\right\}\cap B_{\tau}\right|\leq &\int_{B_{\tau}}\left|\Hnabla v(x)\right|^p\,dx+\left|\left\{u>0\right\}\cap B_{\tau}\right|+\sigma J_p(v,B_1)\\
            \leq & \int_{B_{\tau}}\left|\Hnabla v(x)\right|^p\,dx+\left|B_{\tau}\right|+\sigma J_p(v,B_1),
		\end{align*}
    which, since $u\geq 0$ a.e. in $B_1$ by assumption, yields that
		\begin{equation}\label{diff-norm-grad-zero-lev-set}
			\int_{B_{\tau}}\big(\left|\Hnabla u(x)\right|^p-\left|\Hnabla v(x)\right|^p\big)\,dx
			\le \left|\left\{u=0\right\}\cap B_{\tau}\right|+\sigma J_p(v,B_1).
		\end{equation}
    Moreover, by definition of $v$ in \eqref{defin-of-special-competitor}, we have
	\begin{equation}\begin{split}\label{J_p-v-B1}
		J_p(v,B_1) =\;& \int_{B_1}\Big( |\Hnabla v(x)|^p+\chi_{\{v>0\}}(x)\Big)\,dx\\
                \le\;&\int_{B_{\tau}}|\Hnabla v(x)|^p\,dx +\int_{B_1\setminus B_{\tau}}|\Hnabla v(x)|^p\,dx+|B_1|\\ 
                =\; & \int_{B_{\tau}}|\Hnabla \bar{v}(x)|^p\,dx +\int_{B_1\setminus B_{\tau}}|\Hnabla u(x)|^p\,dx+|B_1|\\ 
                \le\;&\int_{B_1}\left|\Hnabla u\right|^p\,dx+\left|B_1\right|\le\;\left|B_1\right| (a^p+1),
	\end{split}\end{equation} 	
    where in the second inequality we use the fact that $\bar{v}$ is the $(\G,p)$-harmonic replacement of $u$ in $B_{\tau}$ and therefore $\bar{v}$ minimize $p$-Dirichlet energy on $B_\tau$.

    Furthermore, if $p\ge2$, by \eqref{second-ineq-lemma-p-harm-repl}
	and \eqref{diff-norm-grad-zero-lev-set} we deduce that
		\begin{equation*}\label{second-ineq-condition-almost-minim}
			\int_{B_{\tau}}\left|\Hnabla u(x)-\Hnabla v(x)\right|^p\,dx\le
			C \left|\left\{u=0\right\}\cap B_{\tau}\right|+C\sigma J_p(v,B_1),
		\end{equation*}
		for some positive universal constant $C>0$, depending only on $p$. 
		Consequently, exploiting \eqref{J_p-v-B1}, we obtain that
		\begin{equation}\label{third-ineq-condition-almost-minim}
			\int_{B_{\tau}}\left|\Hnabla u(x)-\Hnabla v(x)\right|^p\,dx
			\le C \left|\left\{u=0\right\}\cap B_{\tau}\right|+C\sigma(a^p+1),
		\end{equation}
		up to renaming constant $C>0$, depending only on $p$ and $Q$.
		
If instead $p^\#<p<2$, using \eqref{first-ineq-lemma-p-harm-repl}, \eqref{diff-norm-grad-zero-lev-set} e \eqref{J_p-v-B1} we get 
			\begin{equation}\label{casonuovo}\begin{split}
					& \int_{B_{\tau}}\left|\Hnabla u(x)-\Hnabla v(x)\right|^p\,dx\\
					&\qquad \quad\le\; C\left(\int_{B_{\tau}}\left(\left|\Hnabla u(x)\right|^p-\left|\Hnabla v(x)\right|^p\right)\,dx\right)^{\frac{p}{2}}\left(\int_{B_{\tau}}\big(\left|\Hnabla u(x)\right|+\left|\Hnabla v(x)\right|\big)^p\,dx\right)^{1-\frac{p}{2}}
					\\ &\qquad \quad\le\; C\Big(\left|\left\{u=0\right\}\cap B_{\tau}\right|+\sigma J_p(v,B_1)\Big)^{\frac{p}{2}}
					\left(\int_{B_{\tau}}2^p\left|\Hnabla u(x)\right|^p\,dx\right)^{1-\frac{p}{2}}	\\&\qquad \quad\le\; C
					\Big(\left|\left\{u=0\right\}\cap B_{\tau}\right|+\sigma(a^p+1)\Big)^{\frac{p}{2}}
					a^{p\left(1-\frac{p}{2}\right)}	
					.\end{split}\end{equation}  
    \medskip
	
	\noindent{\em Step 2: measure estimates for the zero level set.}		
		Now, we claim that
		\begin{equation}\label{estimate-zero-set-almost-minim}
			\left|B_{\tau}\cap \left\{u=0\right\}\right|\le C_1\varepsilon^{2+\delta},
		\end{equation}
		for some $C_1>0$ and $\delta>0$.  \medskip
	
	\noindent{\em Step 2.1: comparison with a linear function.}		
		To prove our claim \eqref{estimate-zero-set-almost-minim}, we consider function $\ell :\G \to \R$ defined by 
		\begin{equation}\label{defin-linear-funct}
			\ell(x):= b+\left \langle \mathbf{q}(x), \pi_x(x) \right \rangle=b+\sum_{j=1}^{m_1}q_jx_j,\quad {\mbox{ with }}\quad x\in \G \quad {\mbox{ and }}\quad b:= \fint_{B_{1/10}}u(x)\,dx.
		\end{equation}
		We remark that
		\begin{equation} \label{av-u-l}
                \begin{split}
			(u-\ell)_{B_{1/10}}:=\fint_{B_{1/10}}\big(u(x)-\ell(x)\big)\,dx&=b-\left(b+\fint_{B_{1/10}}\left \langle \mathbf{q}(x), \pi_x(x) \right \rangle\,dx\right)\\
            &=\fint_{B_{1/10}}\sum_{j=1}^{m_1} q_jx_j\,dx=0
            \end{split}
            \end{equation}
            where last equality is a consequence of the symmetry with respect to the identity element of the Carnot-Carath\'eodory ball $B_{1/10}$. Then by the Poincar\'e inequality (see e.g. \cites{Jer,Lu92}) we have that 
	\begin{equation}\label{poincare}
            \begin{split}      
	   \|u-\ell-(u-\ell)_{B_{1/10}}\|_{L^p(B_{1/10})}=
		\left\|u-\ell\right\|_{L^p(B_{1/10})}&\le C\left\|\Hnabla (u-\ell)\right\|_{L^p(B_{1/10})}\\
      &\leq C \left\|\Hnabla (u-\ell)\right\|_{L^p(B_1)},
        \end{split}
        \end{equation}
        for some $C>0$ universal.
		
		Since by Proposition \ref{campi-omogenei0} $\Hnabla \ell =\mathbf{q}$, \eqref{poincare} together with hypothesis \eqref{second-alternative-dichotomy} leads to
		\begin{equation}\label{stima-lp-u-l}
			\fint_{B_{1/10}}\left|u(x)-\ell(x)\right|^p\,dx\le 
			C\fint_{B_1}|\Hnabla (u-\ell)(x)|^p\,dx= C\fint_{B_1}|\Hnabla u(x)-\bq(x)|^p\,dx\le C\varepsilon^pa^p.
		\end{equation}
		Finally we remark that, since by assumption $u\ge 0$, it holds that $\ell^-\le \left|u-\ell\right|$, so by \eqref{stima-lp-u-l}, we obtain that 
		\begin{equation}\label{ineq-average-negative-part}
			\fint_{B_{1/10}}(\ell^-(x))^p\,dx\le C\varepsilon^pa^p,
		\end{equation}
        for some $C>0$ universal.
		  \medskip
	
	\noindent{\em Step 2.2: lower bounds on the $\G$-affine function.}		
		Now we claim that if $\varepsilon$ is sufficiently small,
		\begin{equation}\label{lower-bound-linear-funct-1}
			\ell\ge c_1a\quad\mbox{in }B_{\tau},
		\end{equation}	
		for some $c_1>0$. To check this, we argue by contradiction assuming that
		$$ \min_{x\in \overline{B_{\tau}}}\ell(x)<ca$$
		for any $c>0$. 
		We notice that, for every $x\in B_{\tau}$, recalling that $\tau<\frac{1}{10\Lambda_\G^2}$, we get
            \begin{equation*}
                \left | \ell(x)-b \right |=\left | \left \langle \mathbf{q}(x),\pi_x(x) \right \rangle \right |\leq \left | \mathbf{q} \right |\left | x^{(1)} \right |\leq \left | \mathbf{q} \right |\left | x \right|_{\G}\leq \left | \mathbf{q} \right | \Lambda_\G \,d_c(e,x)\leq \Lambda_\G \frac{\left | \mathbf{q} \right |}{10\Lambda_\G^2}=\frac{\left | \mathbf{q} \right |}{10\Lambda_\G},
            \end{equation*}
            where $\Lambda_\G\geq 1$ is the constant given by \eqref{rhod} which only depends on the group $\G$. 
        As a consequence 
		$$ -\frac{\left | \mathbf{q} \right |}{10\Lambda_\G}\le \ell(x)-b\le\frac{\left | \mathbf{q} \right |}{10\Lambda_\G} \qquad \text{ for any }x\in B_{\tau}$$
	and therefore
		$$  ca>\min_{x\in \overline{B_{\tau}}}\ell(x)\ge b-\frac{\left | \mathbf{q} \right |}{10\Lambda_\G}$$
	that leads to
	\begin{equation}\label{stima-sopra-b}
		b\le  ca +\frac{|\mathbf{q}|}{10\Lambda_\G},
        \end{equation}
        for any $c>0$.
        Now, taking into account the usual identifications given by exponential coordinates, let us consider 
	\begin{align*}
            {\mathcal{B}}:=\Big\{x=\left [ x^{(1)},x^{(2)} \right ]\in \R^{m_1}\times \R^{m_2}\equiv \G:\;&x^{(1)}_j=-\frac{tq_j}{|\mathbf{q}|}+\eta_j,\; \text{for }j=1,\ldots,m_1; \\
             &x^{(2)}_i=\xi_i,\; \text{for }i=m_1+1,\ldots,m_2 \;{\mbox{ 
                for some }} (t,\eta,\mathbf{\xi})\in\mathcal{A}\Big\},
        \end{align*}
        where we set 
        \begin{equation*}
            \mathcal{A}:=\left \{ (t,\eta,\mathbf{\xi})\in \R\times \R^{m_1}\times \R^{m_2}: \; t\in\left[\frac{1}{4\Lambda_\G},\frac{3}{10\Lambda_\G}\right],\;\sum_{j=1}^{m_1}\eta_j^2\leq \frac{1}{100\Lambda_\G^2}{\mbox{  and  }} \sum_{i=m_1+1}^{m_2}\xi_i^2\leq \frac{1}{100\Lambda_\G^4} \right\}.
        \end{equation*}
	We notice that if $x\in{\mathcal{B}}$ then 
        \begin{equation*}\begin{split}                   d_c(e,x)^4&\leq\Lambda_\G^4|x|_{\G}^4 = \Lambda_\G^4\Big(\sum_{j=1}^{m_1}\big(-\frac{tq_j}{|\bq|}+\eta_j\big)^2\Big)^2+\Lambda_\G^4\sum_{i=m_1+1}^{m_2}\xi_i^2\leq 4\Lambda_\G^4\Big(t^2+\sum_{j=1}^{m_1}\eta_j^2\Big)^2+\Lambda_\G^4\sum_{i=m_1+1}^{m_2}\xi_i^2 \\
        &\leq 4\Lambda_\G^4 \Big(\frac{9}{100\Lambda_\G^2}+\frac1{100\Lambda_\G^2}\Big)^2+\frac{\Lambda_\G^4}{100\Lambda_\G^4}=\frac{4}{100}+\frac{1}{100}=\frac{1}{20}<\frac{1}{10},
        \end{split}
        \end{equation*}
        so we have that
        \begin{equation}\label{calB-in-B1}
			{\mathcal{B}}\subseteq B_{1/10}.
        \end{equation}
        Furthermore, by \eqref{estimate-norm-q} and \eqref{stima-sopra-b}, we have that, for $x\in{\mathcal{B}}$,
		\begin{eqnarray*}
			\ell(x)=b-t|\mathbf{q}|+\sum_{j=1}^{m_1}q_j\eta_j \le ca +\frac{|\mathbf{q}|}{10\Lambda_\G}-\frac{|\mathbf{q}|}{4\Lambda_\G}+\frac{|\mathbf{q}|}{10\Lambda_\G}
            = ca- \frac{1}{20\Lambda_\G}|\mathbf{q}|.
		\end{eqnarray*}
	Now, using hypothesis \eqref{estimate-norm-q}, we get 
        $$ \ell(x)\leq ca - \frac{1}{160\Lambda_\G}a.$$
        Then, taking $c\in\left(0, \frac{1}{320\Lambda_\G}\right)$, we have infer that $$ \ell(x)\le -\frac{a}{320\Lambda_\G}.$$
 
		Accordingly, using this and \eqref{calB-in-B1} into \eqref{ineq-average-negative-part}, we obtain that
		\begin{eqnarray*}
			&& C\,|B_1|\,\varepsilon^pa^p\geq C\,|B_{1/10}|\,\varepsilon^pa^p\geq\int_{B_{1/10}}(\ell^-(x))^p\,dx\int_{{\mathcal{B}}}(\ell^-(x))^p\,dx\ge\int_{{\mathcal{B}}}\left( \frac{a}{320\Lambda_\G}\right)^p\,dx\ge \overline{c} a^p,
		\end{eqnarray*}
		for some positive universal constant $\overline{c}$. This establishes the desired contradiction if $\e$ is sufficiently small, and thus the proof of \eqref{lower-bound-linear-funct-1} is complete.
		
  \medskip
		
    \noindent{\em Step 2.3: conclusion of the proof of \eqref{estimate-zero-set-almost-minim}}.
	We can now address the completion of the proof of the measure estimate in \eqref{estimate-zero-set-almost-minim}. To this end, we distinguish the three following cases: $p\in (1,Q),$ $p=Q$ and $p>Q$.\medskip
	
	\medskip
	
	\noindent{\em Step 2.3.1: the case $p<Q$.}	
	If $p<Q,$ recalling the Poincar\'e-Sobolev inequality (see e.g. \cites{FLW95,Lu94}), we get 
		\begin{equation}\label{Sobolev-Poincare-ineq}
            \begin{split}
		\left(\int_{B_{1/10}}\left|u(x)-\ell(x)\right|^{p^*}\,dx\right)^{1/p^*}&\le C\left(\int_{B_{1/10}}\left|\Hnabla \big(u(x)-\ell(x)\big)\right|^{p}\,dx\right)^{1/p}\\
        &\leq C\left(\int_{B_{1}}\left|\Hnabla \big(u(x)-\ell(x)\big)\right|^{p}\,dx\right)^{1/p}\\            \end{split}
            \end{equation}
		for some $C>0$ universal, where
		\[p^*:= \frac{Qp}{Q-p}.\]		
		Then, by virtue of \eqref{second-alternative-dichotomy}, \eqref{defin-linear-funct} and \eqref{Sobolev-Poincare-ineq}, we get that 
		\begin{eqnarray*}
			&&\left(\displaystyle\int_{B_{1/10}}\left|u(x)-\ell(x)\right|^{p^*}\,dx\right)^{1/p^*}
			\le C	 \left(\int_{B_1}\left|\Hnabla \big(u(x)-\ell(x)\big)\right|^{p}\,dx\right)^{1/p}\\
			&&\qquad \qquad \qquad \qquad \qquad \qquad  \ \ \ \ \,= C \left(\int_{B_1}\left|\Hnabla u(x)-\bq(x)\right|^{p}\,dx\right)^{1/p}
			\le C\varepsilon a.
		\end{eqnarray*}
		that together with \eqref{lower-bound-linear-funct-1} entail that
		\begin{align*}
			C\varepsilon a \ge&\left(\int_{B_{1/10}}\left|u(x)-\ell(x)\right|^{p^*}\,dx\right)^{1/p^*}\\\vspace{1em}
            \ge& \left( \int_{B_{\tau}\cap  \left\{u=0\right\}}\left|\ell(x)\right|^{p^*}\,dx\right)^{1/p^*} 
            \ge c_1a\left|B_{\tau}\cap  \left\{u=0\right\}\right|^{1/p^*},
		\end{align*}
		and thus, up to renaming constants,		
		\begin{equation}\label{Sobolev-Poincare-ineq-p-<-Q-1}
			\left|B_{\tau}\cap  \left\{u=0\right\}\right|\le C\varepsilon^{p^*}.
		\end{equation}
		Now we notice that
		\[p^*=\frac{Qp+p^2-p^2}{Q-p}=p+\frac{p^2}{Q-p}.\]
		Therefore, setting 
		\begin{equation}\label{casepminoren}\delta:=\frac{p^2}{Q-p}>0,
		\end{equation}
		we obtain \eqref{estimate-zero-set-almost-minim} from \eqref{Sobolev-Poincare-ineq-p-<-Q-1}
		in the case $p<Q$.\medskip
	
	\noindent{\em Step 2.3.2: the case $p>Q$.}		
		If instead $p>Q$, by Morrey-type inequality \cite{Lu96}*{Theorem 1.1}(see e.g. also \cite{FLW95}) it result 
		\begin{equation*}
		\sup\limits_{x,y\in B_{1/10}, x\neq y}\left|u(x)-\ell(x)-(u(y)-\ell(y))\right|\le 
			C\left( \int_{B_1}\left|\Hnabla \big(u(x)-\ell(x)\big)\right|^{p}\,dx\right)^{1/p}	\end{equation*}
		for some $C>0$ depending on $p$ and $Q$. Now, by \eqref{av-u-l}, we note that for any $x\in B_{1/10}$
        \begin{equation*}
        \begin{split}
            |u(x)-\ell(x)|&=|u(x)-\ell(x)-(u-l)_{B_{1/10}}|\\&=\big|u(x)-\ell(x)-\fint_{B_{1/10}}u(y)-\ell(y)\,dy\big| =\fint_{B_{1/10}}\left|u(x)-\ell(x)-(u(y)-\ell(y))\right|\, dy\\&\leq\sup\limits_{x,y\in B_{1/10}, x\neq y}\left|u(x)-\ell(x)-(u(y)-\ell(y))\right|,
        \end{split}    
        \end{equation*}
	that implies, making again use of \eqref{second-alternative-dichotomy} and \eqref{defin-linear-funct},
		\begin{align*}
			\sup_{B_{1/10}}\left|u-\ell\right|\le C\left( \int_{B_1}\left|\Hnabla (u(x)-\ell(x))\right|^{p}\,dx\right)^{1/p} \le C\e
			a.
		\end{align*}
		As a consequence, by \eqref{lower-bound-linear-funct-1}, for all $x\in B_{\tau}$,
		\begin{eqnarray*}
			C\e a\ge \sup_{B_{1/10}}\left|u-\ell\right|\ge \sup_{B_{\tau}}(\ell-u)\ge \ell(x)-u(x)\ge
			c_1a-u(x),
		\end{eqnarray*}
		and thus $u(x)\ge c_1a-C\e a>0$ for all $x\in B_{\tau}$, as long as $\e$ is sufficiently small.
		Accordingly, it follows that
		\begin{equation}\label{si3875bvc9876bv546980987-0987-70-986}
		|B_{\tau}\cap  \left\{u=0\right\}|=0,\end{equation}
		and this gives \eqref{estimate-zero-set-almost-minim} in the case $p>Q$ as well.
		\medskip
	
	\noindent{\em Step 2.3.3: the case $p=Q$.}
It remains to analyze the case $p=Q$. For this purpose,
		we point out that $u-\ell\in W^{1,Q}(B_1)$, and so we can apply Theorem 2.2 in \cite{Lu95} to obtain that
		\begin{equation}\label{app-moser-trud}
			\int_{B_\tau}\exp\left(A\frac{\left|u(x)-\ell(x)\right|}{\left\|\Hnabla (u-\ell)
				\right\|_{L^Q(B_1)}}\right)^{\frac{Q}{Q-1}}\,dx
			\le C\left|B_\tau\right|\leq C\left|B_1\right|,
		\end{equation}
		where $A>0$ and $C>0$ are positive universal constants. 
		
		Now, since for all $t\geq 0$, it exist a constant $c_0>0$ sufficiently big such that $e^t\ge c_0 t^{Q}$, by \eqref{app-moser-trud}, \eqref{second-alternative-dichotomy} and \eqref{defin-linear-funct} it follows
		\begin{equation}\label{ineq-case-p-=-Q-2}
			\int_{B_\tau}\left|u(x)-\ell(x)\right|^{\frac{Q^2}{Q-1}}\,dx\le C
			\left\|\Hnabla u-\bq\right\|_{L^Q(B_\tau)}^{\frac{Q^2}{Q-1}}\leq C
                \left\|\Hnabla u-\bq\right\|_{L^Q(B_1)}^{\frac{Q^2}{Q-1}}
                \le C\varepsilon^{\frac{Q^2}{Q-1}}
			a^{\frac{Q^2}{Q-1}},	
            \end{equation}
		for some relabeled constant $C>0$. Thus by \eqref{lower-bound-linear-funct-1} and \eqref{ineq-case-p-=-Q-2},
		\begin{eqnarray*}
			&& C\varepsilon^{\frac{Q^2}{Q-1}}
			a^{\frac{Q^2}{Q-1}}\ge
			\int_{B_\tau}\left|u(x)-\ell(x)\right|^{\frac{Q^2}{Q-1}}\,dx\ge \int_{B_{\tau}\cap\{u=0\}}
			\left|\ell(x)\right|^{\frac{Q^2}{Q-1}}\,dx\\&&\qquad\qquad\ \ \, \,\ge c_1^{\frac{Q^2}{Q-1}}a^{\frac{Q^2}{Q-1}}|B_{\tau}\cap\{u=0\}|
			.\end{eqnarray*}
		
	Now, we point out that 
		\[\frac{Q^2}{Q-1}=\frac{Q^2-Q+Q}{Q-1}=Q+\frac{Q}{Q-1},\]
		and therefore, choosing
		$$\delta:=\frac{Q}{Q-1}>0,$$
		we establish \eqref{estimate-zero-set-almost-minim} in the case $p=Q$.\medskip
	
	\noindent{\em Step 3: energy comparison in Lebesgue spaces.}	
		If $p\ge2$, \eqref{third-ineq-condition-almost-minim} together with \eqref{estimate-zero-set-almost-minim} ensures
		\begin{equation}\label{forth-ineq-condition-almost-minim}
			\int_{B_{\tau}}\left|\Hnabla u(x)-\Hnabla v(x)\right|^p\,dx\le C_1\varepsilon^{p+\delta}
			+C\sigma(a^p+1).
		\end{equation}
If instead $1<p<2$, by \eqref{casonuovo} and \eqref{estimate-zero-set-almost-minim} we get	
			\begin{equation}\label{forth-ineq-condition-almost-minimBIS}
				\int_{B_{\tau}}\left|\Hnabla u(x)-\Hnabla v(x)\right|^p\,dx\le
				C \Big(C_1\e^{p+\delta}+\sigma(a^p+1)\Big)^{\frac{p}{2}}
				a^{p\left(1-\frac{p}{2}\right)}.
			\end{equation}	\medskip
	
	\noindent{\em Step 4: estimates the $\G$-linear perturbation of the $(\G,p)$-harmonic replacement.}	
		Now we aim to show that, even if $v-\langle \bq,\pi_{\cdot}\rangle$ is not the $(\G,p)$-harmonic replacement of $u-\langle \bq,\pi_{\cdot}\rangle$, as in the case $p=2$ (see Lemma 3.2 in \cite{FFM}),  it satisfies an appropriate equation in divergence form, as the ones studied in section \ref{sec:reg}.	\medskip
	
	\noindent{\em Step 4.1: perturbative energy estimates.}	
    If $p\ge2$, by \eqref{second-alternative-dichotomy} and \eqref{forth-ineq-condition-almost-minim}, we obtain
		\begin{equation}\label{soiw39vh43v5834v5660987654321uytr}\begin{split}
				\int_{B_{\tau}}\left|\Hnabla v(x)-\bq(x)\right|^p\,dx \le\;& 2^{p-1}\left(
				\int_{B_{\tau}}\left|\Hnabla u(x)-\bq(x)\right|^p\,dx
				+\int_{B_{\tau}}\left|\Hnabla v(x)-\Hnabla u(x)\right|^p\,dx\right)\\
				\le\; &  C_2\varepsilon^pa^p+C_1\varepsilon^{p+\delta}+C\sigma(a^p+1),
		\end{split}\end{equation}
		up to renaming constants.
		
    While if $1<p<2$, using \eqref{second-alternative-dichotomy}
			and \eqref{forth-ineq-condition-almost-minimBIS} we get
			\begin{equation}\label{soiw39vh43v5834v5660987654321uytrBIS}\begin{split}
					\int_{B_{\tau}}\left|\Hnabla v(x)-\bq(x)\right|^p\,dx \le\;& 2^{p-1}\left(
					\int_{B_{\tau}}\left|\Hnabla u(x)-\bq(x)\right|^p\,dx
					+\int_{B_{\tau}}\left|\Hnabla v(x)-\Hnabla u(x)\right|^p\,dx\right)\\
					\le\; &  C_2\varepsilon^pa^p+C \Big(C_1\e^{p+\delta}+\sigma(a^p+1)\Big)^{\frac{p}{2}}
					a^{p\left(1-\frac{p}{2}\right)},
			\end{split}\end{equation}
		up to renaming constants.		Now in the case of $p\ge2$, we suppose $\sigma\le c_0\varepsilon^p$,  for some constant $c_0>0$ that we will determine precisely later. Since by assumption $a\in[a_0,a_1]$, by \eqref{soiw39vh43v5834v5660987654321uytr}, we infer that
		\begin{equation}\label{first-ineq-p-harm-repl-minus-lin-funct}
			\int_{B_{\tau}}\left|\Hnabla v(x)-\bq(x)\right|^p\,dx\le C_2\varepsilon^pa^p+C_1\varepsilon^{p+\delta}+Cc_0\varepsilon^p(a^p+1)
			\le C\varepsilon^pa^{p},
		\end{equation}
		up to relabeling $C>0$.
		
		Similarly, if $1<p<2$, we take $\sigma\le c_0\varepsilon^2$, with $c_0$ to be made precise later. In this case, we deduce from \eqref{soiw39vh43v5834v5660987654321uytrBIS} that
			\begin{equation}\label{first-ineq-p-harm-repl-minus-lin-functBIS-2000}\begin{split}
				\int_{B_{\tau}}\left|\Hnabla v(x)-\bq(x)\right|^p\,dx&\le C_2\varepsilon^pa^p+C \Big(C_1\e^{p+\delta}+c_0 \e^2(a^p+1)\Big)^{\frac{p}{2}}
				a^{p\left(1-\frac{p}{2}\right)}
				\\&\le C_1\varepsilon^{(p+\delta)\frac{p}{2}} a^{p }+C_2\varepsilon^{p} a^{p },
			\end{split}\end{equation}
			up to renaming the constants.
		Hitherto we have not used the assumption $p>p^\#=p^\#(Q)$, however now we use it \label{lemma-second-alternative-dichotomy-improved:PAGINA}to reabsorb the term $\varepsilon^{(p+\delta)\frac{p}{2}} a^{p }$ into the term $\varepsilon^{p} a^{p }$ appearing in \eqref{first-ineq-p-harm-repl-minus-lin-functBIS-2000}. 
	Indeed, to avoid trivial situations we can suppose that $Q>2$, so we fall the case $p< Q$, and thus, recalling the value of $\delta$ given in \eqref{casepminoren}, we obtain that
		\begin{equation}\label{condizione-su-p}\text{if} \ p>p^\#=\frac{2Q}{Q+2} \quad \text{then}\quad \left(p+\delta\right)\frac p 2 =\left(p+\frac{p^2}{Q-p}\right)\frac p 2 =\frac{Qp^2}{2(Q-p)}>p,
        \end{equation}
		thus if $p\in \left(Q,2\right)$, by \eqref{first-ineq-p-harm-repl-minus-lin-functBIS-2000}, we obtain that, 
				\begin{equation}\label{first-ineq-p-harm-repl-minus-lin-funct-1<p<2}
				\int_{B_{\tau}}\left|\Hnabla v(x)-\bq(x)\right|^p\,dx\le C\varepsilon^{p} a^{p },\end{equation}
			for some constant $C>0$.	
			
			As a consequence of \eqref{first-ineq-p-harm-repl-minus-lin-funct} and \eqref{first-ineq-p-harm-repl-minus-lin-funct-1<p<2} we conclude that, if $p>p^\#$
			\begin{equation}\label{first-ineq-p-harm-repl-minus-lin-functBIS-2}
				\int_{B_{\tau}}\left|\Hnabla v(x)-\bq(x)\right|^p\,dx\le C\varepsilon^{p} a^{p }.
		\end{equation} \medskip

       \noindent{\em Step 4.3: pointwise estimates for the gradient.}	
			Now, we claim that, for all $\e>0$ is sufficiently small,
			for all $x\in B_{\tau/2}$,
			\begin{equation}\label{second-ineq-p-harm-repl-minus-lin-funct}
				\left|\Hnabla v(x)-\bq(x)\right|\le C (\varepsilon a)^{\nu},
			\end{equation}
			for some $C>0$ and $\nu\in(0,1)$. Indeed, suppose by contradiction that
			\begin{equation}\label{assurdo-massimalita-stime-pt-grad}
				\max_{x\in\overline{B_{\tau/2}}}|\Hnabla v(x)-\bq(x)|> C (\varepsilon a)^{\nu}\end{equation}
			for all $C>0$ and $\nu\in(0,1)$. Since $v$ is $(\G,p)$-harmonic in $B_{\tau}$, by Theorem 1.3 in \cite{CM22}, it result that $\Hnabla v\in C^{0,\alpha}_{loc}(B_\tau,H\G)$ (and then $\Hnabla v-\bq \in C^{0,\alpha}_{loc}(B_\tau,H\G)$) for some $\alpha\in (0,1]$ depending on $\G$ and $p$. Thus we can assume that the maximum in \eqref{assurdo-massimalita-stime-pt-grad} is achieved, i.e. it exist $\overline{x}\in\overline{B_{\tau/2}}$
			be such that
			\begin{equation}\label{massimalita-xsgn}
				|\Hnabla v(\overline{x})-\bq(\overline{x})|=\max_{x\in\overline{B_{\tau/2}}}|\Hnabla v(x)-\bq(x)|
				> C (\varepsilon a)^{\nu}.\end{equation}
    Moreover, again for Theorem 1.3 in \cite{CM22} we have that for any $0<r<\tau/4$, and for any $x\in B_{r}(\overline{x})$
    \begin{equation}\label{applic-cm}
        \max_{1\leq i\leq m_1}|X_iv(x)-X_iv(\overline{x})|\leq c \left(\frac{r}{\tau}\right)^\alpha \fint_{B_{\tau/4}(\overline{x})}|\Hnabla v(x)|\,dx \leq c \left(\frac{r}{\tau}\right)^\alpha \|\Hnabla v\|_{L^\infty(B_{\tau/4}(\overline{x}))}\end{equation}
	for some $c>0$ and $\alpha\in(0,1]$, depending only on $\G$ and $p$. Now applying, Theorem 1.1 in \cite{CM22} and \eqref{defin-of-special-competitor}, we have
\begin{equation}\label{applic-cm-2}
    \|\Hnabla v\|_{L^\infty(B_{\tau/4}(\overline{x}))}\le \|\Hnabla v\|_{L^\infty(B_{3\tau/4})}\le  \overline{C}\|\Hnabla v\|_{L^p(B_{\tau})}\le \overline{C} \|\Hnabla u\|_{L^p(B_{\tau})}\le \overline{C} a_1,
\end{equation}
for some $\overline{C}>$ depending only on $\G$ and $p$.
Now we chose $\e>0$ sufficiently small such that $$\left(\frac {C (\e a)^{\nu} }{2 c\,\overline{C}a_1}\right)^{1/\alpha}<\frac{1}{4}.$$ Thus choosing $r=r^\ast:=\tau\,\left(\frac {C (\e a)^{\nu} }{2 c\,\overline{C}a_1} \right)^{1/\alpha}$ in \eqref{applic-cm}, we have that $B_{r^\ast}(\overline{x})\subset B_{\tau/4}(\overline{x})\subset B_\tau$ and combing \eqref{applic-cm} \eqref{applic-cm-2} and \eqref{massimalita-xsgn}, for all $x\in B_{r^\ast}(\overline{x})$ we have
        \begin{eqnarray*}&&
				|X_iv(x)-q_i|\ge |X_iv(\overline{x})-q_i|-|(X_iv(x)-q_i) -(X_iv(\overline{x})-q_i)|
                    \\&&\qquad\qquad \qquad \ge C (\varepsilon a)^{\nu}	-c\,\overline{C}\left(\frac{r^\ast}{\tau}\right)^\alpha a_1
                    \\&&\qquad\qquad \qquad = 
				C (\varepsilon a)^{\nu}- \frac{C (\varepsilon a)^{\nu}}2=\frac{C (\varepsilon a)^{\nu}}2
	\end{eqnarray*}
for all $i=1,\ldots,m_1$. As a consequence, for some constant $\hat{c}$ depending only on $\G$, it result
			\begin{eqnarray*}
				&&\int_{B_{\tau}}\left|\Hnabla v(x)-\bq(x)\right|^p\,dx \ge 
				\int_{ B_{r^\ast}(\overline{x}) }
				\left|\Hnabla v(x)-q\right|^p\,dx \ge \hat{c} \int_{ B_{r^\ast}(\overline{x})} \max_{1\leq i\leq m_1} |X_iv(x)-q_i|\, dx\\
            &&\qquad\qquad\qquad \qquad \qquad \quad \geq\hat{c}\left(\frac{C (\varepsilon a)^{\nu}}2\right)^p \left|B_{r^\ast}(\overline{x})\right|=\hat{c}\left(\frac{C (\varepsilon a)^{\nu}}2\right)^p \tau^Q\left(\frac {C (\e a)^{\nu} }{2 c\,\overline{C}a_1} \right)^{Q/\alpha}|B_1|\\
            &&\qquad\qquad\qquad \qquad \qquad \quad =\hat{c}\,\frac{ C^{p+\frac{Q}{\alpha}}|B_\tau|}{2^{p+\frac{Q}{\alpha}}\, (c\,\overline{C}a_1)^{\frac{Q}{\alpha}}}\, (\e a)^{\nu\left(p+\frac{Q}\alpha\right)}.
			\end{eqnarray*}
			Hence, exploiting \eqref{first-ineq-p-harm-repl-minus-lin-functBIS-2}, we find that
			$$ C\e^p a^{p}\ge \hat{c}\,\frac{C^{p+\frac{Q}{\alpha}} |B_\tau|}{2^{p+\frac{Q}{\alpha}}\, (c\,\overline{C}a_1)^{\frac{Q}{\alpha}} }\,
			(\e a)^{\nu\left( p+\frac{Q}\alpha\right)},$$
			which yields a contradiction as soon as $\e$ is chosen sufficiently small and $\nu\in\left(0,\frac{p}{p+\frac{Q}\alpha}\right).$
						The proof of \eqref{second-ineq-p-harm-repl-minus-lin-funct} is thus complete.\medskip

	\noindent{\em Step 4.4 : the linearized equation and regularity estimates.}	
			Let us define now the function $F:\R^{m_1}\to \R^{m_1}$ as $F(z):=\left|z\right|^{p-2}z$. 
            Let $x\in B_{\tau/4}$, by the usual identification $H\G_x\cong\R^{m_1}$, we have
			\begin{eqnarray}\label{teo-fond}
				&&\left|\Hnabla v(x)\right|^{p-2}\Hnabla v(x)-\left|\bq(x)\right|^{p-2}\bq(x)=
				F(\Hnabla v(x)-F(\Hnabla\langle\bq(x),\pi_x(x)\rangle))\\ \nonumber
                    &&\qquad\qquad\qquad=F(\Hnabla v(x))-F(\bq(x))=\int_0^1\frac{d}{dt}F(t\Hnabla v(x)+(1-t)\bq(x))\,dt\\
				&&\qquad\qquad\qquad=\int_0^1D_zF(t\Hnabla v(x)+(1-t)\bq(x))(\Hnabla v(x)-\bq(x))\,dt, \nonumber
			\end{eqnarray}
            where, for $z\in H\G_x\cong\R^{m_1}$, we set $D_zF=(D_{z_1}F\ldots,D_{z_{m_1}}F)$ the Euclidean gradient of $F$. Now, since $v$ is $(\G,p)$-harmonic in $B_\tau$ taking $\mathrm{div}_\G$ of both sides in \eqref{teo-fond}, we obtain that
			\begin{align}\label{unif-ellip-equat-lemma-dich-improved}
				\mathrm{div}_\G \big(A(x)\Hnabla \big(v(x)-\langle \bq(x),\pi_x (x)\rangle\big)\big)=0\quad \mbox{ in }  B_{\tau/4},
			\end{align}
			with 
			\begin{equation}\label{A-casonostro}
			A(x):=\int_0^1 DF\big(t\Hnabla v(x)+(1-t)\bq(x)\big)\,dt.
			\end{equation}
            Notice that we are in the setting of Lemma \ref{lemma:unifell} with the choice $\eta:=\Hnabla v-\bq$. Indeed, by \eqref{estimate-norm-q} and \eqref{second-ineq-p-harm-repl-minus-lin-funct}, for all $x\in B_{\tau/4}$,
			$$ |\Hnabla v(x)-\bq(x)|\le C (\varepsilon a)^{\nu}\le\frac{a}{16}<\frac{|\bq|}2,
			$$
		as long as $\e$ is sufficiently small. Thus, by exploiting Lemma \ref{lemma:unifell}, for all $x\in B_{\tau/4}$ and $\xi\in H\G_x$ we obtain
			\begin{equation}\label{ellit-A}
			     \lambda\,|q|^{p-2}|\xi|^2\le \langle A(x)\xi,\xi\rangle \le\Lambda\, |q|^{p-2}|\xi|^2,
                \end{equation}
		for some $\Lambda\ge\lambda>0$, depending only on $p$. Recalling \eqref{estimate-norm-q}, \eqref{ellit-A} leads to
		\begin{equation}\label{ellit-A2}
		 \lambda\,a_0^{p-2}|\xi|^2\le\lambda\,a^{p-2}|\xi|^2 \le \langle A(x)\xi,\xi\rangle \le\Lambda\, a^{p-2}|\xi|^2\le \Lambda\, a_1^{p-2}|\xi|^2,
        \end{equation}
			up to renaming $\lambda$ and $\Lambda$, depending on $p$. \medskip
	
	\noindent{\em Step 5: further estimates and conclusion of the proof of Lemma \ref{lemma-second-alternative-dichotomy-improved}.}	
			Now, we point out that we are in the setting of Section \ref{sec:reg}. Indeed, by \eqref{ellit-A2}, we can choose $\nu=\lambda a_0^{p-2}$ and $L=\Lambda a_1^{p-2}$ in \eqref{st.cond1}. Moreover, since by Theorem 1.3 in \cite{CM22}, $\Hnabla v-\bq \in C^{0,\mu}(B_{\tau/8},H\G)$, for some $\mu\in(0,1]$ depending only on $\G$ and $p$, and $A$ as in Lemma \ref{lemma:unifell} is such that $A\in C^{0,\gamma}(B_{\tau/8},H\G)$, for some $\gamma\in(0,1]$ depending only on $p$ (trivially for $p>2$, and by a direct computation in the case $1<p<2$), it follows that $A$ as in \eqref{A-casonostro} is such that $A\in C^{0,\alpha}(B_{\tau/8},H\G)$ for some $\alpha\in(0,1]$ depending only on $\G$ and $p$. This implies straightforwardly that also \eqref{st.cond2} is satisfied. Hence, we are in the position to apply Theorem \ref{thm:c1au}, with $\Omega=B_{\tau/8}$, and $x_0=e$. Let $\bar R>0$ given by Theorem \ref{thm:c1au} with the above choices. Now, up relabeling $\tau$ with $\min\{\tau, \bar R\}$, by \eqref{stima-l-inf-tris}, for every $x\in B_{\tau/8}$, we obtain
			\begin{equation}\label{stima-reg-coeff-var}
				\left|\Hnabla v(x)-\bq(x)\right|^p\le \sup_{B_{\tau/8}}\left|\Hnabla v-\bq(x)\right|^p\le C\fint_{B_{\tau/4}}
				\left|\Hnabla v(y)-\bq\right|^p\,dy\le C\e^p a^{p}, 
			\end{equation}
			for some $C>0$, depending on $Q$ and $p$, where in the last inequality we have also used \eqref{first-ineq-p-harm-repl-minus-lin-functBIS-2}.
			
			Consequently, for all $x\in B_{\tau/8}$,
			\begin{equation}\label{stima_puntuale-gradv-q}
				\left|\Hnabla v(x)-\bq(x)\right|\le 
				C\varepsilon a,
			\end{equation}
			up to renaming $C$, depending on $Q$ and $p$.

        Denoting by $\bar{q}_j:=X_jv(e)-q_j$ for $j=1,\ldots,m_1$, let us define the constant horizontal section $$\bar{\bq}(x):=\sum_{j=1}^{m_1}\bar{q}_jX_j(x), \quad \text{for } x\in \G.$$ By \eqref{stima_puntuale-gradv-q}, for all $x\in B_{\tau/8}$ we have that 
	\begin{equation}\label{stima-bar-q}
		|\bar{\bq}(x)|=|\bar{\bq}(e)|= |\Hnabla v(e)-\mathbf{q}(e)|\le C\varepsilon a.
        \end{equation}
    Hence combining \eqref{stima_puntuale-gradv-q} and \eqref{stima-bar-q}, we deduce that, for all $x\in B_{\tau/8}$,
	\begin{equation*}
		\left|\Hnabla v(x)-\bq(x)-\bar{\bq}(x)\right|\leq \left|\Hnabla v(x)-\bq(x)\right|+\left|\bar{\bq}(x)\right|\le C\varepsilon a,
	\end{equation*}
    up to renaming $C>0$ universal.

	Now, by \eqref{eq:teo-nuovo} of Theorem \ref{thm:c1au}, for all $i=1,\ldots,m_1$ and $x\in B_\rho\subset B_{\tau/8}$, it exists a constant $C>0$ and $\mu\in (0,1]$ depending only on $\G$ and $p$ such that 
    \begin{equation*}
        |X_iv(x)-q_i-\bar{q}_i|=|X_iv(x)-X_iv(e)|\leq C d_c(x,e)^{\mu} \fint_{B_{\tau/8}}|\Hnabla v(x)-\bq(x)|\, dx
    \end{equation*}
    that together Jensen's inequality and \eqref{stima_puntuale-gradv-q} leads to 
    \begin{equation*}
    \begin{split}
        |X_iv(x)-q_i-\bar{q}_i|^p &=|X_iv(x)-X_iv(e)|^p\\&\leq C d_c(x,e)^{p\mu} \fint_{B_{\tau/8}}|\Hnabla v(x)-\bq(x)|^p\, dx\leq C\rho^{p\mu}\e^pa^p.
    \end{split}
    \end{equation*} 
    Thus, to renaming constants, depending only on $\G$ and $p$,  for all $x\in B\rho$, we obtain, 
    \begin{equation*}
        \left|\Hnabla v(x)-\bq(x)-\bar{\bq}(x)\right|^p <C\max_{1\leq i\leq m_1}|X_iv(x)-q_i-\bar q_i|\leq C_2 \rho^{p\mu}\e^pa^p     \end{equation*}
    that implies
    \begin{equation}\label{ineq-average-nabla-p-harm-repl-minus-tilde-q}
    \fint_{B_\rho}\left|\Hnabla v(x)-\bq(x)-\bar{\bq}(x)\right|^p \, dx \leq C_2 \rho^{p\mu}\e^pa^p     \end{equation}		
for some $\mu\in(0,1]$ and $C_2>0$ depending on $\G$ and $p$.
			
Then, if $p\ge 2$, putting together \eqref{forth-ineq-condition-almost-minim} and \eqref{ineq-average-nabla-p-harm-repl-minus-tilde-q}, we obtain
			\begin{equation}\label{ineq-nabla-almost-minim-minus-tilde-q}\begin{split}&
					\fint_{B_{\rho}}\left|\Hnabla u(x)-\bq(x)-\bar{\bq}(x)\right|^p\,dx\\
					&\qquad \ \le\; 2^{p-1}\left(
					\fint_{B_{\rho}}\left|\Hnabla u(x)-\Hnabla v(x)\right|^p\,dx+
					\fint_{B_{\rho}}\left|\Hnabla v(x)-\bq(x)-\bar{\bq}(x)\right|^p\,dx
					\right)\\
					&\qquad \ \le\; 2^{p-1}C_1\varepsilon^{p+\delta}
					\rho^{-Q}+2^{p-1}{C}\sigma(a^p+1)\rho^{-Q}+2^{p-1}C_2\rho^{\mu p}\varepsilon^pa^{p}.\end{split}
			\end{equation}
			Now, setting $\alpha_0:=\mu$ and, for every $\alpha \in (0,\alpha_0)$, we choose
			\begin{equation}\label{jdiet7ub8v9wq43275bvc687c693--ppe5vb}  
				\rho:=\min\left\{ (2^{p+1}C_2)^{\frac{1}{(\alpha- \alpha_0)p}},\frac{\tau}{8}\right\}, \qquad
				\e_0:= \left(\frac{\rho^{\alpha p+Q}a_0^p}{2^{p+1}C_1}\right)^{\frac1{\delta}}
\quad{\mbox{and}}\quad
 c_0:=\frac{\rho^{\alpha p+Q} a_0^p}{2^{p+1}\, {C}(a_1^p+1)},
\end{equation}
we have that,
for every $\e\in(0,\e_0]$ and $\sigma\in(0,c_0\e^p]$, 
			\begin{eqnarray*}
				&&		2^{p-1}C_2\rho^{{\mu} p} \le \frac{1}{4}\rho^{\alpha p},\\
				&&			2^{p-1}C_1\varepsilon^{p+\delta} \rho^{-Q}\le\frac{1}{4}\rho^{\alpha p}\varepsilon^p a^p\\
				{\mbox{and }}&& 2^{p-1} {C}\sigma(a^p+1)\rho^{-Q}\le\frac{1}{4}\rho^{\alpha p}\varepsilon^p a^p.
			\end{eqnarray*}	
			As a consequence of this and \eqref{ineq-nabla-almost-minim-minus-tilde-q},
			\begin{align*}
				&\fint_{B_{\rho}}\left|\Hnabla u(x)-\bq-\bar{\bq}\right|^p\,dx\le
				\frac{1}{4}\rho^{\alpha p}\varepsilon^pa^{p }+\frac{1}{4}\rho^{\alpha p}\varepsilon^pa^p+\frac{1}{4}\rho^{\alpha p}\varepsilon^pa^{p }\le \rho^{\alpha p}\varepsilon^pa^{p },
			\end{align*}
			which gives the desired result in \eqref{tesi-lemma-improv} by setting $\widetilde{q}_i:=q_i+\bar{q}_i$, for all $i=1,\ldots,m_1$ in \eqref{def-q-tilde}
			
				Moreover, from \eqref{stima-bar-q} we have that $$ |\bq-\widetilde{\bq}|=|\bar{\bq}|\le C\e a,$$
			which establishes \eqref{stima-q-q-tilde}.	This completes the proof of Lemma \ref{lemma-second-alternative-dichotomy-improved} when $p\ge2$.
				
		In the case $p^\#<p<2$, we use \eqref{forth-ineq-condition-almost-minimBIS} and \eqref{ineq-average-nabla-p-harm-repl-minus-tilde-q} and we see that
				\begin{eqnarray*}
					&&\hspace{-4em}\fint_{B_\rho} |\Hnabla u(x)-\bq(x)-\bar{\bq}(x)|^p\,dx\\
					&\le&2^{p-1}\left(
					\fint_{B_{\rho}}\left|\Hnabla u(x)-\Hnabla v(x)\right|^p\,dx+
					\fint_{B_\rho} |\Hnabla v(x)-\bq(x)-\bar{\bq}(x)|^p\,dx
					\right)\\&\le& 2^{p-1} C \Big(C_1\e^{p+\delta}+\sigma(a^p+1)\Big)^{\frac{p}{2}}
					a^{p\left(1-\frac{p}{2}\right)}\rho^{-Q}+2^{p-1}C_2\rho^{\mu p}\varepsilon^pa^{p}
					\\&\le& 2^{p-1}C_1\varepsilon^{\frac{(p+\delta)p}2}a^{p\left(1-\frac{p}{2}\right)}
					\rho^{-Q}+2^{p-1}{C}\sigma^{\frac{p}2}(a^p+1)^{\frac{p}2}a^{p\left(1-\frac{p}{2}\right)}\rho^{-Q}+2^{p-1}C_2\rho^{\mu p}\varepsilon^pa^{p} \\
					&\le& 2^{p-1}C_1\varepsilon^{p+\widetilde\delta}a^{p\left(1-\frac{p}{2}\right)}
					\rho^{-Q}+2^{p-1}{C}\sigma^{\frac{p}2}(a^p+1)^{\frac{p}2}a^{p\left(1-\frac{p}{2}\right)}\rho^{-Q}+2^{p-1}C_2\rho^{\mu p}\varepsilon^pa^{p},
				\end{eqnarray*}
				where $\widetilde\delta:=(p+\delta)p/2-p>0$, by \eqref{condizione-su-p} (and up
				to renaming $C$ and $C_1$, depending on $\G$ and $p$).
				
				In this case, setting $\alpha_0:=\mu$ and, for all $\alpha\in(0,\alpha_0)$, we take $\rho$
				as in \eqref{jdiet7ub8v9wq43275bvc687c693--ppe5vb} and   
				\begin{equation*}
				\e_0:= \left(\frac{\rho^{\alpha p+Q}a_0^{\frac{p^2}2}}{2^{p+1}C_1}\right)^{\frac1{\widetilde\delta}}
				\quad{\mbox{and}}\quad
				c_0:=\frac{\rho^{2\alpha +\frac{Q}{p}} a_0^{p}}{4^{\frac{p+1}p}\, {C}^{\frac2{p}}(a_1^p+1)},\end{equation*}		
				obtaining that, for all $\e\in(0,\e_0]$ and $\sigma\in(0,c_0\e^2]$,
				$$ 	\fint_{B_\rho} |\Hnabla u(x)-\bq(x)-\bar{\bq}(x)|^p\,dx\le \rho^{\alpha p}\e^pa^p.$$
Hence, by setting $\widetilde{q}_i:=q_i+\bar{q}_i$ for all $i=1,\ldots,m_1$ in \eqref{def-q-tilde} the desired results in \eqref{tesi-lemma-improv} 
and \eqref{stima-q-q-tilde} follow from this and \eqref{stima-bar-q}.

The proof of Lemma \ref{lemma-second-alternative-dichotomy-improved} is thereby complete.
		\end{proof}
  Iterating Lemma \ref{lemma-second-alternative-dichotomy-improved} we obtain the following estimates:

\begin{cor}\label{corollary-lemma-second-alternative-dichotomy-improved}
	Let $p>p^\#=\frac{2Q}{Q+2}$. Let $u$ be an almost minimizer for $J_p$ in $B_1$ (with constant $\kappa$ and exponent $\beta$) and
	$$ a:=\left(\fint_{B_1}\left|\Hnabla u(x)\right|^p \,dx\right)^{1/p}.$$
	Suppose that it exists $a_1>a_0>0$ such that \begin{equation}\label{A0a1}
		a\in[ a_0,a_1]\end{equation} and that $u$ satisfies \eqref{bound-q} and \eqref{second-alternative-dichotomy}.	
	
	Then there exist $\varepsilon_0$, $\kappa_0$  and $\gamma\in(0,1)$, depending on $Q$, $p$, $\beta$, $a_0$ and $a_1$, such that, for every $\varepsilon\in(0, \varepsilon_0]$ and $\kappa\in(0, \kappa_0\varepsilon^P]$, with $P:=\max\{p,2\}$, then	\begin{equation}\label{first-conclusion-corollary}
		\left\|u-\ell\right\|_{C^{1,\gamma}(B_{1/2})}\le C\varepsilon a.
	\end{equation}
	The positive constant $C$ depends only on $Q$ and $p$ and $\ell$ is a $\G$-affine function of slope $\bq$.
	
	Moreover, 
    \begin{equation}\label{second-conclusion-corollary}
		\left\|\Hnabla u\right\|_{L^{\infty}(B_{1/2})}\le \widetilde{C}a,
	\end{equation}
	with $\widetilde{C}>0$ depending only on $Q$ and $p$.
\end{cor}

\begin{remark}
	We point out that a consequence of \eqref{first-conclusion-corollary}
	in Corollary \ref{corollary-lemma-second-alternative-dichotomy-improved}
	is that, if $\e$ is sufficiently small, 
	\begin{equation}\label{rmk-coroll-hgrad-not-zero}
		\Hnabla u\neq \mathbf{0}\quad {\mbox{in }}B_{1/2},\end{equation}
	where $\mathbf{0}$ is the null horizontal section.
	Indeed, by \eqref{first-conclusion-corollary}, we have that, for all $x\in B_{1/2}$,
	\[\left|\Hnabla u(x)-\bq(x)\right| \le C\varepsilon a,
	\]
	which gives that
	\[
	\left|\Hnabla u(x)\right| \ge |\bq(x)|-\left|\Hnabla u(x)-\bq(x)\right| \ge \left|\bq\right|-C\varepsilon a.
	\]
	Finally, by \eqref{bound-q}, we get
	\[\left|\Hnabla u(x)\right| \ge \frac{a}{4}-C\varepsilon a>0,\]
	as soon as $\varepsilon$ is sufficiently small, which yields \eqref{rmk-coroll-hgrad-not-zero}.
	
	Furthermore, we can conclude that
	\begin{equation}\label{rmk-coroll-assurdo}
		u>0\quad {\mbox{in }} B_{1/2}.\end{equation}
	To check this, we suppose by contradiction that
	there exists a point $x_0\in B_{1/2}$ such that $u(x_0)=0$ (by assumption $u\ge0$). As a consequence, since $u\in C^{1,\gamma}(B_{1/2})$, we see that $\Hnabla u(x_0)=0$, and this contradicts \eqref{rmk-coroll-hgrad-not-zero}. This concludes the proof of \eqref{rmk-coroll-assurdo}.
\end{remark}


We also point out the following scaling property
of almost minimizers:

\begin{lem}\label{remark-rescaling-almost minimizer-1}
	Let $u$ be an almost minimizer for $J_p$ in $B_1$ with constant $\kappa$ and exponent $\beta$. 
	For any $r\in(0,1)$, let
	\begin{equation}\label{defin-rescaling}
		u_r(x):=\frac{u(\delta_r(x))}{r}.
	\end{equation} 
	Then, $u_r$ is			
	an almost minimizer for $J_p$ in $B_{1/r}$ with constant $\kappa r^\beta$ and exponent $\beta$, namely
	\begin{equation}\label{THdefin-rescaling} J_p(u_r,B_\varrho(x_0))\leq(1+\kappa r^\beta \varrho^\beta)J_p(v,B_\varrho(x_0)),
	\end{equation}
	for every ball $B_\varrho(x_0)$ such that $\overline{B_\varrho(x_0)}\subset B_{1/r}$ and for every $v\in HW^{1,p}(B_\varrho(x_0))$ such that $u_r-v \in HW^{1,p}_0(B_\varrho(x_0))$.
\end{lem}

\begin{proof} 
	By definition of almost minimizers, we know that
	\begin{equation}\label{condition-almost minimality}
		J_p(u,B_\vartheta(y_0))\leq(1+\kappa \vartheta^\beta)J_p(w,B_\vartheta(y_0))
	\end{equation}
	for every ball $B_\vartheta(y_0)$ such that $\overline{B_\vartheta(y_0)}\subset B_{1/r}$ and for every $w\in HW^{1,p}(B_\vartheta(y_0))$ such that $u-w \in HW^{1,p}_0(B_\vartheta(y_0))$.
	
	Now, given $x_0\in B_{1/r}$, we take $\varrho$ and $v$ as in the statement of Lemma \ref{remark-rescaling-almost minimizer-1} and we define
	\begin{equation}\label{defin-v_r}
		w(x):= rv\left(\delta_{1/r}(x)\right).
	\end{equation}	
	Then, using the notation $y_0:=\delta_r( x_0)$ and $\vartheta:=r\varrho$,
	for all $x\in\partial B_\vartheta(y_0)$, we have that $\delta_{1/r}(x)\in\partial B_\varrho(x_0)$ and therefore, 
	\begin{equation*}
		w(x)-u(x)=rv\left(\delta_{1/r}(x)\right)-u(x) =ru_r\left(\delta_{1/r}(x)\right)-u(x)=0.
	\end{equation*}
	Accordingly, we can use $w$ as a competitor for $u$ in \eqref{condition-almost minimality},
	thus obtaining that
	\begin{equation}\label{00PPJnd2}
		\int_{B_{r\varrho}(y_0)}\Big(\left|\Hnabla u(y)\right|^p+\chi_{\left\{u>0\right\}}(y)\Big)\,dy
		\leq(1+\kappa r^\beta\varrho^\beta)
		\int_{B_{r\varrho}(y_0)}\Big(\left|\Hnabla w(y)\right|^p+\chi_{\left\{w>0\right\}}(y)\Big)\,dy.
	\end{equation}			
	
	Furthermore, using, consistently with \eqref{defin-rescaling}, the notation $w_r(x):=\frac{w(\delta_r(x))}{r}$, with the change of variable $x:=\delta_{1/r}(y)$ we see that
	\begin{equation}\label{00PPJnd}\begin{split}&
			\int_{B_{r\varrho}(y_0)}\Big(\left|\Hnabla w(y)\right|^p+\chi_{\left\{w>0\right\}}(y)\Big)\,dy= {r^Q}
			\int_{B_{\varrho}(x_0)}\Big(\left|\Hnabla w(\delta_r(x))\right|^p+\chi_{\left\{w>0\right\}}(\delta_r(x))\Big)\,dx\\
			&\qquad \qquad=r^Q\int_{B_{\varrho}(x_0)}\Big(\left|\Hnabla w_r(x)\right|^p+\chi_{\left\{w_r>0\right\}}(x)\Big)\,dx
	\end{split}\end{equation}
	and a similar identity holds true with $u$ and $u_r$ replacing $w$ and $w_r$.
	
	Also, recalling \eqref{defin-v_r},
	we observe that $v=w_r$. Plugging this information and \eqref{00PPJnd}
	into \eqref{00PPJnd2}, we obtain the desired result in \eqref{THdefin-rescaling}.\end{proof}

\begin{proof}[Proof of Corollary \ref{corollary-lemma-second-alternative-dichotomy-improved}]
	Thanks to Lemma \ref{remark-rescaling-almost minimizer-1}, up to scaling, without loss of generality, we can suppose that
	\begin{equation}\label{riscalamento}
		{\mbox{$u$ is an almost minimizer for $J_p$ in $B_2$ (with constant $\kappa$ and exponent $\beta$).}}
	\end{equation}
	Now, we divide the proof of Corollary \ref{corollary-lemma-second-alternative-dichotomy-improved} into separate steps.
	
	\medskip
	\noindent{\em Step 1: iteration of Lemma \ref{lemma-second-alternative-dichotomy-improved}.}
	We prove that we can iterate Lemma \ref{lemma-second-alternative-dichotomy-improved} indefinitely with $\alpha:=\min\left\{\frac{\alpha_0}2,\frac{\beta}{P}\right\}$, with $P:=\max\{p,2\}$ and $\alpha_0$ is the one given in Lemma\ref{lemma-second-alternative-dichotomy-improved}. More precisely, we claim that, for all $ k\ge 0$, there exists a constant horizontal section $\bq_k:\G \to H\G$ \begin{equation}\label{def-qk}
		\bq_k(x):=\sum_{j=1}^{m_1}q_{k,j}X_j(x), \qquad \text{for some }\ q_k=:(q_{k,1},\ldots,q_{k,m_1})\in\R^{m_1}
	\end{equation}
	such that
	\begin{equation}\label{bounds-qk}\begin{split}&
			|\bq_k|\in\left[ \frac{a}4-\frac{\widetilde C\e a\left( 1-\rho^{k\alpha}\right)}{1-\rho^{\alpha}}, \,C_0a+\frac{\widetilde C\e a\left( 1-\rho^{k\alpha}\right)}{1-\rho^{\alpha}}\right]
			,\quad \text{for all }x\in\G\\&
			\left(\fint_{B_{\rho^k}}\left|\Hnabla u(x)-\bq_{k}(x)\right|^p\,dx\right)^{1/p}\le\rho^{{k\alpha}}\varepsilon a\\{\mbox{and }}\quad&
			\left(\fint_{B_{\rho^k}}\left|\Hnabla u(x)\right|^p\,dx\right)^{1/p}\in
			\left[ \frac{|\bq|}2,2|\bq|\right],
		\end{split}
	\end{equation} 
	where $\rho\in(0,1)$, $C_0>0$ and $\widetilde C>0$ are universal contants.
	
	To prove this, we argue by induction. When $k=0$, we pick $\bq_0:=\bq$. Then, in this case, the desired claims in \eqref{bounds-qk} follow from \eqref{bound-q} and \eqref{second-alternative-dichotomy}.
	
	Now we perform the inductive step by assuming that \eqref{bounds-qk} is satisfied for $k$ and we establish the claim for $k+1$. Setting $r:=\rho^k$, and $u_r(\cdot):=\frac{u(\delta_r(\cdot))}{r}$, by inductive assumption we have
	$$ \left(\fint_{B_{1}}\left|\Hnabla u_r(x)-\bq_{k}(x)\right|^p\,dx\right)^{1/p}\le r^{\alpha}\varepsilon a=\varepsilon_k a,$$
	with $\varepsilon_k:=r^{\alpha}\varepsilon=\rho^{k\alpha}\varepsilon$.
	
	Notice that the inductive assumption also yields \eqref{estimate-norm-q}, as soon as $\varepsilon$ is chosen conveniently
	small. Therefore, thanks to Lemma \ref{remark-rescaling-almost minimizer-1} as well,
	we are in position of using Lemma \ref{lemma-second-alternative-dichotomy-improved}
	on the function $u_r$ with $\sigma:=\kappa r^\beta$.
	We stress that the structural condition $\sigma\le c_0\varepsilon^P$ in Lemma \ref{lemma-second-alternative-dichotomy-improved} translates here into $\kappa\le c_0\varepsilon^P$, which is precisely the requirement in the statement of Corollary \ref{corollary-lemma-second-alternative-dichotomy-improved} (by taking $\kappa_0$ there less than or equal to $c_0$).
	In this way, we deduce from \eqref{tesi-lemma-improv} and \eqref{stima-q-q-tilde}
	that there exists a constant horizontal section $\bq_{k+1}$ such that
	\begin{equation}\label{uso-lemma-imp}\begin{split}
			\left(\fint_{B_\rho}\left|\Hnabla u_r(x)-\bq_{k+1}(x)\right|^p\,dx\right)^{1/p}\le\rho^{\alpha}\varepsilon_k a\qquad{\mbox{and}}\qquad
			\left|\bq_k-{\bq}_{k+1}\right|\le \widetilde{C}\varepsilon_k a.	
	\end{split}\end{equation}
	On the other hand, since $\Hnabla\left(u(\delta_r)\right)=r\Hnabla u(\delta_r(x))$, scaling back, we find that
	$$ \left(\fint_{B_{\rho^{k+1}}}\left|\Hnabla u(x)-\bq_{k+1}(x)\right|^p\,dx\right)^{1/p}\le\rho^\alpha\rho^{{k\alpha}}\varepsilon a=
	\rho^{{(k+1)\alpha}}\varepsilon a.$$
	Moreover, using again inductive assumption and \eqref{uso-lemma-imp}, we have
	\begin{eqnarray*}&& |\bq_{k+1}|\le |\bq_k-\bq_{k+1}|+|\bq_k|\le\widetilde{C}\varepsilon_k a+
		C_0a+\frac{\widetilde C\e a\left( 1-\rho^{k\alpha}\right)}{1-\rho^{\alpha}}\\&&\qquad
		=C_0a+\frac{\widetilde C\e a\left( 1-\rho^{k\alpha}\right)}{1-\rho^{\alpha}}+\widetilde{C}\rho^{k\alpha}\varepsilon a=
		C_0a+\frac{\widetilde C\e a\left( 1-\rho^{(k+1)\alpha}\right)}{1-\rho^{\alpha}}
	\end{eqnarray*}
	and
	\begin{eqnarray*}&& |\bq_{k+1}|\ge |\bq_k|-|\bq_k-\bq_{k+1}|\ge
		\frac{a}4-\frac{\widetilde C\e a\left( 1-\rho^{k\alpha}\right)}{1-\rho^{\alpha}}-
		\widetilde{C}\varepsilon_k a\\&&\qquad
		=\frac{a}4-\frac{\widetilde C\e a\left( 1-\rho^{k\alpha}\right)}{1-\rho^{\alpha}}-
		\widetilde{C}\rho^{k\alpha}\varepsilon  a=\frac{a}4-\frac{\widetilde C\e a\left( 1-\rho^{(k+1)\alpha}\right)}{1-\rho^{\alpha}}.\end{eqnarray*}
	In addition,
	\begin{eqnarray*}&&
		\left|\left(\fint_{B_{\rho^{k+1}}}\left|\Hnabla u(x)\right|^p\,dx\right)^{1/p}-|\bq_{k+1}|\right|=
		\frac{1}{|B_{\rho^{k+1}}|^{1/p}}
		\left|\left(\int_{B_{\rho^{k+1}}}\left|\Hnabla u(x)\right|^p\,dx\right)^{1/p}-|\bq_{k+1}||B_{\rho^{k+1}}|^{1/p}\right|\\&&\qquad=
		\frac{1}{|B_{\rho^{k+1}}|^{1/p}}
		\left| \left\|\Hnabla u\right\|_{L^p(B_{\rho^{k+1}})}-\left\|\bq_{k+1}\right\|_{L^p(B_{\rho^{k+1}})}\right|\le
		\frac{1}{|B_{\rho^{k+1}}|^{1/p}}
		\left\|\Hnabla u-\bq_{k+1}\right\|_{L^p(B_{\rho^{k+1}})}\\
		&&\qquad=\left(\fint_{B_{\rho^{k+1}}}\left|\Hnabla u(x)-\bq_{k+1}(x)\right|^p\,dx\right)^{1/p}\le\e a,
	\end{eqnarray*}
	which yields that
	\begin{eqnarray*}&&
		\left|\left(\fint_{B_{\rho^{k+1}}}\left|\Hnabla u(x)\right|^p\,dx\right)^{1/p}-|\bq|\right|=
		\left|\left(\fint_{B_{\rho^{k+1}}}\left|\Hnabla u(x)\right|^p\,dx\right)^{1/p}-|\bq_0|\right|\\&&\qquad\le
		\left|\left(\fint_{B_{\rho^{k+1}}}\left|\Hnabla u(x)\right|^p\,dx\right)^{1/p}-|\bq_{k+1}|\right|+\sum_{j=0}^k\left||\bq_{j+1}|-|\bq_j|\right|\\&&\qquad\le
		\e a+\widetilde{C}a\sum_{j=0}^k\varepsilon_j
		\le \e a+\widetilde{C}\varepsilon a\sum_{j=0}^{+\infty}\rho^{j\alpha}=
		\left(1+\frac{\widetilde{C}}{1-\rho^\alpha}\right)\e a\\&&\qquad\le
		\left(1+\frac{\widetilde{C}}{1-\rho^\alpha}\right)\e |\bq|\le\frac{|\bq|}p.
	\end{eqnarray*}
	These observations conclude the proof of the inductive step and establish \eqref{bounds-qk}.\medskip
	
	
	\noindent{\em Step 2: Morrey-Campanato estimates.}
	We now want to exploit the Morrey-Campanato estimate of Theorem \ref{theo-isomorph-Campanato-spaces-Holder-spaces}
	here applied to the function $\Hnabla u-\bq$,
	with the following choices $$
 \Omega=B_{1/2},\quad \text{and} \quad \lambda=\alpha/Q.$$ In order to do this, we claim that, for every $B\subset B_{1/2}$,
	\begin{equation}\label{primo-claim-campanato}
		\frac{1}{|B|^{1+2\lambda}}\inf_{\boldsymbol{\xi}}\int_{B}\left|\Hnabla u(x)-\bq(x)-\boldsymbol{\xi}(x)\right|^2\,dx\,\le C\e^2 a^2,
	\end{equation}
	up to renaming $C>0$, where the infimum is taken over all the constant horizontal sections $\bxi:\G\to H\G$ as in \eqref{b-xi}.
	
	To prove the claim \eqref{primo-claim-campanato}, we distinguish two cases: either $|B|\ge1$ or $|B|\in(0,1)$.
	
	If $|B|\ge1$, we use \eqref{second-alternative-dichotomy}
	and we get that
	\begin{eqnarray*}&&
		|B|^{-(1+p\lambda)}\inf_{\bxi}\int_{B}\left|\Hnabla u(x)-\bq(x)-\bxi(x)\right|^p\,dx\le
		\int_{B_{1/p}}\left|\Hnabla u(x)-\bq(x)\right|^p\,dx
		\\
		&&\qquad\qquad\qquad\qquad\le|B_1|\fint_{B_{1}}\left|\Hnabla u(x)-\bq(x)\right|^p\,dx
		\le |B_1|\,\varepsilon^p a^p,
	\end{eqnarray*}
	which gives \eqref{primo-claim-campanato} in the case of $|B|\geq 1$.
	
	If instead $|B|\in(0,1)$, let be $k_0=k_0(B)\in\N$ such that $B\subseteq B_{\rho^{k_0}}$ and $ B_{\rho^{k}}\subseteq B$ for all $k>k_0$. By \eqref{bounds-qk} we have that
	\begin{eqnarray*}&&
		|B|^{-(1+p\lambda)}\inf_{\bxi}\int_{B}\left|\Hnabla u(x)-\bq(x)-\bxi(x)\right|^p\,dx\le
		|B_{\rho^{k_0+1}}|^{-(1+p\lambda)}\int_{B_{\rho^{k_0}}}\left|\Hnabla u(x)-\bq_{k_0}(x)\right|^p\,dx\\&&\qquad\qquad\qquad\qquad=
		|B_1|\,\rho^{-Q(k_0+1)(1+p\lambda)}\,|B_{\rho^{k_0}}|
		\fint_{B_{\rho^{k_0}}}\left|\Hnabla u(x)-\bq_{k_0}(x)\right|^p\,dx\\&&\qquad\qquad\qquad\qquad
		\leq C|B_1|\,\rho^{-Q(k_0+1)(1+p\lambda)+Qk_0+pk_0\alpha }\,\varepsilon^p a^p\\&&\qquad\qquad\qquad\qquad= C|B_1|\,\rho^{(-Q-p\alpha)}\,\varepsilon^p a^p,
	\end{eqnarray*}
	which gives \eqref{primo-claim-campanato} up to renaming $C$. The proof of \eqref{primo-claim-campanato}
	is thereby complete.\medskip
	
	\noindent{\em Step 4: conclusion of the proof.}
	Since, by \eqref{second-alternative-dichotomy}
	$$ \left\|\Hnabla u-\bq\right\|_{L^p(B_{1/2})}\le C\varepsilon a,$$
	and by \eqref{equivalent-Campanato-seminorm} and \eqref{primo-claim-campanato}, we have
	$$ [\Hnabla u-\bq]_{\mathcal{E}^{p,\lambda}(B_{1/2},H\G)}\le C\varepsilon a,$$
	up to renaming $C$, we can apply the Theorem \ref{theo-isomorph-Campanato-spaces-Holder-spaces}, from which it follows that
	\begin{equation}\label{stiama-holder-hgradu-q}
		\left[\Hnabla u-\bq\right]_{C^{0,\lambda}(B_{1/2},H\G)}\le
		C\varepsilon a,\end{equation} up to renaming constant $C$ with $\lambda=\alpha/Q\in(0,1)$.
	
	Now we define the $\G$-affine function $\ell:\G\to \R$ as $\ell(x):=u(e)+\left \langle \bq(x),\pi_x(x) \right \rangle$, for $x\in\G$. By Proposition \ref{campi-omogenei0}, we have $\Hnabla\ell=\bq$. For all $x\in B_{1/2}$, let $\delta_x:= d_c(0,x)$ and $\gamma_x\colon[0,\delta_x]\to B_{1/2}$ a sub-unitary curve 
	\[
	\gamma_x(0)=e,\quad \gamma_x (\delta_x)=x\quad \text{ and }\quad \dot \gamma_x(t)=\sum_{i=1}^{m_1}h_i(t)X_i(\gamma_x(t))\quad \text{ for a.e.\ $t\in [0,\delta_x]$},
	\]
	with $\sum_{j=1}^{m_1}(h_j(t))^2\leq 1$, for a.e.\ $t\in [0,\delta_x]$, (such $\gamma$ exists thanks to Chow's Theorem, \cite{BLU}*{Theorem 19.1.3}). By \eqref{stiama-holder-hgradu-q} we conclude
	\begin{align*}
		|u(x)- \ell(x)|&=\left |u(x)-u(e)-\left \langle \bq(x),\pi_x(x)\right\rangle-\left\langle\bq(e),\pi_e(e) \right \rangle\right|\\
		&=\left|\int_0^{\delta_x} \frac{d}{dt}\left(u(\gamma_x(t))-\left \langle \bq(\gamma_x(t)),\pi_{\gamma_x(t)}(\gamma_x(t)) \right \rangle\right)\,dt\right | \\
		&\leq C\int_0^{\delta_x}|\Hnabla u(\gamma_x(t))-\bq(\gamma_x(t))|\,dt \leq C\varepsilon a,
	\end{align*}
	up to renaming $C>0$ a universal constant. By arbitrariness of $x\in B_{1/2}$, we conclude $$\|u-\ell\|_{L^\infty(B_{1/2})}\le C\varepsilon a.$$ This and \eqref{stiama-holder-hgradu-q} establish \eqref{first-conclusion-corollary}.
\end{proof}

\section{Lipschitz continuity of almost minimizers and proof of Theorem \ref{theor-Lipsch-contin-alm-minim}}\label{sec:lip1}

We are now in position of establishing the Lipschitz regularity result in Theorem \ref{theor-Lipsch-contin-alm-minim}. 

\begin{proof}[Proof of Theorem \ref{theor-Lipsch-contin-alm-minim}]
	In the light of Lemma \ref{remark-rescaling-almost minimizer-1}, up to a rescaling,
	we can assume that $u$ is an almost minimizer with constant \begin{equation}\label{1-e2LS}
		\widetilde{\kappa}:=\kappa s^{\beta},\end{equation} which can be made arbitrarily small by an appropriate choice of $s>0$.
	
    Let us set $P:=\max\{p,2\}$. Let also $\alpha_0\in(0,1]$ be the structural constant given by Lemma \ref{lemma-second-alternative-dichotomy-improved} and define
	$$ \alpha:=\frac12\min\left\{\alpha_0,\frac\beta{P}\right\}.$$
	We also consider $\e_0$ as given by Proposition \ref{dic} and take $\eta\in(0,1)$ and $M\ge1$ as in Proposition \ref{dic} (corresponding here to choice $\varepsilon:=\varepsilon_0/2$).
	
	Let us define 
	\begin{equation}\label{def-a-tau}
		a(\tau):=\bigg(\fint_{B_{\tau}}\left|\Hnabla   u(x)\right|^p\,dx\bigg)^{1/p}.
	\end{equation}
	We now divide the argument into separate steps.			\medskip
	
	\noindent{\em Step 1: estimating the average.}
	We claim that, for every $r\in(0,\eta]$,
	\begin{equation}\label{ineq-a(r)-r<1}
		a(r)\le C(M,\eta)(1+a(1)),
	\end{equation}
	for some $C(M,\eta)>0$, possibly depending on $Q$ and $p$ as well.

	To prove this, we consider the set ${\mathcal{K}}\subseteq\N=\{0,1,2,\dots\}$ containing all the integers $k\in\N$ such that
	\begin{equation}\label{ineq-a(eta^k)}
		a(\eta^k)\le C(\eta)M+2^{-k}a(1),
	\end{equation}
	where
	\begin{equation}\label{def-C-eta}
		C(\eta):= 2\eta^{-Q/p} . \end{equation}
	We stress that for $k=0$ formula \eqref{ineq-a(eta^k)} is clearly true, hence
	\begin{equation}\label{k-not-empty}
		0\in{\mathcal{K}}\ne\varnothing.\end{equation}
	We then distinguish two cases, namely whether \eqref{ineq-a(eta^k)} holds for every $k$ (i.e. ${\mathcal{K}}=\N$) or not (i.e. ${\mathcal{K}}\subsetneqq\N$). 
	\medskip
	
	\noindent{\em Step 1.1: the case $\mathcal{K}=\N$}. For every $r\in(0,\eta],$ we define $k_0\in\N={\mathcal{K}}$ such that $\eta^{k_0+1}<r\le \eta^{k_0}.$ Hence, according to \eqref{def-a-tau} and \eqref{ineq-a(eta^k)}, we get that
	\begin{equation*}\begin{split}
			&a(r)\le \bigg(\frac{1}{|B_1|\,\eta^{(k_0+1)Q} }\int_{B_{\eta^{k_0}}}\left|\Hnabla u(x)\right|^p\,dx\bigg)^{1/p}=\eta^{-Q/p}a(\eta^{k_0})\le \eta^{-Q/p}( C(\eta)M+a(1))\\
			&\qquad\qquad\le \eta^{-Q/p}\max\left\{ C(\eta)M,1\right\}(1+a(1))\le C(M,\eta)(1+a(1)),
	\end{split}\end{equation*}
	provided that $C(M,\eta)
	\ge\eta^{-Q/p}\max\left\{ C(\eta)M,1\right\}$.
	The proof of \eqref{ineq-a(r)-r<1} is thereby complete in this case.
	\medskip
	
	\noindent{\em Step 1.2: the case ${\mathcal{K}}\subsetneqq\N$}			
	By \eqref{k-not-empty}, there exists $k_0\in\N$ such that $\{0,\dots,k_0\}\in{\mathcal{K}}$
	and
	\begin{equation}\label{prop-k0+1}
		k_0+1\not\in{\mathcal{K}}.\end{equation} We notice that, by \eqref{def-a-tau}
	and \eqref{def-C-eta},
	$$\eta^{-Q/p}M< C(\eta)M\le C(\eta)M+2^{-(k_0+1)}a(1)<
	a(\eta^{k_0+1})\le\eta^{-Q/p}a(\eta^{k_0})$$
	and therefore
	\begin{equation}\label{DALS:iskcd0}a(\eta^{k_0})> M.
	\end{equation}
	Furthermore, from \eqref{prop-k0+1},
	\begin{equation}\label{DALS:iskcd}
		a(\eta^{k_0+1})>C(\eta)M+2^{-(k_0+1)}a(1)\ge\frac{C(\eta)M+2^{-k_0}a(1)}{2}\ge \frac{a(\eta^{k_0})}2.
	\end{equation}
	
	Now, we claim that
	\begin{equation}\label{second-alternative-dichotomy-B_eta^k}
		\left(\fint_{B_{\eta^{k_0+1}}}\left|\Hnabla u(x)-\bq(x)\right|^p \,dx\right)^{1/p}\le \varepsilon a(\eta^{k_0}),
	\end{equation}
	for some constant horizontal section $\bq:\G\to H\G$ such that $$\frac{a(\eta^{k_0})}{4}<\left|\bq\right|\le C_0a(\eta^{k_0})$$ and being $C_0$ the constant given by Proposition \ref{dic}.
	
	To prove \eqref{second-alternative-dichotomy-B_eta^k}, we apply the dichotomy result in Proposition \ref{dic} rescaled in the ball $B_{\eta^{k_0}}$.
	In this way, we deduce from \eqref{alt1}, \eqref{alt2} and \eqref{bound-q} that \eqref{second-alternative-dichotomy-B_eta^k} holds true, unless it occurs \[ a(\eta^{k_0+1})=\left(\fint_{B_{\eta^{k_0+1}}}\left|\Hnabla u(x)\right|^p \,dx\right)^{1/p}\le \frac{ a(\eta^{k_0}) }2, \]
	but this contradicts \eqref{DALS:iskcd}.
	This ends the proof of \eqref{second-alternative-dichotomy-B_eta^k}.
	
	Now we apply Corollary \ref{corollary-lemma-second-alternative-dichotomy-improved} rescaled in the ball $B_{\eta^{k_0+1}}$: namely, we use here Corollary \ref{corollary-lemma-second-alternative-dichotomy-improved} with $B_1$ replaced by $B_{\eta^{k_0+1}}$ and $a$ replaced by $a(\eta^{k_0+1})$. 
	To this end, we need to verify that the assumptions of Corollary \ref{corollary-lemma-second-alternative-dichotomy-improved} are fulfilled in this rescaled situation. Specifically, we note that, by \eqref{DALS:iskcd0} and \eqref{DALS:iskcd},
	$$ a(\eta^{k_0+1})\ge\frac{M}{2}.$$
	
	Also, since $k_0\in{\mathcal{K}}$,
	$$ a(\eta^{k_0+1})\le\eta^{-Q/p}a(\eta^{k_0})\le \eta^{-Q/p}\big(C(\eta)M+2^{-k_0}a(1)
	\big)\le \eta^{-Q/p}\big(C(\eta)M+a(1)\big).$$
	These observations give that \eqref{A0a1} is satisfied in this rescaled setting with
	\begin{equation}\label{speriamobene}
		a_0:=\frac{M}2\qquad {\mbox{ and }}\qquad a_1:=\eta^{-Q/p}\big(C(\eta)M+a(1)\big).\end{equation}
	
	Moreover, by \eqref{DALS:iskcd} and \eqref{second-alternative-dichotomy-B_eta^k},
	$$\left(\fint_{B_{\eta^{k_0+1}}}\left|\Hnabla u(x)-\bq(x)\right|^p \,dx\right)^{1/p}\le 2\varepsilon a(\eta^{k_0+1}),
	$$ showing that \eqref{second-alternative-dichotomy} is satisfied (here with $2\varepsilon$ instead of $\varepsilon$).
	
	We also deduce from \eqref{DALS:iskcd} and \eqref{second-alternative-dichotomy-B_eta^k} that
	$$			\frac{\eta^{Q/p}a(\eta^{k_0+1})}{4}\le\frac{a(\eta^{k_0})}{4}<\left|\bq\right|\le C_0a(\eta^{k_0})\le2C_0a(\eta^{k_0+1}),
	$$
	which gives that \eqref{bound-q} is satisfied here (though with different structural constants).
	
	We can thereby exploit Corollary \ref{corollary-lemma-second-alternative-dichotomy-improved}
	in a rescaled version
	to obtain the existence of $\e_0$ (possibly different from the one given by Proposition \ref{dic})
	and $\kappa_0$, depending on $Q$, $p$, $\beta$, $a_0$ and $a_1$,
	such that, if
	\begin{equation}\label{condaggforseokroetu}
		\widetilde\kappa\in(0,\kappa_0\e_0^P],\end{equation} 
	we have that
	\begin{equation}\label{swqevu3576435432647328528765uhgfhdf}
		\left\|\Hnabla u\right\|_{L^{\infty} (B_{\eta^{k_0+1}/2} )}\le \bar{C}a(\eta^{k_0}),
	\end{equation}
	for some structural constant $\bar{C}>0$.
	
	We remark that condition \eqref{condaggforseokroetu} is satisfied by taking $s$ in \eqref{1-e2LS}
	sufficiently small, namely taking $s:= \left(\frac{\kappa_0 \e_0^P}{2\kappa}\right)^{\frac1{\beta}}$.
	Notice in particular that
	\begin{equation}\label{noticeinpartlsworeteuyti8787686}
		{\mbox{$s$ depends on $Q$, $p$, $\kappa$,
				$\beta$ and $\|\Hnabla u\|_{L^p(B_1)}$,}}\end{equation} due to \eqref{speriamobene}.	
	
	As a result of \eqref{swqevu3576435432647328528765uhgfhdf}, for all $r\in\left(0,\frac{\eta^{k_0+1}}2\right)$,
	\begin{equation}\label{ineq-a(r)-r-le-eta^k_0/2}\begin{split}&
			a(r)=\left( \frac1{|B_r|}\int_{B_r}|\Hnabla u(x)|^p\,dx\right)^{1/p}\le \bar{C}a(\eta^{k_0})\le
			\bar{C}\big( C(\eta)M+2^{-k_0}a(1)\big)\\&\qquad\qquad\le\bar{C}\big( C(\eta)M+a(1)\big)\le C(M,\eta)(1+a(1)),
	\end{split}\end{equation}
	as long as $C(M,\eta)\ge\bar{C}(C(\eta)M+1)$.
	
	Moreover, if $r\in\left[\frac{\eta^{k_0+1}}2,\eta\right]$ then we take $k_r\in\N$ such that $\eta^{k_r+1}<r\le\eta^{k_r}$.
	Note that
	$$\frac1{\eta^{k_r}}\le\frac1r\le\frac{2}{\eta^{k_0+1}},$$
	whence
	\begin{equation}\label{L23tSx32}
		k_r \le k_0+C_\star,
	\end{equation}
	where $C_\star:=1+\frac{\log2}{\log(1/\eta)}$.
	
	We now distinguish two cases: if $k_r\in\{0,\dots,k_0\}$ then $k_r\in{\mathcal{K}}$ and therefore
	\begin{equation*}
		a(\eta^{k_r})\le C(\eta)M+2^{-k_r}a(1).
	\end{equation*}
	From this, we obtain that
	\begin{equation}\label{kpd012}
		\begin{split}&
			a(r)\le\left( \frac1{|B_{\eta^{k_r+1}}|}\int_{B_{\eta^{k_r}}}|\Hnabla u(x)|^p\,dx\right)^{1/p}
			=\eta^{-Q/p}a(\eta^{k_r})\\&\qquad\qquad\le\eta^{-Q/p}\big( C(\eta)M+2^{-k_r}a(1)\big)\le C(M,\eta)(1+a(1)).\end{split}
	\end{equation}
	If instead $k_r>k_0$, we employ \eqref{L23tSx32} to see that
	\begin{eqnarray*}&&
		a(r)\le\left( \frac1{|B_{\eta^{k_0+C_\star+1}}|}\int_{B_{\eta^{k_0}}}|\Hnabla u(x)|^p\,dx\right)^{1/p}
		=\eta^{-Q(C_\star+1)/p} a(\eta^{k_0})\\ &&\qquad\qquad\quad\le
		\eta^{-Q(C_\star+1)/p}\big( C(\eta)M+2^{-k_0}a(1)\big)\le\eta^{-Q(C_\star+1)/p}\big( C(\eta)M+a(1)\big)\\&&\quad\qquad\qquad\le
		C(M,\eta)(1+a(1)),
	\end{eqnarray*}
	as long as $C(M,\eta)$ is chosen large enough.
	
	This and \eqref{kpd012} give that for all $r\in\left[\frac{\eta^{k_0+1}}2,\eta\right]$,
	we have that $a(r)\le C(M,\eta)(1+a(1))$. Combining this with \eqref{ineq-a(r)-r-le-eta^k_0/2},
	we deduce that \eqref{ineq-a(r)-r<1} holds true, as desired.\medskip
	
	\noindent{\em Step 2: conclusions.}
	Now, up to scaling and translations, we can extend \eqref{ineq-a(r)-r<1} to all balls with centre $x_0$ in $B_{1/2}$
	and sufficiently small radius. Namely, we have that, for all $r\in(0,\eta]$,
	\begin{equation*}
		a(r,x_0)\le C(M,\eta)(1+a(1)),
	\end{equation*}
	where
	$$ 				a(r,x_0):= \left(\fint_{B_{r}(x_0)}\left|\Hnabla u(x)\right|^p\,dx\right)^{1/p}.$$
	Therefore, according to the Lebesgue Differentiation Theorem (see e.g. \cite{T04}), recalling that $u\in HW^{1,p}(B_1),$ we find that, for almost every $x_0\in B_{1/2}$,
	\begin{align*}
		&\left|\Hnabla u(x_0)\right|=\lim_{r\rightarrow 0}a(r,x_0)\le C(M,\eta)(1+a(1))=C\Big(1+\left\|\Hnabla u\right\|_{L^p(B_1)}\Big),
	\end{align*}
	for some $C>0$,
	which yields
	\begin{equation}\label{first-claim-theorem-Lipsch-contin}
		\left\|\Hnabla u\right\|_{L^{\infty}(B_{1/2})}\le C\Big(1+\left\| \Hnabla u\right\|_{L^p(B_1)}\Big).
	\end{equation}
	So, the first claim in Theorem \ref{theor-Lipsch-contin-alm-minim} is proved.
	
	We now show that the second claim in Theorem \ref{theor-Lipsch-contin-alm-minim}
	holds. For this, we can assume that
	\begin{equation}\label{PROCSH}\{u=0\}\cap B_{s/100}\ne\varnothing\end{equation}
	and we take a Lebesgue point $\bar{x}\in B_{s/100}$ for $\Hnabla u$.
	Up to a left-translation, we can suppose that $\bar{x}=0$ and change our assumption \eqref{PROCSH} into
	\begin{equation}\label{PROCSH2}\{u=0\}\cap B_{s/50}\ne\varnothing.\end{equation}
	We claim that, in this situation,
	\begin{equation}\label{alt-c-PK}
		{\mathcal{K}}=\N.
	\end{equation}
	Indeed, suppose not and let $k_0$ as above (recall \eqref{prop-k0+1}).
	Then, in light of \eqref{second-alternative-dichotomy-B_eta^k},
	we can apply Corollary \ref{corollary-lemma-second-alternative-dichotomy-improved} (rescaled as before)
	and conclude, by \eqref{rmk-coroll-assurdo}, that $u>0$
	in $B_{s/2}$, in contradiction with our hypothesis in \eqref{PROCSH2}.
	This establishes \eqref{alt-c-PK}.
	
	Therefore, in view of \eqref{ineq-a(eta^k)}  and \eqref{alt-c-PK}, for all $k\in\N$,
	$$	a(\eta^k)\le C(\eta)M+2^{-k}a(1),$$
	and as a result
	$$ |\Hnabla u(\bar{x})|=|\Hnabla u(0)|=\lim_{k\to+\infty}a(\eta^k)\le
	\lim_{k\to+\infty}\big(C(\eta)M+2^{-k}a(1)\big)=C(\eta)M.$$
	Recalling also \eqref{noticeinpartlsworeteuyti8787686},
	the proof of the second claim in Theorem \ref{theor-Lipsch-contin-alm-minim} is thereby complete.
\end{proof}

\bibliographystyle{alpha}
 	
\end{document}